\documentclass[10pt]{article}
\usepackage{amsfonts,amssymb,float,ulem,textcomp}
\RequirePackage{ifpdf}
\usepackage{amsmath}
\usepackage{color}
\usepackage{pstricks}
\usepackage{pst-poly}
\usepackage{graphics}
\usepackage{epic}
\usepackage{curves}
\usepackage{tikz-cd}
\usepackage{amsthm}

 \usepackage{pstricks-add}
 \usepackage{epsfig}
 \usepackage{pst-grad} 
 \usepackage{pst-plot} 
 \usepackage[space]{grffile} 
 \usepackage{etoolbox} 
 \makeatletter 
 \patchcmd\Gread@eps{\@inputcheck#1 }{\@inputcheck"#1"\relax}{}{}
 \makeatother

\def\appendix#1{
\addtocounter{section}{1} \setcounter{equation}{0}
\renewcommand{\thesection}{\Alph{section}}
\section*{Appendix \thesection\protect\indent\quad
#1}
}
\renewcommand{\theequation}{\thesection.\arabic{equation}}

\catcode`\@=11
\def\marginnote#1{}

\newcount\hour
\newcount\minute
\newtoks\amorpm
\hour=\time\divide\hour by60 \minute=\time{\multiply\hour by60
\global\advance\minute by-\hour}
\edef\standardtime{{\ifnum\hour<12 \global\amorpm={am}%
        \else\global\amorpm={pm}\advance\hour by-12 \fi
        \ifnum\hour=0 \hour=12 \fi
        \number\hour:\ifnum\minute<10 0\fi\number\minute\the\amorpm}}
\edef\militarytime{\number\hour:\ifnum\minute<100\fi\number\minute}

\newcommand{\tcr}{\textcolor{red}}

\newcommand{\tcw}{\textcolor{white}}

\def\draftlabel#1{{\@bsphack\if@filesw {\let\thepage\relax
      \xdef\@gtempa{\write\@auxout{\string
          \newlabel{#1}{{\@currentlabel}{\thepage}}}}}\@gtempa \if@nobreak
    \ifvmode\nobreak\fi\fi\fi\@esphack} \gdef\@eqnlabel{#1}}
    \def\@eqnlabel{}
\def\@vacuum{}
\def\draftmarginnote#1{\marginpar{\raggedright\scriptsize\tt#1}}

\def\draft{
  \oddsidemargin -.5truein
  \def\@oddfoot{\footnotesize \sl preliminary draft \hfil
    \rm\thepage\hfil\sl\today\quad\militarytime}
  \let\@evenfoot\@oddfoot \overfullrule 3pt
    \let\label=\draftlabel
    \let\marginnote=\draftmarginnote
  \def\@eqnnum{(\theequation)\rlap{\kern\marginparsep\tt\@eqnlabel}%
    \global\let\@eqnlabel\@vacuum}

  }

\voffset= - 1in \hoffset= - 1.0in

\newcommand{\tr}{\,{\rm Tr}\,}

\def\be{\begin{equation}}
\def\ee{\end{equation}}
\def\bea{\begin{align}}
\def\eea{\end{align}}
\def\<{\langle}
\def\>{\rangle}

\def\A{\mathcal{A}}
\def\B{\mathcal{B}}
\def\BB{\mathfrak{B}}
\def\PP{\mathcal{P}}

\def\C{\mathcal{C}}
\def\E{\mathcal{E}}

\def\T{\mathcal{T}}

\def\M{{\mathcal M}}

\def\tr{\mathop{\rm{tr}}}

\def\ocomma{{\phantom{\Bigm|}^{\phantom {X}}_{\raise-1.5pt\hbox{,}}\!\!\!\!\!\!\otimes}}

\renewcommand\emph[1]{{\it #1}}

\newcommand{\col}[1]{{\raise-2pt\hbox{\tiny$\bullet$}\hskip -4.5pt \raise4pt\hbox{\tiny$\bullet$}{{#1}} \raise-2pt\hbox{\tiny$\bullet$}\hskip -4.5pt \raise4pt\hbox{\tiny$\bullet$}}}

\newcommand{\sheet}[2]{{\stackrel{{#1}}{{#2}}}}

\newtheorem{theorem}{Theorem}[section]
\newtheorem{lemma}[theorem]{Lemma}

\theoremstyle{definition}
\newtheorem{example}[theorem]{Example}
\newtheorem{remark}[theorem]{Remark}
\newtheorem{conjecture}[theorem]{Conjecture}

\allowdisplaybreaks

\begin{document}

\title{Symplectic groupoid and cluster algebras}








\author{Leonid O. Chekhov\thanks{Steklov Mathematical Institute, Moscow, Russia, National Research University Higher School of Economics, Russia, and Michigan State University, East Lansing, USA. Email: chekhov@msu.edu.} 
\ and Michael Shapiro\thanks{Department of Mathematics, Michigan State University, East Lansing, MI 48823, USA. Email: mshapiro@msu.edu.}
}

\maketitle

\begin{abstract}
We consider the symplectic groupoid of pairs $(B,\mathbb{A})$ with $\mathbb A$ unipotent upper-triangular matrices and $B\in GL_n$ being such that $\widetilde {\mathbb A}=B{\mathbb A} B^{\text{T}}$ are also unipotent upper-triangular matrices. We explicitly solve this groupoid condition using Fock--Goncharov--Shen cluster variables and show that for $B$ satisfying the standard semiclassical Lie--Poisson algebra, the matrices $B$, $\mathbb A$, and $\widetilde{\mathbb A}$ satisfy the closed Poisson algebra relations expressible in the $r$-matrix form. Identifying entries of $\mathbb A$ and $\widetilde {\mathbb A}$ with geodesic functions for geodesics on the two halves of a closed Riemann surface of genus $g=n-1$ separated by the Markov element, we are able to construct the geodesic function $G_B$ ``dual'' to the Markov element. We thus obtain the complete cluster algebra description of Teichm\"uller space $\mathcal T_{2,0}$ of genus two. We discuss also the generalization of our construction for higher genera. For genus larger than three we need a Hamiltonian reduction based on the rank condition $\hbox{rank\,}({\mathbb A}+{\mathbb A}^{\text{T}})\le 4$; we present the example of such a reduction for $\mathcal T_{4,0}$.

{\it Dedicated to the memory of great mathematician and person Igor Krichever}
\end{abstract}

\section{Introduction}

\begin{itemize}
\item
Denote by $\A_n$ the space of unipotent upper-triangular $n\times n$ real matrices. The symplectic groupoid of upper-triangular matrices 
$\M$
is formed by pairs $(B,{\mathbb A})$ such that $B\in GL_n(\mathbb{R}), {\mathbb A}\in \A_n$ satisfying $B{\mathbb A}B^\text{T}\in \A_n$. Here, $\mathbb A$ is considered as an object and $B$ as a morphism taking $\mathbb A$ to $B\mathbb AB^\text{T}$. Natural source and target maps $s$ and $t$ from $\M$ to $\A_n$ are defined $s((B,\mathbb A))=\mathbb A$, $t((B,\mathbb A))=B\mathbb AB^\text{T}$. 
Let $\M^{(2)}\subset \M\times \M$ be the subset of all compatible pairs $\left\{\left((C,B\mathbb AB^\text{T}),(B,\mathbb A)\right)\right\}$ and $p_1,p_2$ be two natural projections $\M^{(2)}\to \M$ to the first and to the second component, $m:\M^{(2)}\to\M$ be the multiplication map 
$m\left({\left( (C,B\mathbb AB^\text{T} ), (B,\mathbb A) \right)}\right)=(CB,\mathbb A)$.
Symplectic groupoid $\M$ is equipped with a natural symplectic form $\omega$ satisfying $m^*(\omega)=p_1^*(\omega)+p_2^*(\omega)$$\cite{Weinstein}$. Push forwards by $(p_1)_*$  
the nondegenerate Poisson structure $\PP_s$ dual to $\omega$
makes $\A$ into a Poisson manifold equipped with the \emph{reflection} Poisson bracket $\{,\}$ satisfying relation
$(p_1)_*(\PP_s)=-(p_2)_*(\PP_s)$ (see (\ref{A-A}) and (\ref{At-At})).


Recall that the group $SL_n$ is a Poisson variety equipped with the standard trigonometric Poisson-Lie bracket $\{,\}_{SL_n}$. Consider subset $\B_n\subset SL_n$ of matrices $B$ such that there is a pair $(B,\mathbb A)\in\M$. We show that $\B_n$ is a symplectic leaf of $\{,\}_{SL_n}$.
Moreover, we show that the choice of $B\in\B_n$  determines unique $\mathbb A\in\A_n$ such that $(B,\mathbb A)\in\M$. This construction determines a map $\psi$ from $\B_n$ to $\A_n$. It was shown in \cite{ChM4} that the map $\psi:(\B_n,\{,\}_{SL_n})\to(\A_n,\{,\})$ is Poisson.
In this paper we provide another, simpler proof of this claim using the network realization of matrix $B$ and matrices $\mathbb A$ and $B\mathbb AB^\text{T}$ obtained by studying the moduli space $\PP_{SL_n,\square}$ of $SL_n$-pinnings on the planar square introduced in~\cite{GS19}.

The Poisson variety $SL_n,\{,\}_{SL_n}$ has a compatible cluster structure, i.e., the collection of cluster coordinate charts $\C(t)=\{z_i(t)\},t\in\E$ where $\E$ is the exchange graph of the cluster algebra, such that $\{z_i(t),z_j(t)\}=\epsilon_{ij} z_i(t) z_j(t)$, $\epsilon_{ij}$ form a skew-symmetric matrix with half-integer entries. We extend the map $\psi$ to the Poisson map  $\widetilde{\psi}:\B_n\to\A_n\times\A_n$ which maps $B\mapsto ({\mathbb A},B\mathbb AB^\text{T})$. The image ${\mathcal I}=\widetilde{\psi}(\B_n)$ is of codimension $\lfloor{\frac{n}{2}}\rfloor$ subvariety of $\A_n\times\A_n$ defined by the conditions that the values of some Casimirs of $\A_n$ coincide for $\mathbb A$ and $B\mathbb AB^\text{T}$. One can think of $\B_n$ as $\lfloor{\frac{n}{2}}\rfloor$-parameter extension of  $\mathcal I$. The symplectic leaf $\B_n$ possesses a cluster structure obtained by restricting and amalgamating the cluster structure from the moduli space of pinnings.

\item

There is a well known braid group action on $\A_n$. To describe this action we identify the elements of the space $\A_n$ of the objects of symplectic groupoid $\M$ with the space of triangular unipotent bilinear forms. The morphisms are changes of basis that preserve  the triangular unipotent shape of the form. For a fixed form $\Phi$ let's call the basis compatible if the form has the triangular unipotent shape.
The braid group $\mathfrak{B}_n$ acts on the set of compatible bases. The $i$th generator $\beta_i$ maps the basis $(v_1,\dots,v_n)$ to the basis $(v_1,\dots,v_{i+1}, v_i-\Phi(v_i,v_{i+1}) v_{i+1},v_{i+2},\dots,v_n)$.

In the matrix form $\beta_i$ is realized as the action $\beta_i(\mathbb A)=B_i(\mathbb A)\mathbb AB_i(\mathbb A)^\text{T}$ where the matrix $B_i(\mathbb A)$ has the block form
$B_i(\mathbb A)=\begin{pmatrix}
\rm{Id}_{i-1} & 0 & 0 \\ 
0 & \boxed{\begin{matrix} 
a_{i,i+1} & 1 \\
-1 & 0
\end{matrix}} & 0 \\
0 & 0 & \rm{Id}_{n-i-1}
\end{pmatrix}
$ with $\rm{Id}_k$ denoting the $k\times k$ identity matrix and $a_{i,j}$ denoting entries of the matrix $\mathbb A$.

We show below that the braid group action is cluster, i.e., can be obtained as a sequence of cluster mutations.

\item
The algebra $\mathbb{C}(\T_{g,s})$ of functions on the Teichm\"uller space $\T_{g,s}$ of curves $\Sigma_{g,s}$ of genus $g$ with $s$ holes is generated by the geodesic functions $\{G_\ell\}$, $\ell$ is a simple loop on a topological surface  $\Sigma_{g,s}$. Functions $G_\ell$ satisfy skein relations based on two ``resolutions'' of the crossing between two geodesics. All algebraic relations hold for the corresponding geodesic functions (we assume that the pattern outside the circle containing the crossing remains intact): 
$$
\begin{pspicture}(-4,-1)(4,1){\psset{unit=0.7}
\rput(-3.5,0){
\psclip{\pscircle[linewidth=1.5pt, linestyle=dashed](0,0){1}}
\rput(0,0){\psline[linewidth=1.5pt,linecolor=red, linestyle=dashed](1,-1)(-1,1)}
\rput(0,0){\psline[linewidth=3pt,linecolor=white](-1,-1)(1,1)}
\rput(0,0){\psline[linewidth=1.5pt,linecolor=blue, linestyle=dashed](-1,-1)(1,1)}
\endpsclip
\rput(0,-1.2){\makebox(0,0)[ct]{$G_1G_2$}}
\rput(0.8,0.8){\makebox(0,0)[lb]{$\gamma_1$}}
\rput(0.8,-0.8){\makebox(0,0)[lt]{$\gamma_2$}}
}
\rput(0,0){
\psclip{\pscircle[linewidth=1.5pt, linestyle=dashed](0,0){1}}
\rput(0,-1.4){\psarc[linewidth=1.5pt,linecolor=green, linestyle=dashed](0,0){1}{45}{135}}
\rput(0,1.4){\psarc[linewidth=1.5pt,linecolor=green, linestyle=dashed](0,0){1}{225}{315}}
\endpsclip
\rput(0,-1.2){\makebox(0,0)[ct]{$G_I$}}
\rput(-1.2,0){\makebox(0,0)[rc]{$1\cdot$}}
}
\rput(3.3,0){
\psclip{\pscircle[linewidth=1.5pt, linestyle=dashed](0,0){1}}
\rput(-1.4,0){\psarc[linewidth=1.5pt,linecolor=green, linestyle=dashed](0,0){1}{-45}{45}}
\rput(1.4,0){\psarc[linewidth=1.5pt,linecolor=green, linestyle=dashed](0,0){1}{135}{225}}
\endpsclip
\rput(0,-1.2){\makebox(0,0)[ct]{$G_H$}}
\rput(-1.2,0){\makebox(0,0)[rc]{$1\cdot$}}
}
\rput(-1.7,0){
\rput(0,0){\makebox(0,0){$=$}}}
\rput(1.3,0){
\rput(0.2,0){\makebox(0,0){$+$}}}
}
\end{pspicture}
$$
The algebra $\mathbb{C}(\T_{g,s})$ is Poisson with Goldman Poisson bracket, which, for the same choice of geodesics and the corresponding geodesic functions, has the form
$$
\begin{pspicture}(-4,-1)(4,1){\psset{unit=0.7}
\rput(-3.5,0){
\psclip{\pscircle[linewidth=1.5pt, linestyle=dashed](0,0){1}}
\rput(0,0){\psline[linewidth=1.5pt,linecolor=red, linestyle=dashed](1,-1)(-1,1)}
\rput(0,0){\psline[linewidth=3pt,linecolor=white](-1,-1)(1,1)}
\rput(0,0){\psline[linewidth=1.5pt,linecolor=blue,  linestyle=dashed](-1,-1)(1,1)}
\endpsclip
\rput(0,-1.2){\makebox(0,0)[ct]{$\{G_1,G_2\}$}}
\rput(0.8,0.8){\makebox(0,0)[lb]{$\gamma_1$}}
\rput(0.8,-0.8){\makebox(0,0)[lt]{$\gamma_2$}}
}
\rput(0,0){
\psclip{\pscircle[linewidth=1.5pt, linestyle=dashed](0,0){1}}
\rput(0,-1.4){\psarc[linewidth=1.5pt,linecolor=green, linestyle=dashed](0,0){1}{45}{135}}
\rput(0,1.4){\psarc[linewidth=1.5pt,linecolor=green, linestyle=dashed](0,0){1}{225}{315}}
\endpsclip
\rput(0,-1.2){\makebox(0,0)[ct]{$G_I$}}
\rput(-1.2,0){\makebox(0,0)[rc]{$\dfrac12$}}
}
\rput(3.3,0){
\psclip{\pscircle[linewidth=1.5pt, linestyle=dashed](0,0){1}}
\rput(-1.4,0){\psarc[linewidth=1.5pt,linecolor=green, linestyle=dashed](0,0){1}{-45}{45}}
\rput(1.4,0){\psarc[linewidth=1.5pt,linecolor=green, linestyle=dashed](0,0){1}{135}{225}}
\endpsclip
\rput(0,-1.2){\makebox(0,0)[ct]{$G_H$}}
\rput(-1.2,0){\makebox(0,0)[rc]{$\dfrac12$}}
}
\rput(-1.7,0){
\rput(0,0){\makebox(0,0){$=$}}}
\rput(1.3,0){
\rput(0.2,0){\makebox(0,0){$-$}}}
}
\end{pspicture}
$$
\item
In the case $s\ge 1$ an ideal triangulation of surface $\Sigma_{g,s}$ with the set of edges $E$ gives parametrization $\T_{g,s}$ by so-called shear coordinates $y_e,\ e\in E$ introduced by W.Thurston. The shear coordinates have very nice Poisson properties; they are log-canonical coordinates generaing the Goldman bracket. Namely, the Goldman bracket satisfy relations $\{y_e,y_{e'}\}=\epsilon_{e,e'}$ where $\epsilon_{e,e'}=\sum_{\Delta}\epsilon_{e,e'}(\Delta)$ taken over all triangles $\Delta$ of the triangulation and $\epsilon_{e,e'}(\Delta)=1/2$ if $e'$ follows $e$ in counterclockwise direction inside $\Delta$, $\epsilon_{e,e'}(\Delta)=-1/2$ if $e$ follows $e'$ in counterclockwise direction inside $\Delta$,
and $\epsilon_{e,e'}(\Delta)=0$ otherwise. Dually, one can replace triangulation with the dual 3-valent ribbon graph with the same parameter $y_e$ assigned to the edge of the ribbon graph dual to the edge $e$ of the triangulation. We will denote the dual edge by the same letter $e$. 

In \cite{Fock93} an expression for geodesic functions in terms of shear coordinates was found. In \cite{ChF3}, the authors considered the special collection of loops on genus $\lfloor{\frac{n}{2}}\rfloor$ surface with $s=2-(n\ \hbox{mod}\ 2)$ holes (where the parity  $n\ \hbox{mod}\ 2=0$ for even $n$ and $n\ \hbox{mod}\ 2=1$ for odd $n$). That collection gives rise to the corresponding set of geodesic functions $G_{i,j}, i<j$. The map that takes the surface to the matrix 
$\begin{pmatrix}
1 &  & G_{i,j} \\
 & \ddots & \\
0 & & 1
\end{pmatrix}$ is a Poisson embedding $\T_{g,s}\to \A_n$.

Recall that a morphism $B\in \B$ determines uniquely the pair: the source object ${\mathbb A}\in\A_n$ and the target object $\tilde{\mathbb A}=B{\mathbb A}B^\text{T}\in\A_n$. The map $B\mapsto ({\mathbb A},\tilde{\mathbb A})$ defines a fiber bundle with fibers of positive dimension.

We consider in details the case $n=3$. In this case the projection $B\mapsto({\mathbb A},\tilde{\mathbb A})$ has a one-dimensional fiber. By~\cite{ChF3}, each of the matrices ${\mathbb A}$ and $\tilde{\mathbb A}$ encodes a hyperbolic torus with one hole. Moreover, the geodesic length $\ell_{\mathfrak M}$ of the hole is determined by the so-called Markov element 
${\mathfrak M}=G_{1,2}G_{2,3}G_{1,3}-G_{1,2}^2-G_{2,3}^2-G_{1,3}^2=e^{\ell_{\mathfrak M}/2}+e^{-\ell_{\mathfrak M}/2}-2$ which coincide for $\mathbb A$ and $\widetilde{\mathbb  A}$. Since the lengths of the geodesic boundaries of both tori coincide one can glue two hyperbolic tori along the boundary to obtain a hyperbolic genus two surface with no holes.
The extra parameter in $B$ determines the gluing. Therefore the space $\B_3$ equipped with cluster structure describes the Teichm\"uller space of closed genus two curves $\T_{2,0}$ and its cluster structure is compatible with the Goldman Poisson bracket on $\T_{2,0}$. To the best of our knowledge such construction is new.

The previous approaches to construct cluster coordinates for surfaces with no boundaries were documented in \cite{BW} however,  their construction includes taking a Hamiltonian reduction by some expression which is hard to compute and whose cluster meaning needs to be clarified.  
We do not need such a Hamiltonian reduction in our approach for $g=2$ and $g=3$, but we need a Hamiltonian reduction of a different sort for higher genera. 

\item
We omitted this construction's quantum version to make the text more readable. The quantum algebras will be described in a separate publication.

\item
The main results of this paper include 
\begin{itemize}
\item solution of groupoid condition in terms of a planar network;
\item construction of cluster coordinates on the space of (an extension of) pairs (source,target) objects of symplectic groupoid of triangular unipotent forms compatible with the natural Poisson bracket;
\item cluster representation of braid group action
\item construction of cluster coordinates on Teichm\"uller space of closed (with no boundary) genus two hyperbolic Riemann surfaces. Some elements of this construction are built for higher genus Riemann surfaces. 
\end{itemize} 

\end{itemize}

\section{Semi-classical symplectic groupoid}

\subsection{Solving the groupoid compatibility condition}


Recall that the moduli space of pinnings $\PP_{SL_n,\Sigma}$~\cite{GS19} on a hyperbolic Riemann surface $\Sigma$ with marked points is described by the collection of so-called Fock-Goncharov-Shen (FGS-) parameters. 
FGS-parameters are defined by an ideal triangulation $T$ of $\Sigma$ with vertices of triangles at the marked points and geodesic sides of triangles. It is convenient to subdivide each triangle into the lattice of smaller triangles (each side of $T$ is subdivided into $n$ subintervals) and assign one FGS-parameter to each vertex of this lattice (see Figure~\ref{fi:triSL3}). The pinning described by the collection of the FGS-parameter assigns a basis in $\mathbb{R}^n$ to every side of $T$ and, therefore, determines the \emph{transport matrix} between two sides of every triangle of $T$.  All transport matrices can be realized as "boundary measurement" matrices (or response matrices)~\cite{Po, ChSh20} of the appropriately oriented networks with face weights given by the FGS-parameters. The Goldman Poisson bracket is described by the quiver $Q$ with vertices equipped with FGS-parameters $Z_{abc}$, more precisely, $\{Z_{abc},Z_{pqr}\}=\varepsilon_{abc}^{pqr} Z_{abc},Z_{pqr}$ where $\varepsilon_{abc}^{pqr}$ is the number of arrows in $Q$ from vertex $abc$ to the vertex $pqr$ minus  the number of arrows from vertex $pqr$ to the vertex $abc$.

	\begin{figure}[h]
	\centering
\psscalebox{0.8}{
\begin{pspicture}(-3,-3)(3,3.5){
%


\pspolygon[linecolor=lightgray,fillstyle=crosshatch,fillcolor=black](-1.25,2)(1.25,2)(0,0)
\pspolygon[linecolor=lightgray,fillstyle=crosshatch,fillcolor=black](1.25,-2)(2.5,0)(0,0)
\pspolygon[linecolor=lightgray,fillstyle=crosshatch,fillcolor=black](-1.25,-2)(-2.5,0)(0,0)

\psarc[doubleline=true,linewidth=3pt, doublesep=8pt, linecolor=red]{<-}(-3.9,-2.0){3}{-20}{70}
\psarc[doubleline=true,linewidth=3pt, doublesep=8pt, linecolor=red]{<-}(0,4){3}{220}{320}
\psarc[doubleline=true,linewidth=3pt, doublesep=8pt, linecolor=red]{<-}(3.9,-2.0){3}{105}{200}

\put(0,0){\psline[linecolor=blue,linewidth=1pt]{-}(0.5,0)(2.,0)}
\put(0,0){\psline[linecolor=blue,linewidth=1pt]{-}(0.25,-0.4)(1.,-1.6)}
\put(0,0){\psline[linecolor=blue,linewidth=1pt]{-}(1.5,-1.6)(2.3,-0.4)}
\put(0,0){\psline[linecolor=blue,linewidth=1pt]{-}(-1.5,-1.6)(-2.3,-0.4)}
\put(0,0){\psline[linecolor=blue,linewidth=1pt]{-}(-0.25,-0.4)(-1.,-1.6)}
\put(0,0){\psline[linecolor=blue,linewidth=1pt]{-}(-0.5,0)(-2.,0)}
\put(0,0){\psline[linecolor=blue,linewidth=2pt,linestyle=dashed]{-}(-3.75,-2.)(0.0,3.5)}
\put(0,0){\psline[linecolor=blue,linewidth=2pt,linestyle=dashed]{-}(0.0,3.5)(3.75,-2)}
\put(0,0){\psline[linecolor=blue,linewidth=2pt,linestyle=dashed]{-}(-3.75,-2.)(3.75,-2)}
\put(0,0){\pscircle[linecolor=black,fillstyle=solid,fillcolor=white]{.35}}
\put(0.3,-0.1){\makebox(0,0)[br]{\hbox{{\tiny $Z_{111}$}}}}
\put(0,0){\psline[linecolor=blue,linewidth=1pt]{-}(0.25,0.4)(1.,1.6)}
\put(1.25,2){\pscircle[linecolor=black,fillstyle=solid,fillcolor=white]{.35}}
\put(1.55,1.9){\makebox(0,0)[br]{\hbox{{\tiny $Z_{012}$}}}}
\put(0,0){\psline[linecolor=blue,linewidth=1pt]{-}(-0.25,0.4)(-1.,1.6)}
\put(-1.25,2){\pscircle[linecolor=black,fillstyle=solid,fillcolor=white]{.35}}
\put(-0.95,1.9){\makebox(0,0)[br]{\hbox{{\tiny $Z_{102}$}}}}
\put(0,0){\psline[linecolor=blue,linewidth=1pt]{-}(0.75,2)(-.75,2)}
\put(2.5,0){\pscircle[linecolor=black,fillstyle=solid,fillcolor=white]{.35}}
\put(2.8,-.1){\makebox(0,0)[br]{\hbox{{\tiny $Z_{021}$}}}}
\put(-2.5,0){\pscircle[linecolor=black,fillstyle=solid,fillcolor=white]{.35}}
\put(-2.2,-.1){\makebox(0,0)[br]{\hbox{{\tiny $Z_{201}$}}}}
\put(1.25,-2){\pscircle[linecolor=black,fillstyle=solid,fillcolor=white]{.35}}
\put(1.55,-2.1){\makebox(0,0)[br]{\hbox{{\tiny $Z_{120}$}}}}
\put(-1.25,-2){\pscircle[linecolor=black,fillstyle=solid,fillcolor=white]{.35}}
\put(-.95,-2.1){\makebox(0,0)[br]{\hbox{{\tiny $Z_{210}$}}}}

\put(-1.5,-1.1){\psframe[framearc=0.2,linecolor=black,fillstyle=solid,fillcolor=white](0,0)(0.8,0.8)}
\put(-.8,-.9){\makebox(0,0)[br]{\hbox{{\Large $\overline T_1$}}}}
\put(0.8,-1.1){\psframe[framearc=0.2,linecolor=black,fillstyle=solid,fillcolor=white](0,0)(0.8,0.8)}
\put(1.,-.9){\makebox(0,0)[bl]{\hbox{{\Large $\widetilde T_1$}}}}
\put(-0.4,0.7){\psframe[framearc=0.2,linecolor=black,fillstyle=solid,fillcolor=white](0,0)(0.8,0.8)}
\put(0,.9){\makebox(0,0)[bc]{\hbox{{\Large $T_1$}}}}

}
\end{pspicture}
}
\caption{\small
FGS-parameters of $\PP_{SL_3,\triangle}$. Double arrows show the direction of transport matrices $T_1,
\widetilde T_1, \overline T_1$.
}
\label{fi:triSL3}
\end{figure}
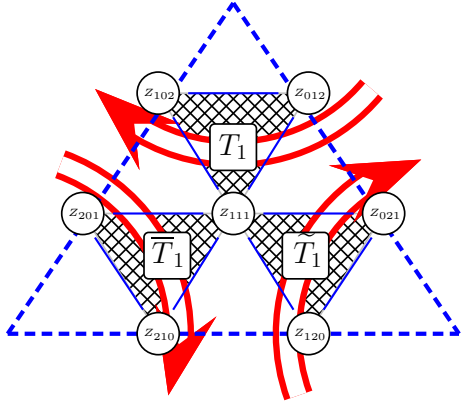

In this paper, we consider only two cases: either $\Sigma$ is the disk with three punctures (marked points) on the boundary, or the disk with four punctures  (marked points) on the boundary. In the first case, we call $\Sigma$ \emph{triangle} and denote it by $\triangle$, in the second case, we call $\Sigma$ \emph{square} and denote it by $\square$.

In this section we solve the  groupoid compatibility problem: \emph{ find the pairs $(B,\mathbb A)$ such that $B\in SL_n$,  $\mathbb A\in\A_n$, and $\widetilde{\mathbb  A}=B\mathbb AB^\text{T}\in\A_n$ in terms of transport matrices in the square.}

Let us consider $\PP_{SL_n,\square}$. Our main object is the network and the corresponding quiver obtained by the \emph{amalgamation} of two FGS triangle-shaped networks (see, Fig.~\ref{quadrangle})~\cite{FG1}. Let $T_1$, $\widetilde T_1$, and $\overline T_1$ be (upper-triangular, non-normalized) transport matrices in the right triangle and let $T_2$, $\widetilde T_2$, and $\overline T_2$ be (lower-triangular, non-normalized) transport matrices in the left triangle. Cluster variables in different triangles Poisson commute. 

\begin{figure}[h]
	\centering
\begin{pspicture}(-2,-1.2)(2,1.2){\psset{unit=0.7}
\rput(1.2,0){\pspolygon[linecolor=blue,linestyle=dashed,linewidth=2pt](-2,-1.7)(2,-1.7)(0,1.7)
\psarc[doubleline=true,linewidth=1pt, doublesep=1pt, linecolor=black]{<-}(-2,-1.7){1.6}{0}{60}
\psarc[doubleline=true,linewidth=1pt, doublesep=1pt, linecolor=black]{<-}(2,-1.7){1.6}{120}{180}
\psarc[doubleline=true,linewidth=1pt, doublesep=1pt, linecolor=black]{<-}(0,1.7){1.6}{240}{300}
}
\rput{180}(-1.2,0){\pspolygon[linecolor=blue,linestyle=dashed,linewidth=2pt](-2,-1.7)(2,-1.7)(0,1.7)
\psarc[doubleline=true,linewidth=1pt, doublesep=1pt, linecolor=black]{->}(-2,-1.7){1.6}{0}{60}
\psarc[doubleline=true,linewidth=1pt, doublesep=1pt, linecolor=black]{->}(2,-1.7){1.6}{120}{180}
\psarc[doubleline=true,linewidth=1pt, doublesep=1pt, linecolor=black]{->}(0,1.7){1.6}{240}{300}
}
\put(1.2,0.5){\makebox(0,0)[bc]{\hbox{{$T_1$}}}}
\put(0.1,-1.6){\makebox(0,0)[bc]{\hbox{{$\overline T_1$}}}}
\put(2.3,-1.6){\makebox(0,0)[bc]{\hbox{{$\widetilde T_1$}}}}
\put(-1.2,-0.5){\makebox(0,0)[tc]{\hbox{{$T_2$}}}}
\put(-0.1,1.6){\makebox(0,0)[tc]{\hbox{{$\widetilde T_2$}}}}
\put(-2.3,1.6){\makebox(0,0)[tc]{\hbox{{$\overline T_2$}}}}
}
\end{pspicture}
\caption{\small
Transport matrices in the triangulation of the square.
}
\label{quadrangle}
\end{figure}
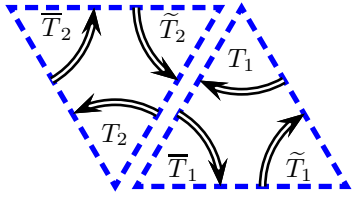

Every transport matrix $T$ satisfies the Lie--Poisson algebra~\cite{ChSh20}
\be\label{LP}
\Bigl\{ \sheet{1}{T}\ocomma \sheet{2}{T}\Bigr\}=-r \sheet{1}{T} \sheet{2}{T}+\sheet{1}{T} \sheet{2}{T}r=r^{\text{T}} \sheet{1}{T} \sheet{2}{T}- \sheet{1}{T} \sheet{2}{T}r^{\text{T}},
\ee
where $r=\sum_{i,j}\theta(i-j)\sheet{1}E_{i,j}\sheet{2}E_{j,i}$ is the trigonometric $r$-matrix; $\theta(x)=\{1\ x>0; 1/2\ x=0; 0\ x<0\}$, and $r+r^\text{T}=P$---the permutation matrix. Having a transport matrix $T$, the transport matrix $T^\star$ in the opposite direction is related to the inverse of $T$ as
\be
T^\star = S T^{-1} S,\ \hbox{where}\ S=\sum_{i=1}^n (-1)^i E_{i, n+1-i},\ S^2=\pm Id,\ S^\text{T}=\pm S
\ee 
with sign $+$ for even $n$ and sign $-$ for odd $n$.

Then the basic relation proved in \cite{ChSh20},\cite{GSV}, 
$$
\Bigl\{ \sheet{1}{T_1}\ocomma \sheet{2}{\widetilde T^\star_1}\Bigr\}=\sheet{1}{T_1} \sheet{2}{\widetilde T^\star_1}r,
$$
implies 
\be
\Bigl\{ \sheet{1}{T_1}\ocomma \sheet{2}{\widetilde T_1}\Bigr\}=- \sheet{1}{T_1} (\sheet{2}S r \sheet{2} S) \sheet{2}{\widetilde T_1}
\ee
with analogous relations holding for the pairs $(\widetilde T_1,\overline T_1)$ and $(\overline T_1, T_1)$ permuted in the cyclic order. 

The corresponding relation for $T_2$ reads
\be
\Bigl\{ \sheet{1}{\widetilde T_2}\ocomma \sheet{2}{T_2}\Bigr\}=- \sheet{2}{T_2} (\sheet{1}S r \sheet{1} S) \sheet{1}{\widetilde T_2},
\ee
with analogous relations holding for other two pairs, $(\overline T_2, \widetilde T_2)$ and $(T_2, \overline T_2)$.

The {\bf groupoid path conditions} 
proved in \cite{ChSh20} in the quantum case, imply the following relations in the semiclassical limit $q\to 1$:
\be\label{gr1}
\widetilde T_1 S \overline T_1 S T_1 S=Id
\ee 
and
\be\label{gr2}
\widetilde T_2 S \overline T_2 S T_2 S=Id.
\ee

We use the network Fig.~\ref{quadrangle} to solve relaxed (\emph{non-normalized}) {\bf symplectic groupoid compatibility }problem. Let $\BB_n$ be the Borel subgroup of nondgenerate upper-triangular matrices of $SL_n$.   We say that the pair $(B,\mathbb A)$, $B\in SL_n$, $\mathbb A\in \BB_n$ satisfies the \emph{non-normalized} symplectic groupoid compatibility condition if $\mathbb A\in \BB_n$, $B\mathbb A B^\text{T}\in\BB_n$.
The following construction solves the non-normalized symplectic groupoid compatibility condition for the network depicted in Fig.~\ref{quadrangle}.

\begin{theorem}
For the matrix $B$ given by the product of transport matrices
\be\label{B}
B=T_2T_1.
\ee
in the network depicted in Fig.~(\ref{quadrangle}), the non-normalized symplectic groupoid condition that both $\mathbb A$ and $\widetilde{\mathbb A}:=B\mathbb AB^\text{T}$ are upper-triangular is resolved by taking
\be\label{A}
\mathbb A=T_1^{-1}S\widetilde T_2\widetilde T_1^\text{T} S.
\ee
Then
\be\label{t-A}
\widetilde{\mathbb A}:=B\mathbb AB^\text{T}=S [\overline T_2]^{-1}\bigl[ \overline T_1^\text{T}\bigr]^{-1} S T_2^\text{T}
\ee
is automatically upper triangular itself. When deriving these expressions we used the groupoid path relations (\ref{gr1}) and (\ref{gr2}).
\end{theorem}

\begin{proof} 
Indeed, factorize $\mathbb A=T_1^{-1}S\widetilde T_2\widetilde T_1^\text{T} S=
(-1)^{n+1}\cdot T_1^{-1}\cdot S\widetilde T_2 S\cdot S\widetilde T_1^\text{T} S$. Note that all the matrices  $T_1$, $S T_2 S$, and $S T_1^\text{T} S$ are upper-triangular by construction which implies that $\mathbb A$ is upper-triangular too.
Similarly, $\widetilde{\mathbb A}=S [\overline T_2]^{-1}\bigl[ \overline T_1^\text{T}\bigr]^{-1} S T_2^\text{T}=(-1)^{n+1}\cdot S [\overline T_2]^{-1}S\cdot S\bigl[ \overline T_1^\text{T}\bigr]^{-1} S\cdot T_2^\text{T}$ is upper-triangular and it remains to show that 
$\widetilde{\mathbb A}=B\mathbb A B^\text{T}$.

$B\mathbb A B^\text{T}=T_2 T_1 T_1^{-1}S\widetilde T_2\widetilde T_1^\text{T} S T_1^\text{T} T_2^\text{T}=
\left(T_2 S\widetilde T_2\right)\cdot\left(\widetilde T_1^\text{T} S T_1^\text{T}\right) T_2^\text{T}$. Using groupoid conditions~(\ref{gr1},\ref{gr2})
we rewrite the latter expression as $\left(S[\overline T_2]^{-1} S\right)\cdot\left(S[\overline T_1^\text{T}]^{-1} S\right)\cdot T_2^\text{T}=
S [\overline T_2]^{-1}[\overline T_1^\text{T}]^{-1} S T_2^\text{T}=\widetilde{\mathbb A}$.
\end{proof}

Now we discuss the normalization conditions that matrices  $\mathbb A$ and  $\widetilde{\mathbb A}$ are unipotent.
For  $B\in SL_n$, define $\delta_k=\det(B_{[n-k+1,n]}^{[1,k]})$, $\tilde\delta_k= \det(B_{[1,n-k]}^{[k+1,n]})$ for all $k\in[1,n-1]$, $\delta_0=\tilde\delta_4=1$, $\delta_4=\tilde\delta_0=\det(B)$. Here,  $I,J$ are two subsets of $[1,n]$ of the same cardinality $|I|=|J|=\ell$,
  $B_I^J$ denotes the $\ell\times\ell$ submatrix of $B$ formed by rows from $I$ and columns from $J$. 

\begin{lemma} 
Let $B\in SL_n$ satisfy  conditions $\delta_k/\tilde\delta_k=1$ for all $k\in [1,n-1]$. Then we have a unique $\mathbb A\in \A_n$ such that $B\mathbb A B^\text{T}\in\A_n$.
\end{lemma}

\begin{proof} Express variables $\mathbb A_{k,\ell}$  $1\le k<\ell \le n$ from the system of equations that the lower triangular part of $\tilde{\mathbb A}=B\mathbb A B^\text{T}$ is zero and substitute these expressions in diagonal elements of $\tilde{\mathbb A}$.  We obtain
$$
\dfrac{\tilde{\mathbb A}_{k,k}}{\mathbb A_{k,k}}=(-1)^{n+1}\left(\dfrac{ \tilde\delta_{n-k}}{\delta_{n-k}}\right)\left(\dfrac{\delta_{n-k+1}}{\tilde\delta_{n-k+1}}\right).
$$ 
Assumig unipotency condition $\widetilde{\mathbb A}_{k,k}={\mathbb A_{k,k}}=1$ these equations imply that
there exists a unique solution of normalized symplectic groupoid compatibility condition if $\delta_k/\tilde\delta_k=(-1)^{(n-k)(n+1)}\ \forall k\in[0,n]$.
\end{proof}

\begin{remark} Note that  functions $\det(B), \delta_k/\widetilde{\delta}_k$  provide complete collections of independent Casimir functions of the standard 
$R$-matrix Poisson-Lie bracket for $GL_n$.
\end{remark}

\begin{remark} It is easy to satisfy conditions of unipotency of $\mathbb A$ and $\delta_i/\widetilde\delta_i=1$ simultaneously (see Theorem~\ref{Casimirs}) in terms of network Fig.~\ref{quadrangle}.
Since the matrices $\mathbb A$ and $\widetilde{\mathbb A}$ are uniquely determined by the matrix $B$ in a general position, and by the condition of unipotency for the matrix $\mathbb A$, the above solution is, in fact, unique! If, in addition, we require the matrix 
$\widetilde{\mathbb A}$ to be unipotent, it imposes additional restrictions on Casimir elements of the matrix $B$, see \cite{Bondal}. We resolve these conditions in the next section.
\end{remark}

\begin{remark} The quantum version of transport matrices were considered in \cite{ChSh20}. Their entries are elements of quantum torus algebra. It is straightforward to generalize the construction of Theorem~\ref{A} to the upper-triangular matrices with entries in quantum torus algebra.
The quantum construction will be published in a separate paper.
\end{remark}

\subsection{Semiclassical symplectic groupoid algebra}



One of the most attractive features of the symplectic groupoid construction is that elements of any admissible pair $(B,\mathbb A)$ satisfy closed Poisson algebra. In \cite{ChM4} it was shown that the Poisson--Lie relations on the matrix $B$ induce the reflection equation relations on $\mathbb A$ and $\widetilde {\mathbb A}$ as well as all other commutation relations between $B$ and $\mathbb A$ and $\widetilde {\mathbb A}$; in particular, elements of $\mathbb A$ and $\widetilde {\mathbb A}$ mutually Poisson commute. The proof in \cite{ChM4} was straightforward, but technically cumbersome: it was based on expressing elements of $\mathbb A$ in terms of $B$ provided that $B\mathbb AB^{\text{T}}$ is unipotent and on deriving induced Poisson relations.  We now derive all these algebraic relations from Poisson relations enjoyed by transport matrices.

\begin{theorem}\label{thm:reflectionPB}
The trigonometric Poisson-Lie bracket on $SL_n$ induces \emph{reflection} Poisson bracket on $\A_n$ and other Poisson relations between elements of $B$, $\mathbb A$, and $\widetilde {\mathbb A}$. 
\end{theorem}

\subsubsection{$B-B$ relations}
That $B=T_2 T_1$ itself satisfies the Poisson--Lie algebra follows immediately from that each matrix $T_1$ and $T_2$ satisfies the Poisson--Lie bracket and $T_1$ Poisson commute with $T_2$:
\begin{align}
\bigl\{ \sheet{1} B \ocomma \sheet{2}{B}\bigr\}=&\bigl\{ \sheet{1} T_2 \ocomma \sheet{2}T_2\bigr\}  \sheet{1} {T_1}\sheet{2} T_1 +  \sheet{1} T_2\sheet{2} T_2 \bigl\{ \sheet{1} T_1 \ocomma \sheet{2}T_1\bigr\}=r   \sheet{1} T_2 \sheet{2}T_2  \sheet{1} T_1\sheet{2} T_1 -   \sheet{1} T_2 \sheet{2}T_2  r \sheet{1} T_1\sheet{2} T_1 +   \sheet{1} T_2 \sheet{2}T_2  r \sheet{1} T_1\sheet{2} T_1 -  \sheet{1} T_2 \sheet{2}T_2  \sheet{1} T_1\sheet{2} T_1 r\nonumber\\
&=  r \sheet{1} T_2 \sheet{2}T_2  \sheet{1} T_1\sheet{2} T_1 -   \sheet{1} T_2 \sheet{2}T_2  \sheet{1} T_1\sheet{2} T_1 r= r   \sheet{1} B \sheet{2}B - \sheet{1} B \sheet{2}B r. 
\end{align}

\subsubsection{$B-\mathbb A$ relations}
Here we have
\begin{align}
\bigl\{ \sheet{1} B \ocomma \sheet{2}{\mathbb A}\bigr\}=
& \sheet{2} {T_1^{-1}}\sheet{2} S \bigl\{ \sheet{1} T_2 \ocomma \sheet{2}{\widetilde T_2}\bigr\} \sheet{2} {\widetilde T_1^\text{T}}\sheet{2} S\sheet{1}T_1
+\sheet{1} T_2  \bigl\{ \sheet{1} T_1 \ocomma \sheet{2} {T_1^{-1}} \bigr\} \sheet{2} S \sheet{2}{\widetilde T_2} \sheet{2} {\widetilde T_1^\text{T}}\sheet{2} S
+\sheet{1} T_2  \sheet{2} {T_1^{-1}} \sheet{2} S \sheet{2}{\widetilde T_2} \bigl\{ \sheet{1} T_1 \ocomma \sheet{2} {\widetilde T_1^\text{T}}  \bigr\}  \sheet{2} S\nonumber\\
=& \sheet{2}{ T_1^{-1}}\sheet{2} S \bigl( \sheet{1} T_2 \sheet{2}S r^\text{T}\sheet{2}S \sheet{2}{\widetilde T_2}\bigr) \sheet{2} {\widetilde T_1^\text{T}}\sheet{2} S\sheet{1}T_1
+\sheet{1} T_2  \bigl( -\sheet{2} {T_1^{-1}} r^\text{T} \sheet{1} T_1 +  \sheet{1} T_1 r^\text{T} \sheet{2} T_1^{-1} \bigr) \sheet{2} S \sheet{2}{\widetilde T_2} \sheet{2} {\widetilde T_1^\text{T}}\sheet{2} S
+\sheet{1} T_2  \sheet{2} {T_1^{-1}} \sheet{2} S \sheet{2}{\widetilde T_2} \bigl( - \sheet{1} T_1  \sheet{2} {\widetilde T_1^\text{T} }\sheet{2}S r^{t_2}\sheet{2}S  \bigr) \sheet{2} S\nonumber\\
=&\sheet{1} T_2   \sheet{1} T_1 r^\text{T} \sheet{2} {T_1^{-1}} \sheet{2} S \sheet{2}{\widetilde T_2} \sheet{2} {\widetilde T_1^\text{T}}\sheet{2} S
-\sheet{1} T_2  \sheet{2} {T_1^{-1}} \sheet{2} S \sheet{2}{\widetilde T_2} \sheet{1} T_1  \sheet{2} {\widetilde T_1^\text{T}} \sheet{2}S r^{t_2}\sheet{2}S \sheet{2} S\nonumber\\
=&\sheet{1}B r^\text{T} \sheet{2}{\mathbb A}- \sheet{1}B  \sheet{2}{\mathbb A} r^{t_2},
\end{align}
where $r^{t_2}$ is a partially transposed $r$-matrix.

\subsubsection{$B-\widetilde{\mathbb A}$ relations}
\begin{align}
\bigl\{ \sheet{1} B \ocomma \sheet{2}{\widetilde{\mathbb A}}\bigr\}=
& \sheet{2} S  \bigl\{ \sheet{1} T_2 \ocomma \sheet{2}{\overline T_2^{-1}}\bigr\} \bigl[\sheet{2} {\overline T_1^\text{T}}\bigr]^{-1}\sheet{2} S\sheet{2}{T_2^\text{T}}\sheet{1}T_1
+\sheet{2} S \sheet{2}{\overline T_2^{-1}}  \bigl[\sheet{2} {\overline T_1^\text{T}}\bigr]^{-1}  \sheet{2} S \bigl\{ \sheet{1} T_2 \ocomma \sheet{2}{T_2^\text{T}} \bigr\} \sheet{1}T_1
+ \sheet{1} T_2 \sheet{2} S \sheet{2}{\overline T_2^{-1}}   \bigl\{  \sheet{1}T_1  \ocomma  \bigl[\sheet{2} {\overline T_1^\text{T}}\bigr]^{-1} \bigr\}  \sheet{2} S  \sheet{2}{T_2^\text{T}}
\nonumber\\
=& \sheet{2} S  \bigl(\sheet{1}S r \sheet{1}S \sheet{1}T_2 \sheet{2}{\overline T_2^{-1}}\bigr) \bigl[\sheet{2} {\overline T_1^\text{T}}\bigr]^{-1}\sheet{2} S\sheet{2}{T_2^\text{T}}\sheet{1}T_1
+\sheet{2} S \sheet{2}{\overline T_2^{-1}}  \bigl[\sheet{2} {\overline T_1^\text{T}}\bigr]^{-1}  \sheet{2} S \bigl( -  \sheet{2}{T_2^\text{T}} r^{t_2} \sheet{1} T_2 + \sheet{1} T_2 r^{t_2} \sheet{2}{T_2^\text{T}} \bigr) \sheet{1}T_1
- \sheet{1} T_2 \sheet{2} S \sheet{2}{\overline T_2^{-1}}   \bigl(  \bigl[\sheet{2} {\overline T_1^\text{T}}\bigr]^{-1} \sheet{1} S r^{t_1}\sheet{1} S  \sheet{1}T_1 \bigr)  \sheet{2} S  \sheet{2}{T_2^\text{T}}
\nonumber\\
=&r^\text{T} \sheet{2}S \sheet{1}T_2 \sheet{2}{\overline T_2^{-1}} \bigl[\sheet{2} {\overline T_1^\text{T}}\bigr]^{-1}\sheet{2} S\sheet{2}{T_2^\text{T}}\sheet{1}T_1
-\sheet{2} S \sheet{2}{\overline T_2^{-1}}  \bigl[\sheet{2} {\overline T_1^\text{T}}\bigr]^{-1}  \sheet{2} S \sheet{2}{T_2^\text{T}} r^{t_2} \sheet{1} T_2  \sheet{1}T_1
\nonumber\\
=&r^\text{T}\sheet{1}B  \sheet{2}{\widetilde{\mathbb A}}-  \sheet{2}{\widetilde{\mathbb A}} r^{t_2}\sheet{1}B,
\end{align}
where we have used that $r^{t_1}\sheet{1} S\sheet{2} S=\sheet{1} S\sheet{2} S r^{t_2}$ and $\sheet{1} S\sheet{2} S r= r^{\text{T}}\sheet{1} S\sheet{2} S$. 

\subsubsection{$\mathbb A-\widetilde{\mathbb A}$ relations}
\begin{align}
\bigl\{ \sheet{1}{\mathbb A} \ocomma \sheet{2}{\widetilde{\mathbb A}}\bigr\}=
& \sheet{2} S \sheet{2}{{\overline T}_2^{-1}}  \bigl\{ \sheet{1}{ T_1^{-1}} \ocomma \bigl[\sheet{2} {\overline T_1^\text{T}}\bigr]^{-1}\bigr\} \sheet{2} S\sheet{2}{T_2^\text{T}}\sheet{1} S \sheet{1}{\widetilde{T_2}}\sheet{1}{\widetilde{T_1}^\text{T}}\sheet{1} S
+\sheet{2} S \sheet{2}{{\overline T}_2^{-1}} \sheet{1}{ T_1^{-1}} \sheet{1} S \sheet{1}{\widetilde{T_2}} \bigl\{ \sheet{1}{\widetilde{T_1}^\text{T}}  \ocomma \bigl[\sheet{2} {\overline T_1^\text{T}}\bigr]^{-1}\bigr\} \sheet{2} S\sheet{2}{T_2^\text{T}} \sheet{1} S\nonumber\\
&+ \sheet{2} S  \sheet{1}{ T_1^{-1}} \sheet{1} S  \bigl\{ \sheet{1}{\widetilde{T_2}} \ocomma \sheet{2}{{\overline T}_2^{-1}}  \bigr\} \bigl[\sheet{2} {\overline T_1^\text{T}}\bigr]^{-1} \sheet{2} S\sheet{2}{T_2^\text{T}} \sheet{1}{\widetilde{T_1}^\text{T}}\sheet{1} S
+ \sheet{2} S \sheet{2}{{\overline T}_2^{-1}} \bigl[\sheet{2} {\overline T_1^\text{T}}\bigr]^{-1} \sheet{2} S  \sheet{1}{ T_1^{-1}} \sheet{1} S  \bigl\{ \sheet{1}{\widetilde{T_2}} \ocomma  \sheet{2}{T_2^\text{T}}   \bigr\}  \sheet{1}{\widetilde{T_1}^\text{T}}\sheet{1} S\nonumber\\
=& \sheet{2} S \sheet{2}{{\overline T}_2^{-1}}  \bigl(  \sheet{1}{ T_1^{-1}} \bigl[ \sheet{2} {\overline T_1^\text{T}}\bigr]^{-1} \sheet{1}S r^{t_1}\sheet{1}S\bigr) \sheet{2} S\sheet{2}{T_2^\text{T}}\sheet{1} S \sheet{1}{\widetilde{T_2}}\sheet{1}{\widetilde{T_1}^\text{T}}\sheet{1} S
+\sheet{2} S \sheet{2}{{\overline T}_2^{-1}} \sheet{1}{ T_1^{-1}} \sheet{1} S \sheet{1}{\widetilde{T_2}} \bigl( \sheet{2}S r^{\text{T}}\sheet{2}S \sheet{1}{\widetilde{T_1}^\text{T}}  \bigl[\sheet{2} {\overline T_1^\text{T}}\bigr]^{-1}\bigr) \sheet{2} S\sheet{2}{T_2^\text{T}} \sheet{1} S\nonumber\\
&+ \sheet{2} S  \sheet{1}{ T_1^{-1}} \sheet{1} S  \bigl( -\sheet{1}{\widetilde{T_2}} \sheet{2}{{\overline T}_2^{-1}} \sheet{2}S r^{\text{T}}\sheet{2}S \bigr) \bigl[\sheet{2} {\overline T_1^\text{T}}\bigr]^{-1} \sheet{2} S\sheet{2}{T_2^\text{T}} \sheet{1}{\widetilde{T_1}^\text{T}}\sheet{1} S
+ \sheet{2} S \sheet{2}{{\overline T}_2^{-1}} \bigl[\sheet{2} {\overline T_1^\text{T}}\bigr]^{-1} \sheet{2} S  \sheet{1}{ T_1^{-1}} \sheet{1} S  \bigl( -\sheet{1}S r^{t_2}\sheet{1}S \sheet{1}{\widetilde{T_2}} \sheet{2}{T_2^\text{T}}   \bigr)  \sheet{1}{\widetilde{T_1}^\text{T}}\sheet{1} S.\nonumber
\end{align}
In this expression, the first term is canceled with the fourth term and the second term is canceled with the third term, so we obtain that
\be\label{A-At}
\bigl\{ \sheet{1}{\mathbb A} \ocomma \sheet{2}{\widetilde{\mathbb A}}\bigr\}=0.
\ee

\subsubsection{$\mathbb A-{\mathbb A}$ relations}
\begin{align}
\bigl\{ \sheet{1}{\mathbb A} \ocomma \sheet{2}{\mathbb A}\bigr\}=
&   \bigl\{ \sheet{1}{ T_1^{-1}} \ocomma \sheet{2}{ T_1^{-1}} \bigr\}  \sheet{1} S\sheet{1}{\widetilde T_2} \sheet{1}{\widetilde {T_1^\text{T}}}\sheet{1} S \sheet{2} S\sheet{2}{\widetilde T_2} \sheet{2}{\widetilde {T_1^\text{T}}}\sheet{2} S 
+\sheet{1}{ T_1^{-1}} \sheet{1} S \sheet{2}{ T_1^{-1}} \sheet{2} S \bigl\{ \sheet{1}{\widetilde T_2} \ocomma \sheet{2}{\widetilde T_2} \bigr\}  \sheet{1}{\widetilde {T_1^\text{T}}}\sheet{1} S  \sheet{2}{\widetilde {T_1^\text{T}}}\sheet{2} S 
+\sheet{1}{ T_1^{-1}} \sheet{1} S \sheet{1}{\widetilde T_2}  \sheet{2}{ T_1^{-1}} \sheet{2} S \sheet{2}{\widetilde T_2}  \bigl\{ \sheet{1}{\widetilde {T_1^\text{T}}} \ocomma   \sheet{2}{\widetilde {T_1^\text{T}}} \bigr\}  \sheet{1} S \sheet{2} S \nonumber\\
& +\sheet{1}{ T_1^{-1}} \sheet{1} S  \sheet{1}{\widetilde T_2}  \bigl\{ \sheet{1}{\widetilde {T_1^\text{T}}} \ocomma \sheet{2}{ T_1^{-1}}  \bigr\}  \sheet{1} S  \sheet{2} S \sheet{2}{\widetilde T_2} \sheet{2}{\widetilde {T_1^\text{T}}}\sheet{2} S 
+\sheet{2}{ T_1^{-1}} \sheet{2} S  \sheet{2}{\widetilde T_2}  \bigl\{ \sheet{1}{ T_1^{-1}}  \ocomma  \sheet{2}{\widetilde {T_1^\text{T}}} \bigr\}  \sheet{2} S  \sheet{1} S \sheet{1}{\widetilde T_2} \sheet{1}{\widetilde {T_1^\text{T}}}\sheet{1} S \nonumber\\
=&   \bigl( r\sheet{1}{ T_1^{-1}} \sheet{2}{ T_1^{-1}}- \sheet{1}{ T_1^{-1}} \sheet{2}{ T_1^{-1}}r \bigr)  \sheet{1} S\sheet{1}{\widetilde T_2} \sheet{1}{\widetilde {T_1^\text{T}}}\sheet{1} S \sheet{2} S\sheet{2}{\widetilde T_2} \sheet{2}{\widetilde {T_1^\text{T}}}\sheet{2} S 
+\sheet{1}{ T_1^{-1}} \sheet{1} S \sheet{2}{ T_1^{-1}} \sheet{2} S \bigl(r^\text{T} \sheet{1}{\widetilde T_2} \sheet{2}{\widetilde T_2} -\sheet{1}{\widetilde T_2} \sheet{2}{\widetilde T_2}r^\text{T} \bigr)  \sheet{1}{\widetilde {T_1^\text{T}}}\sheet{1} S  \sheet{2}{\widetilde {T_1^\text{T}}}\sheet{2} S \nonumber\\
&\qquad\qquad+\sheet{1}{ T_1^{-1}} \sheet{1} S \sheet{1}{\widetilde T_2}  \sheet{2}{ T_1^{-1}} \sheet{2} S \sheet{2}{\widetilde T_2}  \bigl( r^\text{T}\sheet{1}{\widetilde {T_1^\text{T}}}  \sheet{2}{\widetilde {T_1^\text{T}}} - \sheet{1}{\widetilde {T_1^\text{T}}}  \sheet{2}{\widetilde {T_1^\text{T}}}r^\text{T} \bigr)  \sheet{1} S \sheet{2} S \nonumber\\
& +\sheet{1}{ T_1^{-1}} \sheet{1} S  \sheet{1}{\widetilde T_2}  \bigl( -\sheet{1}{\widetilde {T_1^\text{T}}} \sheet{1}S r^{t_2}\sheet{1}S \sheet{2}{ T_1^{-1}}  \bigr)  \sheet{1} S  \sheet{2} S \sheet{2}{\widetilde T_2} \sheet{2}{\widetilde {T_1^\text{T}}}\sheet{2} S 
+\sheet{2}{ T_1^{-1}} \sheet{2} S  \sheet{2}{\widetilde T_2}  \bigl(  \sheet{2}{\widetilde {T_1^\text{T}}} \sheet{2}S r^{t_2}\sheet{2}S  \sheet{1}{ T_1^{-1}}   \bigr) \sheet{2} S  \sheet{1} S \sheet{1}{\widetilde T_2} \sheet{1}{\widetilde {T_1^\text{T}}}\sheet{1} S \nonumber\\
=&r \sheet{1}{\mathbb A}\sheet{2}{\mathbb A}-\sheet{1}{\mathbb A}\sheet{2}{\mathbb A}r- \sheet{1}{\mathbb A} r^{t_2}\sheet{2}{\mathbb A} + \sheet{2}{\mathbb A} r^{t_2} \sheet{1}{\mathbb A}=-r^{\text{T}}\sheet{1}{\mathbb A}\sheet{2}{\mathbb A} +\sheet{1}{\mathbb A}\sheet{2}{\mathbb A}r^\text{T}- \sheet{1}{\mathbb A} r^{t_2}\sheet{2}{\mathbb A} + \sheet{2}{\mathbb A} r^{t_2} \sheet{1}{\mathbb A},\label{A-A}
\end{align}
that is, we have obtained the semiclassical {\bf reflection equation}.

\subsubsection{$\widetilde{\mathbb A}-\widetilde{\mathbb A}$ relations}
Analogously to the previous relation, we obtain
\be\label{At-At}
\bigl\{ \sheet{1}{\widetilde{\mathbb A}} \ocomma \sheet{2}{\widetilde{\mathbb A}}\bigr\}=-r \sheet{1}{\widetilde{\mathbb A}}\sheet{2}{\widetilde{\mathbb A}} +\sheet{1}{\widetilde{\mathbb A}}\sheet{2}{\widetilde{\mathbb A}}r+\sheet{1}{\widetilde{\mathbb A}} r^{t_2}\sheet{2}{\widetilde{\mathbb A}} - \sheet{2}{\widetilde{\mathbb A}} r^{t_2} \sheet{1}{\widetilde{\mathbb A}},
\ee
so the mapping $\mathbb A\mapsto \widetilde{\mathbb A}$ is anti-Poisson.


\section{Casimir elements}

In this section we will describe, first, the Casimir elements (or, using a widely accepted jargon, Casimirs) of the Goldman Poisson bracket collection on the moduli space of pinning $\PP_{SL_n,\square}$  and, then, of the reflection bracket on $\A_n$. Since FGS-parameters $Z_{abc}$ are log-canonical coordinates for both brackets, it is convenient to use these parameters for computing Casimirs in both cases. Indeed, the bracket becomes constant in $\log Z_{abc}$, and the equations for Casimirs become linear. In particular, it implies that the generators of the algebra of Casimir elements can be expressed as Laurent monomials in $Z_{abc}$.

Below we show the pictorial representation of monomials $C_i$ and $\tilde C_i$ that describe Casimirs of both brackets. One picture describes one monomial. Each vertex of a quiver below is assigned a corresponding FGS-parameter. Blue arrows of the quiver describe the Poisson-Lie bracket
between parameters assigned to endpoints of the arrow.
Red numbers at the vertices mean the exponent of the corresponding variable in the monomial $C_i$ or $\tilde C_i$ as indicated in the caption; unnumbered vertices do not contribute.

\begin{figure}[H]
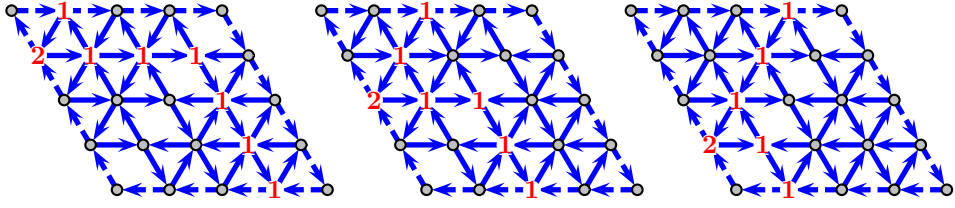

\begin{center}

\end{center}
\caption{\small
Pre-Casimirs $\widetilde C_2$, $\widetilde C_3$, and $\widetilde C_4$ (from left to right) for the full-rank quiver $Q_\square$ for $\PP_{GL_4}$.
}
\label{fi:Casimirs-1}
\end{figure}

\begin{figure}[H]
\begin{center}
%
\end{center}
\caption{\small
Pre-Casimirs $C_2$, $C_3$, and $C_4$ (from left to right) for the full-rank quiver $Q_\square$ for $\PP_{GL_4}$.
}
\label{fi:Casimirs-2}
\end{figure}

\begin{figure}[H]
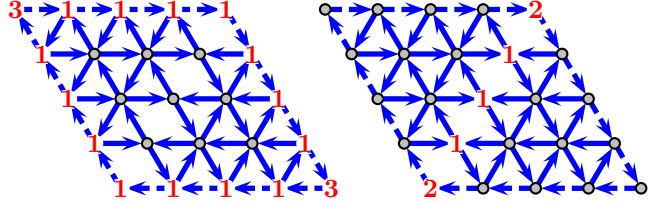

\begin{center}
%
\end{center}
\caption{\small
Two exceptional Casimirs $C_0$ and $C_1$ (from left to right) for the full-rank quiver $Q_\square$ for $\PP_{GL_4}$.
}
\label{fi:Casimirs-3}
\end{figure}

The 
quiver $Q_\square$ associated to the moduli space of pinnings $\PP_{GL_n,\square}$ 
therefore contains $(n+1)^2$ cluster variables.  The vertices of the quiver form a skew (rhombus-shaped) lattice. Let's label vertices by pairs $(i,j)\in[0,n]\times [0,n]$, $i$ increasing from left to right, $j$ from bottom to top, and the corresponding cluster variable by $K_{ij}$. 

Note that not all variables $K_{ij}$ of $Q_\square$ are used in expression of transport matrix $B$. Namely, all variables in the bottom row do not contribute to $B$. 
Also, only the product of all variables in the top row enters as a factor in the expression for any entry of matrix $B$. The corresponding quiver is shown on Fig.~\ref{fi:Casimirs-4}.

\begin{figure}[H]
\begin{center}
\begin{pspicture}(-2,-0.7)(2,1.2){\psset{unit=0.7}
\newcommand{\ODIN}{%
{\psset{unit=1}
\put(0,0){\pscircle[fillstyle=solid,fillcolor=white,linecolor=white]{.15}}
\put(0,0){\makebox(0,0)[cc]{\hbox{\tcw{\large$\mathbf 1$}}}}
\put(0,0){\makebox(0,0)[cc]{\hbox{\tcr{$\mathbf 1$}}}}
}}
\newcommand{\DVA}{%
{\psset{unit=1}
\put(0,0){\pscircle[fillstyle=solid,fillcolor=white,linecolor=white]{.15}}
\put(0,0){\makebox(0,0)[cc]{\hbox{\tcw{\large$\mathbf 2$}}}}
\put(0,0){\makebox(0,0)[cc]{\hbox{\tcr{$\mathbf 2$}}}}
}}
\newcommand{\TRI}{%
{\psset{unit=1}
\put(0,0){\pscircle[fillstyle=solid,fillcolor=white,linecolor=white]{.15}}
\put(0,0){\makebox(0,0)[cc]{\hbox{\tcw{\large$\mathbf 3$}}}}
\put(0,0){\makebox(0,0)[cc]{\hbox{\tcr{$\mathbf 3$}}}}
}}
\multiput(-2.5,0.85)(1,0){3}{\psline[linecolor=blue,linewidth=2pt]{->}(0.1,0)(.9,0)}
\multiput(-2,0)(1,0){2}{\psline[linecolor=blue,linewidth=2pt]{->}(0.1,0)(.9,0)}
\multiput(-1.5,-0.85)(1,0){1}{\psline[linecolor=blue,linewidth=2pt]{->}(0.1,0)(.9,0)}
\multiput(2.5,-0.85)(-1,0){3}{\psline[linecolor=blue,linewidth=2pt]{->}(-0.1,0)(-.9,0)}
\multiput(2,-0)(-1,0){2}{\psline[linecolor=blue,linewidth=2pt]{->}(-0.1,0)(-.9,0)}
\multiput(1.5,0.85)(-1,0){1}{\psline[linecolor=blue,linewidth=2pt]{->}(-0.1,0)(-.9,0)}
\multiput(1.5,0.85)(0.5,-0.85){2}{\psline[linecolor=blue,linestyle=dashed,linewidth=2pt]{->}(0.05,-0.085)(.45,-0.765)}
\multiput(.5,.85)(0.5,-0.85){2}{\psline[linecolor=blue,linewidth=2pt]{->}(0.05,-0.085)(.45,-0.765)}
\multiput(0,0)(0.5,-0.85){1}{\psline[linecolor=blue,linewidth=2pt]{->}(0.05,-0.085)(.45,-0.765)}
\multiput(-1.5,-0.85)(-0.5,0.85){2}{\psline[linecolor=blue,linestyle=dashed,linewidth=2pt]{->}(-0.05,0.085)(-.45,0.765)}
\multiput(-0.5,-.85)(-0.5,0.85){2}{\psline[linecolor=blue,linewidth=2pt]{->}(-0.05,0.085)(-.45,0.765)}
\multiput(0,0)(-0.5,0.85){1}{\psline[linecolor=blue,linewidth=2pt]{->}(-0.05,0.085)(-.45,0.765)}
%
\multiput(0.5,-0.85)(0.5,0.85){2}{\psline[linecolor=blue,linewidth=2pt]{->}(0.05,0.085)(.45,0.765)}
\multiput(1.5,-0.85)(0.5,0.85){1}{\psline[linecolor=blue,linewidth=2pt]{->}(0.05,0.085)(.45,0.765)}
\multiput(-0.5,0.85)(-0.5,-0.85){2}{\psline[linecolor=blue,linewidth=2pt]{->}(-0.05,-0.085)(-.45,-0.765)}
\multiput(-1.5,0.85)(-0.5,-0.85){1}{\psline[linecolor=blue,linewidth=2pt]{->}(-0.05,-0.085)(-.45,-0.765)}
%
%
%
%
%
\put(-2.5,0.95){\psline[linecolor=blue,linestyle=dashed,linewidth=2pt]{<-}(0.05,-0.085)(2.,0.765)}
\put(1.5,0.95){\psline[linecolor=blue,linestyle=dashed,linewidth=2pt]{<-}(0.05,-0.085)(-2.,0.765)}
\put(0.5,0.95){\psline[linecolor=blue,linewidth=2pt]{->}(0.05,-0.085)(-1.,0.765)}
\multiput(-2.5,0.85)(0.5,-0.85){3}{
\multiput(0,0)(1,0){5}{\pscircle[fillstyle=solid,fillcolor=lightgray]{.1}}}
\put(-0.5,1.7){\pscircle[fillstyle=solid,fillcolor=lightgray]{.1}}
}
\end{pspicture}
\end{center}
\caption{\small
The quiver of factorization parameters for matrix $B\in GL_n$, $n=4$ (the quiver for $B$-system). The cluster variable at the top vertex is the product $\prod_{j=0}^n K_{n,j}$. In the remaining part of the quiver the vertex $(i,j)\in[1,n-1]\times [0,n]$ carries the variable $K_{i,j}$.
}
\label{fi:Casimirs-4}
\end{figure}
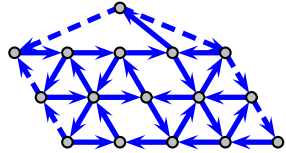

\begin{theorem}\label{Casimirs}
The full list of Casimir elements is as follows:
\begin{itemize}
\item[(i)] 
$Q_\square$ (see, Figures~\ref{fi:Casimirs-1},\ref{fi:Casimirs-2},\ref{fi:Casimirs-3}) has exactly $n+1$ Casimirs: $C_0$, $C_1$ and $\widetilde C_iC_i$ for $i=2,\dots,n$.
\item[(ii)] the quiver for the transport matrix $B$ normalized to the unit determinant ($\det B=1$) contains exactly $n-1$ Casimirs $\widetilde C_i/C_i$. 
\item[(iii)] the quiver for $\mathcal A_n$ obtained from the quiver $Q_\square$ with Moebius-like amalgamation of variables in the upper and lower rows contains $n(n+1)$ cluster variables and has exactly $2n$ Casimirs: $C_0$, $C_1$ and all $\widetilde C_i$ and $C_i$ for $i=2,\dots,n$.
\end{itemize}
\end{theorem}

\begin{proof}  (i) The fact that monomials $C_0$, $C_1$ and $C_i\widetilde C_i$ for $i=2,\dots,n$ are Casimirs of the Poisson -Lie bracket follows immediately from the Figs.~\ref{fi:Casimirs-1}, \ref{fi:Casimirs-2}, and \ref{fi:Casimirs-3} and the definition of monomials by counting the signed number of arrows connecting the monomial with any vertex of the quiver. It is easy to observe that 
$\{\log(C_0),\log(C_1),\log(C_i)+\log(\widetilde C_i)\}$ are linearly independent functions, hence  
$\{C_0,C_1\}\cup\{ \C_i\widetilde C_i\}_{i=2,\dots,n}$ form the set of transcendentally independent generators. Finally, since the Poisson bracket in logarithms of FGS-parameters is constant, its corank can be directly computed and coincides with the dimension $n+1$ of the nullspace of the skew-symmetric coefficient matrix which implies that there are exactly $n+1$ generators of the field of Casimir functions. 

(ii) Denote by $K_i$ the product of all variables in the $i$th row $K_i=\prod_{j=0}^n K_{i,j}$, $i\in [1,n]$.  Obviously, $\{K_i,K_j\}=0$. Determinant of transport matrix $\det(B)=K_n^n\cdot K_{n-1}^{n-1}\cdot\dots\cdot K_2^2\cdot K_1$. Hence, the condition $\det(B)=1$ is rewritten as $K_n=\dfrac{1}{\prod_{j=1}^{n-1} K_j^{j/n}}$. It is easy to observe that $\{\det(B),B_{k,\ell}\}=0$.
Indeed, we will show that $\{\det(B),K_{i,j}\}=0\ \forall (i,j)\in [1,n-1]\times [0,n]$.  For $(i,j)$ satisfying $j\ne 0$, $j\ne n$, $j\ne i$ we observe that $\{K_a,K_{i,j}\}=0$ $\forall a\in[1,n]$.
Note that $\{K_{0,j},\det(B)\}=\left(\frac{1}{2}\cdot\frac{j+1}{n}-\frac{j}{n}+\frac{1}{2}\cdot\frac{j-1}{n}\right)K_{0,j}\det(B)=\left(\frac{j}{2n}+\frac{1}{2n}-\frac{j}{n}+\frac{j}{2n}-\frac{1}{2n}\right)K_{0,j}\det(B)=0$. Similar arguments show that $\det(B)$ Poisson commutes with $K_{i,j}$ in all remaining cases. Hence, the condition $\det(B)=1$ defines a Poisson submanifold $SL_n$ of $GL_n$.
The Poisson bracket on factorization parameters $K_{i,j}$ of matrix $B$ is recorded in the structure of quiver (Fig.~\ref{fi:Casimirs-4}). The structure of quiver directly implies that $\widetilde C_i/ C_i$  is a Casimir for $i=2,\dots,n$. It is well known that the Poisson bracket induced on $SL_n$ coincides
with the standard Poisson-Lie bracket. Moreover, corank of the standard Poisson-Lie bracket is $n-1$, implying that functions $\widetilde C_i/ C_i$ generate the field of the rational Casimir functions.

(iii) Let's recall that the Poisson bracket on the amalgamated network becomes constant in logarithmic coordinates $\log K_{i,j}$ and all functions $\log C_0,\log C_1, \log C_i, \log \widetilde C_i$ are linear in $\log K_{i,j}$, and they are linearly independent. A direct observation of the quiver implies that each of this function is a Casimir, hence  $C_0,C_1,C_i,  {\widetilde C}_i, \ i=2,\dots,n$ form a system of $2n$ independent Casimir functions on the amalgamated network.  
To show the completeness of this system of Casimirs we consider the rank  of the Poisson tensor on the amalgamated network.   The statement (i) claims that corank of the Poisson tensor before amalgamation is $n+1$. In logarthmic coordinates the amalgamation projects the space of network weights onto the new space of dimension $n+1$  lesser with $(n+1)$-dimensional linear fiber.
The corank of the  Poisson tensor is easily computed from the corresponding quiver~\ref{fi:Casimirs-5}, and it equals $2n$.
This fact accomplishes the proof of part (iii) of the Theorem.
\end{proof}

\vskip 2.5cm

\begin{figure}[H]
\begin{center}
\begin{pspicture}(-2,-1.2)(2,1.2){\psset{unit=1}
\newcommand{\ODIN}{%
{\psset{unit=1}
\put(0,0){\pscircle[fillstyle=solid,fillcolor=white,linecolor=white]{.15}}
\put(0,0){\makebox(0,0)[cc]{\hbox{\tcw{\large$\mathbf 1$}}}}
\put(0,0){\makebox(0,0)[cc]{\hbox{\tcr{$\mathbf 1$}}}}
}}
\newcommand{\DVA}{%
{\psset{unit=1}
\put(0,0){\pscircle[fillstyle=solid,fillcolor=white,linecolor=white]{.15}}
\put(0,0){\makebox(0,0)[cc]{\hbox{\tcw{\large$\mathbf 2$}}}}
\put(0,0){\makebox(0,0)[cc]{\hbox{\tcr{$\mathbf 2$}}}}
}}
\newcommand{\TRI}{%
{\psset{unit=1}
\put(0,0){\pscircle[fillstyle=solid,fillcolor=white,linecolor=white]{.15}}
\put(0,0){\makebox(0,0)[cc]{\hbox{\tcw{\large$\mathbf 3$}}}}
\put(0,0){\makebox(0,0)[cc]{\hbox{\tcr{$\mathbf 3$}}}}
}}
\multiput(-3,1.7)(1,0){4}{\psline[linecolor=blue,linewidth=2pt]{->}(0.1,0)(.9,0)}
\multiput(-2.5,0.85)(1,0){3}{\psline[linecolor=blue,linewidth=2pt]{->}(0.1,0)(.9,0)}
\multiput(-2,0)(1,0){2}{\psline[linecolor=blue,linewidth=2pt]{->}(0.1,0)(.9,0)}
\multiput(-1.5,-0.85)(1,0){1}{\psline[linecolor=blue,linewidth=2pt]{->}(0.1,0)(.9,0)}
\multiput(2.5,-0.85)(-1,0){3}{\psline[linecolor=blue,linewidth=2pt]{->}(-0.1,0)(-.9,0)}
\multiput(2,-0)(-1,0){2}{\psline[linecolor=blue,linewidth=2pt]{->}(-0.1,0)(-.9,0)}
\multiput(1.5,0.85)(-1,0){1}{\psline[linecolor=blue,linewidth=2pt]{->}(-0.1,0)(-.9,0)}
\multiput(1.,1.7)(0.5,-0.85){3}{\psline[linecolor=blue,linestyle=dashed,linewidth=2pt]{->}(0.05,-0.085)(.45,-0.765)}
\multiput(.5,.85)(0.5,-0.85){2}{\psline[linecolor=blue,linewidth=2pt]{->}(0.05,-0.085)(.45,-0.765)}
\multiput(0,0)(0.5,-0.85){1}{\psline[linecolor=blue,linewidth=2pt]{->}(0.05,-0.085)(.45,-0.765)}
\multiput(-1.5,-0.85)(-0.5,0.85){3}{\psline[linecolor=blue,linestyle=dashed,linewidth=2pt]{->}(-0.05,0.085)(-.45,0.765)}
\multiput(-0.5,-.85)(-0.5,0.85){3}{\psline[linecolor=blue,linewidth=2pt]{->}(-0.05,0.085)(-.45,0.765)}
\multiput(0,0)(-0.5,0.85){2}{\psline[linecolor=blue,linewidth=2pt]{->}(-0.05,0.085)(-.45,0.765)}
\multiput(0.5,.85)(-0.5,0.85){1}{\psline[linecolor=blue,linewidth=2pt]{->}(-0.05,0.085)(-.45,0.765)}
\multiput(0.5,-0.85)(0.5,0.85){2}{\psline[linecolor=blue,linewidth=2pt]{->}(0.05,0.085)(.45,0.765)}
\multiput(1.5,-0.85)(0.5,0.85){1}{\psline[linecolor=blue,linewidth=2pt]{->}(0.05,0.085)(.45,0.765)}
\multiput(-0.,1.7)(-0.5,-0.85){3}{\psline[linecolor=blue,linewidth=2pt]{->}(-0.05,-0.085)(-.45,-0.765)}
\multiput(-1.,1.7)(-0.5,-0.85){2}{\psline[linecolor=blue,linewidth=2pt]{->}(-0.05,-0.085)(-.45,-0.765)}
\multiput(-2,1.7)(-0.5,-0.85){1}{\psline[linecolor=blue,linewidth=2pt]{->}(-0.05,-0.085)(-.45,-0.765)}
\psbezier[linecolor=blue,linestyle=dashed,linewidth=2pt]{<-}(-2.93,1.77)(1,5.5)(4.5,2)(2.57,-0.78)
\psbezier[linecolor=white,linewidth=3pt]{->}(-1.93,1.77)(2,5.5)(4.5,-1.5)(2.58,-0.87)
\psbezier[linecolor=blue,linewidth=2pt]{->}(-1.93,1.77)(2,5.5)(4.5,-1.5)(2.58,-0.87)
\psbezier[linecolor=white,linewidth=3pt]{->}(-0.93,1.77)(3,5.5)(5.5,-3.5)(1.58,-0.87)
\psbezier[linecolor=blue,linewidth=2pt]{->}(-0.93,1.77)(3,5.5)(5.5,-3.5)(1.58,-0.87)
\psbezier[linecolor=white,linewidth=3pt]{->}(0.07,1.77)(4,5.5)(6.5,-4.5)(0.58,-0.87)
\psbezier[linecolor=blue,linewidth=2pt]{->}(0.07,1.77)(4,5.5)(6.5,-4.5)(0.58,-0.87)
\psbezier[linecolor=blue,linestyle=dashed,linewidth=2pt]{->}(.93,1.77)(-4,6.7)(-5.5,-3.85)(-1.57,-0.92)
\psbezier[linecolor=white,linewidth=3pt]{<-}(-0.08,1.75)(-4,6.7)(-5.5,-3.85)(-0.58,-0.87)
\psbezier[linecolor=blue,linewidth=2pt]{<-}(-0.08,1.75)(-4,6.7)(-5.5,-3.85)(-0.58,-0.87)
\psbezier[linecolor=white,linewidth=3pt]{<-}(-1.1,1.73)(-4,6.7)(-5.5,-3.85)(0.42,-0.87)
\psbezier[linecolor=blue,linewidth=2pt]{<-}(-1.1,1.73)(-4,6.7)(-5.5,-3.85)(0.42,-0.87)
\psbezier[linecolor=white,linewidth=3pt]{<-}(-2.1,1.72)(-4,6.7)(-5.5,-3.85)(1.4,-0.87)
\psbezier[linecolor=blue,linewidth=2pt]{<-}(-2.1,1.72)(-4,6.7)(-5.5,-3.85)(1.4,-0.87)
\multiput(-3.,1.7)(0.5,-0.85){4}{
\multiput(0,0)(1,0){5}{\pscircle[fillstyle=solid,fillcolor=lightgray]{.1}}}
}
\end{pspicture}
\end{center}
\caption{\small
The amalgamated quiver  $[1,n]\times [0,n]$.
}
\label{fi:Casimirs-5}
\end{figure}
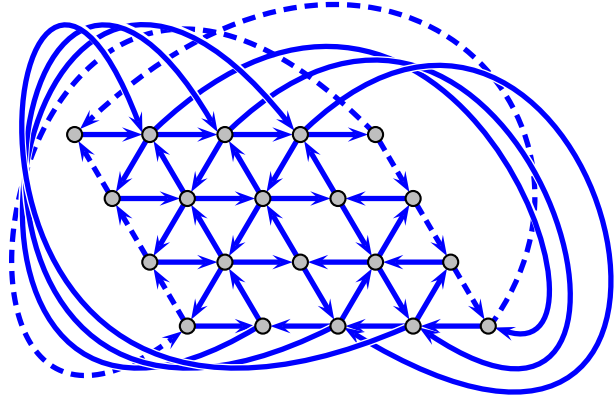

\begin{lemma}\label{uni}
The condition that both $\mathbb A$ and $\widetilde{\mathbb A}:=B\mathbb A B^{\text{T}}$ are unipotent implies that $C_0=C_1=C_i=\widetilde C_i=1$. This implies that all central elements of $B$ equal the unity. 
\end{lemma}

\begin{proof} Direct computation.\end{proof}

\section{The moduli space of framed flag configurations on marked surfaces}\label{s:or}

In this section, we describe an enhanced decorated space of flat connections on non-oriented surfaces. 

Let $\Sigma$ be, generally speaking, a non-orientable two-dimensional topological surface with holes and marked points. We assume that each boundary component of $\Sigma$ has at least one marked point. Let $T$ be a triangulation of $\Sigma$ with vertices at marked points. Each triangle $\Delta$ of $T$ is contractible. Hence we can co-orient all triangles by coloring each triangle's top and bottom faces in green and blue. Triangles glued along their edges form the surface $\Sigma$. We don't assume that gluing is compatible with coloring (i.e., crossing a common edge of two glued triangles, one can go from one color to the other, see ~Fig.~\ref{fig:two-sided surface} ).  Each marked point $p$ has a point on the green face (called green point and denoted by $p^g$) and a corresponding point on the blue face (called blue and denoted by $p^b$). All green marked points are equipped with framed complete flags~\cite{FG1}, i.e., an element of $GL_n/N$, where $N$ is the subgroup of unipotent upper-triangular matrices. By $N_-$ we denote the subgroup of unipotent lower-triangular matrices.
\begin{lemma} A framed flag $x\in GL_n/N$ determines a unique dual framed flag $x^*\in N_-\setminus GL_n$. 
\end{lemma}
\begin{proof} Let $V={\mathbb R}^n$, $\omega_\ell$ be $\ell$-form assigned to the $\ell$-space $x^\ell$, $\omega^*_\ell$ be $\ell$-form assigned to the $\ell$-space $(x^*)^\ell$.
Note that $\omega_k\in \wedge_k V^*$, $(\omega^*)_k\in \wedge^k V$, the standard map $\wedge^k V\oplus \wedge^{n-k} V\to \wedge^n V\simeq {\mathbb R}$ determines the nondegenerate 
pairing $\pi:\wedge^k V\simeq \wedge^{n-k} V^*$.
Define $x^*$ as follows: the space $(x^*)^k= Ann(x^{n-k})$ and the form $(\omega^*)_k=\pi^{-1}(\omega_{n-k})$ restricted on $(x^*)^k$. The restriction of $k$-form is nondegenerate by construction.
\end{proof}
By abusing notation, we define the isomorphism between the framed and dual framed flags by $\pi$.  
A framed configuration assigns to any white vertex $p_i$ a framed flag ${\bf f}(p_i)$ and to the corresponding black vertex $p^b$ the dual framed flag $\pi({\bf f}(p_i))$ such that all flags are pairwise in general position.


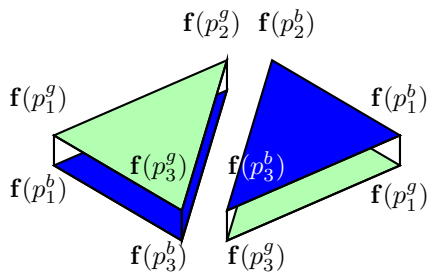
\begin{figure}[H]

\begin{center}
\begin{pspicture}(-1,-1.5)(1,1.5){\psset{unit=2}
\definecolor{lightgreen}{rgb}{0.7, 1, 0.7}

\pspolygon[linecolor=black,fillstyle=solid,fillcolor=blue](-1.,-0.2)(-0.15,-0.7)(0.15,0.3)

\pspolygon[linecolor=black,fillstyle=solid,fillcolor=lightgreen](-1.,0)(-0.15,-0.5)(0.15,0.5)

\pspolygon[linecolor=black,fillcolor=black](-1.,0)(-0.15,-0.5)(-0.15,-0.7)(-1,-0.2)

\pspolygon[linecolor=black,fillcolor=black](-0.15,-0.5)(-0.15,-0.7)(0.15,0.3)(0.15,0.5)

\pspolygon[linecolor=black,fillstyle=solid,fillcolor=lightgreen](1.3,-0.2)(0.15,-0.7)(0.45,0.3)

\pspolygon[linecolor=black,fillstyle=solid,fillcolor=blue](1.3,0)(0.15,-0.5)(0.45,0.5)

\pspolygon[linecolor=black,fillcolor=lightgray](1.3,0)(0.15,-0.5)(0.15,-0.7)(1.3,-0.2)


%
%
 \put(-1.1,0.15){\makebox(0,0)[bc]{\hbox{{${\bf f}(p_1^g)$}}}}
\put(0.05,0.65){\makebox(0,0)[bc]{\hbox{{${\bf f}(p_2^g)$}}}}
\put(-0.3,-0.3){\makebox(0,0)[bc]{\hbox{{${\color{black}{\bf f}(p_3^g)}$}}}}
\put(1.3,0.15){\makebox(0,0)[bc]{\hbox{{${\bf f}(p_1^b)$}}}}
\put(.55,0.65){\makebox(0,0)[bc]{\hbox{{${\bf f}(p_2^b)$}}}}
\put(0.35,-0.3){\makebox(0,0)[bc]{\hbox{{{${\bf f}{\color{white}(p_3^b)}$}}}}}
\put(-1.1,-.45){\makebox(0,0)[bc]{\hbox{{${\bf f}(p_1^b)$}}}}
\put(-0.3,-.9){\makebox(0,0)[bc]{\hbox{{{${\bf f}(p_3^b)$}}}}}
\put(1.3,-0.5){\makebox(0,0)[bc]{\hbox{{${\bf f}(p_1^g)$}}}}
\put(0.35,-0.9){\makebox(0,0)[bc]{\hbox{{${\bf f}(p_3^g)$}}}}
}
\end{pspicture}
%
%
%
%
%
\end{center}
\caption{\small
Two-sided surface: the green faces carry framed flags; the blue faces carry dual framed flags.
}
\label{fig:two-sided surface}
\end{figure}

Let $\vec e=(t,h)$ be an oriented side of triangle $\Delta$ whose endpoints are equipped with the framed flags in general position, $t$ be the tail endpoint of $\vec e$, $h$ be the head of $\vec e$. Let $\phi$ be the forgetful map from the framed flags to unframed flags 
by forgetting corresponding forms, $\phi^*$ be the forgetful map from the dual framed flags to unframed dual flags. The pair 
$\left(\phi({\bf f})(t),{\bf f}(h)\right)$ assigns to the oriented side $\vec e$ the basis ${\bf b}(\vec e)$. The basis is constructed as follows. First of all, flags $\phi({\bf f})(t),\phi({\bf f})(h)$ determine $n$-tuple lines in $V$ with $k$th line $\ell_k$ obtained as intersection of $k$-subspace of the second flag with $n-k+1$-dimensional subspace of the first. Let $\omega^h_k$ be the $k$-form on the $k$-subspace of the flag ${\bf f}(h)$. Choose  the vector $v_k\in\ell_k$ such that 
$\omega^h_k(v_1\wedge\dots v_k)=1$. These conditions consecutively determine $v_1,v_2$, etc uniquely. Gluing two similarly oriented triangles along the side $(t,h)$ we obtain two bases $v_1,\dots,v_n$ and $v'_1,\dots,v'_n$  where $v_k$ is proportional to $v'_{n-k}$. The transition from one basis to another results in an extra multiplication by the antidiagonal matrix $S$ and a diagonal matrix. Gluing two oppositely oriented triangles along side $(t,h)$ we obtain two bases along this side (each in its triangle), with one basis in $V$ (green triangle) and the other in $V^*$. We postulate that gluing corresponds to an operator $V\to V^*$, which transforms the basis in $V$ to the corresponding basis in $V^*$. Therefore, the non-oriented loop leads to a bilinear form.

\section{Expressing $\mathbb A$, $\widetilde{\mathbb A}$}

\subsection{The quiver for $\mathbb A$, $\widetilde{\mathbb A}$}

We now use Lemma~\ref{uni} to construct a special quiver for entries of the matrices $\mathbb A$ and $\widetilde{\mathbb A}$. Starting with expressions (\ref{A}) and (\ref{t-A}), we first observe that due to the transposition operation, (frozen) cluster variables on upper and lower rows are amalgamated in "Moebius-like'' way, with inverting the orientation 
like in Section~\ref{s:or}.
We then declare newly obtained amalgamated variables distinct from the variables at corners of the graph to be new dynamical variables and express all remaining $2n$ frozen variables (including two variables obtained by amalgamating opposite corners of the quiver) via dynamical variables using that all $2n$ Casimir elements must be equal to the unity. Since frozen variables come with power two into these Casimirs, the resulting expressions will contain dynamical cluster variables in powers $0$, $1/2$. and $-1/2$ only. Details of the calculation are the same as in \cite{ChSh20}, so we only present the result. 

\begin{figure}[H]
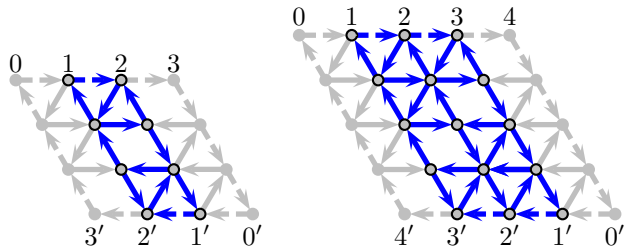

\begin{center}

\end{center}
\caption{\small
The amalgamated quivers for $SL_3$ and $SL_4$. Removed frozen cluster variables are depicted in light gray; the new pairwise amalgamated variables are enumerated ($i$th variable is amalgamated with $i'$th variable).
}
\label{fi:A-q}
\end{figure}

\begin{figure}[H]
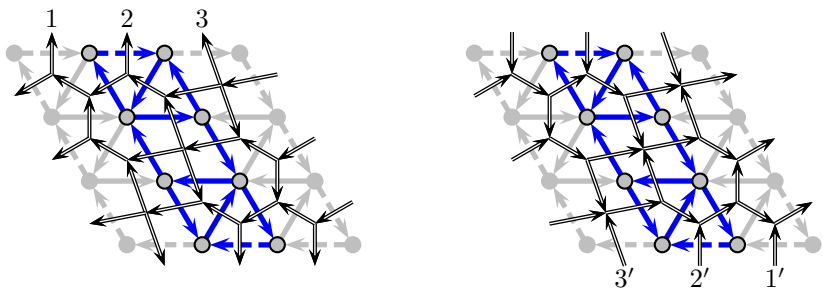

\begin{center}
%
\end{center}
\caption{\small
The directed networks for $SL_3$: on the left is the directed network for $T_1^{-1}S\widetilde T_2$ and on the right is the one for $\widetilde T_1^\text{T} S$.. Removed frozen cluster variables are depicted in light gray; pairwise amalgamated sources and sinks are enumerated ($i$th sink is joined with with $i'$th source).
}
\label{fi:A-q-network}
\end{figure}

\begin{figure}[H]
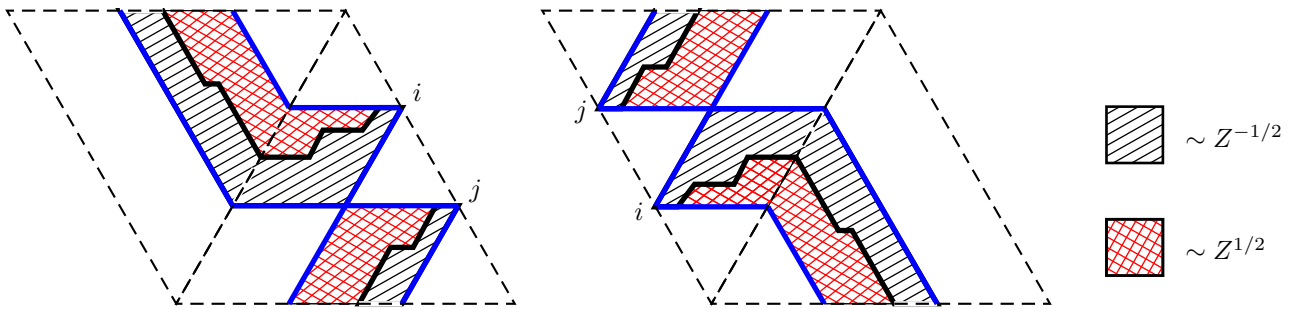

%
\caption{\small
The schematic depiction of calculating a normalized element $a_{i,j}$ of $\mathbb A$ (on the left) and $\widetilde a_{i,j}$ of $\widetilde{\mathbb A}$ (on the right): we have a sum over all admissible paths between $i$th and $j$th sources; variables in the hatched areas come in power $-1/2$; variables in the cross-hatched areas come in power $1/2$.
}
\label{fi:A-scheme}
\end{figure}

The $(i,j)$ entry  of $\mathbb A$ reads $(-1)^{i+j}a_{i,j}$ with positive $a_{i,j}$ given by the sum over all admissible paths from $i$ to $j$ (as shown in the left side of Fig.~\ref{fi:A-scheme}) of products of cluster variables: those in cross-hatched areas come in the power $1/2$ and those in hatched areas come in the power $-1/2$; $a_{i,j}$ are sums of nonnegative terms, which always include two mutually reciprocal terms with either hatched or cross-hatched areas empty, so for any choice of real-valued cluster variables we have that
$$
a_{i,j}>2 \text{ and } \widetilde a_{i,j}>2.
$$ 
This is in line with identifying matrix elements of $\mathbb A$ and $\widetilde{\mathbb A}$ with geodesic functions in the next section.

\begin{example} Calculation of $a_{i,j}$ for $n=4$.

We first redraw the pattern in Figs~\ref{fi:A-q-network} and \ref{fi:A-scheme} by attaching a mirror copy of the right triangular sub-network to the top of the left triangle in Fig.~\ref{fi:A-scheme}. The obtained network is depicted in Fig.~\ref{fi:A-net}; cluster variables correspond to faces (and only inner faces contribute to the normalized entries $a_{i,j}$), variables $s_i$ and $f_i$ are obtained by amalgamations, and variables $\{a,b,c\}$ and $\{p,q,r\}$ are inner variables of the respective right and left triangles. All horizontal double arrows are from right to left and all vertical double arrows are from top to bottom. To simplify the picture, we do not indicate edges of the quiver encoding Poisson relations between cluster variables.

\begin{figure}[H]
\begin{center}
\begin{pspicture}(-5,-0.5)(5,3.5)
{\psset{unit=1}
\psline[linewidth=2pt,linestyle=dashed,linecolor=red](-2.4,-0.3)(-0.6,3.3)
\psline[linewidth=2pt,linestyle=dashed,linecolor=red](1,-0.3)(1,3.3)
\multiput(0,0)(0,1){4}{\psline[doubleline=true,linewidth=1pt, doublesep=1pt, linecolor=black]{<-}(4,0)(4.5,0)}
\multiput(0,0)(0,1){4}{\psline[doubleline=true,linewidth=1pt, doublesep=1pt, linecolor=black]{<-}(-4.5,0)(4.5,0)}
\multiput(-3.5,3)(0.5,-1){3}{\multiput(0,0)(1,0){4}{\psline[doubleline=true,linewidth=1pt, doublesep=1pt, linecolor=black]{->}(0,-0.02)(0,-0.95)}}
\multiput(1.5,1)(1,0){3}{\psline[doubleline=true,linewidth=1pt, doublesep=1pt, linecolor=black]{->}(0,-0.02)(0,-0.95)}
\multiput(2,2)(1,0){2}{\psline[doubleline=true,linewidth=1pt, doublesep=1pt, linecolor=black]{->}(0,-0.02)(0,-0.95)}
\multiput(2.5,3)(1,0){1}{\psline[doubleline=true,linewidth=1pt, doublesep=1pt, linecolor=black]{->}(0,-0.02)(0,-0.95)}
\rput(-2,0.5){\pscircle[linecolor=black,fillstyle=solid, fillcolor=white,linewidth=1pt]{0.3}
\makebox(0,0){\hbox{{$f_3$}}}}
\rput(-1.5,1.5){\pscircle[linecolor=black,fillstyle=solid, fillcolor=white,linewidth=1pt]{0.3}
\makebox(0,0){\hbox{{$f_2$}}}}
\rput(-1,2.5){\pscircle[linecolor=white,fillstyle=solid, fillcolor=white,linewidth=1pt]{0.2}
\makebox(0,0){\hbox{{$f_1$}}}}
\rput(1,0.5){\pscircle[linecolor=black,fillstyle=solid, fillcolor=white,linewidth=1pt]{0.3}
\makebox(0,0){\hbox{{$s_3$}}}}
\rput(1,1.5){\pscircle[linecolor=black,fillstyle=solid, fillcolor=white,linewidth=1pt]{0.3}
\makebox(0,0){\hbox{{$s_2$}}}}
\rput(1,2.5){\pscircle[linecolor=white,fillstyle=solid, fillcolor=white,linewidth=1pt]{0.2}
\makebox(0,0){\hbox{{$s_1$}}}}
\rput(2,0.5){\pscircle[linecolor=black,fillstyle=solid, fillcolor=white,linewidth=1pt]{0.3}
\makebox(0,0){\hbox{{$b$}}}}
\rput(2.5,1.5){\pscircle[linecolor=black,fillstyle=solid, fillcolor=white,linewidth=1pt]{0.3}
\makebox(0,0){\hbox{{$a$}}}}
\rput(3,0.5){\pscircle[linecolor=white,fillstyle=solid, fillcolor=white,linewidth=1pt]{0.3}
\makebox(0,0){\hbox{{$c$}}}}
\rput(-2,2.5){\pscircle[linecolor=white,fillstyle=solid, fillcolor=white,linewidth=1pt]{0.3}
\makebox(0,0){\hbox{{$b$}}}}
\rput(-3,2.5){\pscircle[linecolor=white,fillstyle=solid, fillcolor=white,linewidth=1pt]{0.3}
\makebox(0,0){\hbox{{$a$}}}}
\rput(-2.5,1.5){\pscircle[linecolor=white,fillstyle=solid, fillcolor=white,linewidth=1pt]{0.3}
\makebox(0,0){\hbox{{$c$}}}}
\rput(0,0.5){\pscircle[linecolor=black,fillstyle=solid, fillcolor=white,linewidth=1pt]{0.3}
\makebox(0,0){\hbox{{$r$}}}}
\rput(-1,0.5){\pscircle[linecolor=black,fillstyle=solid, fillcolor=white,linewidth=1pt]{0.3}
\makebox(0,0){\hbox{{$p$}}}}
\rput(-0.5,1.5){\pscircle[linecolor=black,fillstyle=solid, fillcolor=white,linewidth=1pt]{0.3}
\makebox(0,0){\hbox{{$q$}}}}
\rput(4.7,3){\makebox(0,0)[cl]{\hbox{{$1$}}}}
\rput(4.7,2){\pscircle[linecolor=black,fillstyle=solid, fillcolor=white,linewidth=1pt](0.1,0){0.2}
\makebox(0,0)[cl]{\hbox{{$2$}}}}
\rput(4.7,1){\makebox(0,0)[cl]{\hbox{{$3$}}}}
\rput(4.7,0){\makebox(0,0)[cl]{\hbox{{$4$}}}}
\rput(-4.7,3){\makebox(0,0)[cr]{\hbox{{$1'$}}}}
\rput(-4.7,2){\makebox(0,0)[cr]{\hbox{{$2'$}}}}
\rput(-4.7,1){\makebox(0,0)[cr]{\hbox{{$3'$}}}}
\rput(-4.7,0){\pscircle[linecolor=black,fillstyle=solid, fillcolor=white,linewidth=1pt](-0.13,0){0.2}
\makebox(0,0)[cr]{\hbox{{$4'$}}}}
}
\end{pspicture}
\end{center}
\caption{\small
The network for calculating $a_{i,j}$ for $n=4$: it contains two copies of the right triangle (one is a mirror reflection of the other) and a single copy of the left triangle; dashed lines indicate amalgamated boundaries of triangles; cluster variables of faces separated by the dashed lines are obtained by amalgamations of former frozen variables. A matrix element $a_{i,j}$ is obtained by taking a sum over all paths from $i$th source to $j'$th sink; only cluster variables ``inside'' the $i-j'$ rectangle contribute:  in powers $1/2$ for variables of faces above the path and in powers $-1/2$ for variables of faces below the path. In the figure, circled are variables contributing to $a_{2,4}$.}
\label{fi:A-net}
\end{figure}
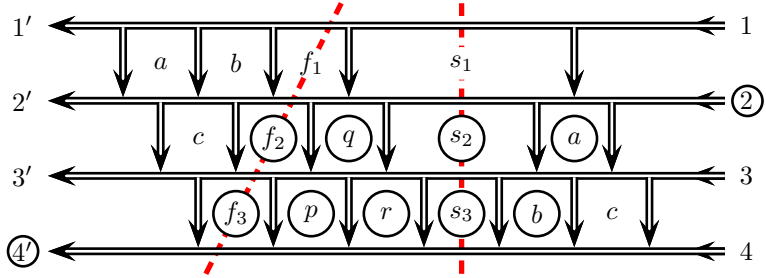

To obtain the normalized element $a_{i,j}$ we take all paths from $i$th source to $j'$th sink; they all lie between the uppermost path and the lowest path; we take cluster variables of faces confined between the uppermost and lowest paths and set into correspondence to every path the product of these variables: those corresponding to faces above the path enter with the power $1/2$ and those corresponding to faces below the path enter with the power $-1/2$. We then take the sum over all paths. 
Diagonal entries $a_{i,i}=1$.
We have
\begin{align*}
&a_{1,2}=(s_1f_1ab)^{1/2}+s_1^{-1/2}(f_1ab)^{1/2}+(s_1f_1)^{-1/2}(ab)^{1/2}+(s_1f_1a)^{-1/2}b^{1/2}+(s_1f_1ab)^{-1/2}\\
&a_{2,3}=(as_2qf_2c)^{1/2}+a^{-1/2}(s_2qf_2c)^{1/2}+(as_2)^{-1/2}(qf_2c)^{1/2}+(as_2q)^{-1/2}(f_2c)^{1/2}+(as_2qf_2)^{-1/2}c^{1/2}+(as_2qf_2c)^{-1/2},
\end{align*}
etc. The longest entry, $a_{2,4}$, comprises 17 terms:
\begin{align*}
&a_{2,4}=(f_2qs_2a)^{1/2}\bigl( (f_3prs_3b)^{1/2}+(f_3prs_3)^{1/2}b^{-1/2}+(f_3pr)^{1/2}(s_3b)^{-1/2}\bigr.\\
&\qquad\qquad\qquad\qquad\bigl. +(f_3p)^{1/2}(rs_3b)^{-1/2}+f_3^{1/2}(prs_3b)^{-1/2}+(f_3prs_3b)^{-1/2}\bigr)\\
&+(f_2qs_2)^{1/2}(ab)^{-1/2}\bigl( (f_3prs_3)^{1/2}+(f_3pr)^{1/2}s_3^{-1/2}+(f_3p)^{1/2}(rs_3)^{-1/2}+f_3^{1/2}(prs_3)^{-1/2}+(f_3prs_3)^{-1/2}\bigr)\\
&+(f_2q)^{1/2}(s_2ars_3b)^{-1/2}\bigl( (f_3p)^{1/2}+f_3^{1/2}p^{-1/2}+(f_3p)^{-1/2}\bigr)\\
&+f_2^{1/2}(qs_2aprs_3b)^{-1/2}( f_3^{1/2}+f_3^{-1/2}) + (f_2qs_2af_3prs_3b )^{-1/2}
\end{align*}

\end{example}

\section{Groupoid of triangular forms and Teichm\"uller spaces of Riemann surfaces}

Groupoid structures are closely related to Teichm\"uller spaces $\mathcal T_{g,s}$ of genus $g$ Riemann surfaces with $s=1,2$ holes.

It is well-known (see \cite{ChF3}) that we can identify entries $a_{i,j}$ of $\mathbb A\in SL_n$  with \emph{geodesic functions}, $a_{i,j}=G_{i,j}$, where $G_{i,j}=e^{\ell_{i,j}/2}+e^{-\ell_{i,j}/2}$, for a special subset of geodesics loops $\gamma_{i,j}$ having lengths $\ell_{i,j}$ on a Poincar\'e uniformized Riemann surface of genus $g=\lfloor (n-1)/2 \rfloor$ with one hole ($s=1$) for odd $n$ and with two holes ($s=2$) for even $n$. Then the braid-group transformations $\mathbb A\mapsto B_{i,i+1}\mathbb A B^{\text{T}}_{i,i+1}$ can be interpreted as Dehn twists along geodesics $\gamma_{i,i+1}$ acting on the corresponding geodesic functions $G_{k,l}$. 

A known subtlety is that dimensions of $\mathcal A_n$ and of the related Teichm\"uller spaces $\mathcal T_{g,s}$ coincide only for $n=3$ and $n=4$; starting with $n=5$, $\mathcal T_{g,s}$ are embedded as non-maximal dimension Poisson leaves into  $\mathcal A_n$. So, we prefer to keep the notation $G_{i,j}$ for geodesic functions, although they coincide with matrix entries $a_{i,j}$ for real positive cluster variables only for $n=3,4$.

Geodesic functions enjoy the classical skein relations and semiclassical Poisson relations (Goldman brackets); using these two relations we can construct any other geodesic function for a smooth Riemann surface of genus $g>1$ out of $2g+1$ specially chosen geodesics.\footnote{up to a $\mathbb Z_2$-symmetry subtlety discussed below in Sec.~\ref{ss:Z2}.} Say, having $n-1$ geodesic functions $G_{i,i+1}$, $i=1,\dots,n-1$, we can construct all $G_{i,j}$: since $\gamma_{i,i+1}$ and $\gamma_{i+1,i+2}$ has a single intersection point, $G_{i,i+1}G_{i+1,i+2}=G_{i,i+2}+\widehat G_{i,i+2}$, where $\widehat G_{i,i+2}$ is the geodesic function not in our list, but $\{G_{i,i+1},G_{i+1,i+2}\}=\frac12 (G_{i,i+2}-\widehat G_{i,i+2})$, and therefore
$$
G_{i,i+2}=\frac12 G_{i,i+1}G_{i+1,i+2} +  \{G_{i,i+1},G_{i+1,i+2}\}.
$$
We can continue this process to construct all $G_{i,j}$.

It was proved in \cite{NR93} that $\mathbb{A}$ of the geometric leaf satisfies the condition $\operatorname{rank}(\mathbb{A}+\mathbb{A}^{\text{T}})\le 4$.
Geodesic functions of loops around holes of the surface form a complete set of Casimir functions of the Goldman Poisson bracket on $\mathcal T_{g,s}$. A complete set of Casimirs of the reflection Poisson bracket is given by the coefficients of powers of $\lambda$ of polynomial $\det(\mathbb{A}+\lambda\mathbb{A}^{\text{T}})$, or, equivalently, by the eigenvalues of $\mathbb{A}^\text{-T}\mathbb{A}$. By Theorem 4.4  \cite{ChM2}, for $\mathbb{A}$ in the geometric leaf the spectrum of $\mathbb{A}^\text{-T}\mathbb{A}$ consists of $\lambda,\lambda^{-1}, \mu,\mu^{-1}$ and $-1$'s for even $n$ and $\lambda,\lambda^{-1}$ and $-1$s for odd $n$. Hence there are maximum two independent eigenvalues for even $n$ and one independent eigenvalue for odd $n$. We can conclude that in the geometric case the algebra of Casimir function is generated by two generators for even $n$ and by one generator for odd $n$. The generators are described by the geodesic functions corresponding of two holes (for even $n$) and one  hole (for odd $n$) of the hyperbolic surface.

\section{Cluster variable description of closed genus-two Riemann surface.}

We now exploit the correspondence between entries of $\mathbb A$ and $\widetilde{\mathbb A}$ satisfying reflection equations (\ref{A-A}), (\ref{At-At}) and the commutation relation (\ref{A-At}) and geodesic functions of geodesics on Riemann surfaces satisfying $SL_2$ Goldman brackets and skein relations.  

\emph{Cluster mutations} $\mu_f$ are operations on directed graphs (quivers) and variables associated with vertices of a quiver; mutations transform a quiver in a standard way, whereas  variables are transformed as follows: for four typical cases in this paper,
$$
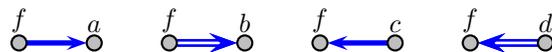
\begin{pspicture}(-1,-0.3)(1,0.7){\psset{unit=1}
\psline[linecolor=blue,linewidth=2pt]{->}(-0.4,0)(0.4,0)
%
\multiput(-0.5,0)(1,0){2}{\pscircle[fillstyle=solid,fillcolor=lightgray]{.1}}
\put(-0.5,0.15){\makebox(0,0)[bc]{\hbox{{$f$}}}}
\put(0.5,0.15){\makebox(0,0)[bc]{\hbox{{$a$}}}}
}
\end{pspicture}
\begin{pspicture}(-1,-0.3)(1,0.7){\psset{unit=1}
%
\psline[doubleline=true,linewidth=1pt, doublesep=1pt, linecolor=blue]{->}(-0.4,0)(0.4,0)
\multiput(-0.5,0)(1,0){2}{\pscircle[fillstyle=solid,fillcolor=lightgray]{.1}}
\put(-0.5,0.15){\makebox(0,0)[bc]{\hbox{{$f$}}}}
\put(0.5,0.15){\makebox(0,0)[bc]{\hbox{{$b$}}}}
}
\end{pspicture}
\begin{pspicture}(-1,-0.3)(1,0.7){\psset{unit=1}
\psline[linecolor=blue,linewidth=2pt]{<-}(-0.4,0)(0.4,0)
%
\multiput(-0.5,0)(1,0){2}{\pscircle[fillstyle=solid,fillcolor=lightgray]{.1}}
\put(-0.5,0.15){\makebox(0,0)[bc]{\hbox{{$f$}}}}
\put(0.5,0.15){\makebox(0,0)[bc]{\hbox{{$c$}}}}
}
\end{pspicture}
\begin{pspicture}(-1,-0.3)(1,0.7){\psset{unit=1}
%
\psline[doubleline=true,linewidth=1pt, doublesep=1pt, linecolor=blue]{<-}(-0.4,0)(0.4,0)
\multiput(-0.5,0)(1,0){2}{\pscircle[fillstyle=solid,fillcolor=lightgray]{.1}}
\put(-0.5,0.15){\makebox(0,0)[bc]{\hbox{{$f$}}}}
\put(0.5,0.15){\makebox(0,0)[bc]{\hbox{{$d$}}}}
}
\end{pspicture}
$$
we have that the mutation $\mu_f$ in the vertex $f$ results in transformations
$$
\mu_f\{f,a,b,c,d\}=\left\{f^{-1}, \ a(1+f^{-1})^{-1}, \ b(1+f^{-1})^{-2},\ c(1+f),\ d(1+f)^2  \right\}.
$$
For arrows of higher orders, the rule is clear.

Various convenient ways to represent the quiver for $SL_3$ are below:

$$
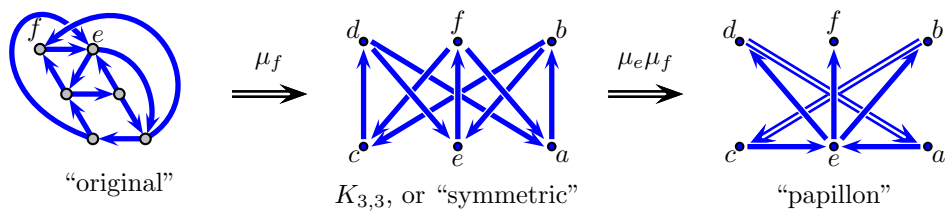
\begin{pspicture}(-1,-1.5)(1,1.3){\psset{unit=0.7}
\rput(-1.5,0.85){\psline[linecolor=blue,linewidth=2pt]{->}(0.1,0)(.9,0)}
\rput(-1,0){\psline[linecolor=blue,linewidth=2pt]{->}(0.1,0)(.9,0)}
\rput(0.5,-0.85){\psline[linecolor=blue,linewidth=2pt]{->}(-0.1,0)(-.9,0)}
\multiput(0,0)(0.5,-0.85){1}{\psline[linecolor=blue,linewidth=2pt]{->}(0.05,-0.085)(.45,-0.765)}
\multiput(-0.5,-.85)(-0.5,0.85){2}{\psline[linecolor=blue,linewidth=2pt]{->}(-0.05,0.085)(-.45,0.765)}
\multiput(0,0)(-0.5,0.85){1}{\psline[linecolor=blue,linewidth=2pt]{->}(-0.05,0.085)(-.45,0.765)}
\psbezier[linecolor=blue,linewidth=2pt]{<-}(-0.57,0.92)(-2,2.5)(-3,0)(-0.5,-0.85)
\psbezier[linecolor=blue,linewidth=2pt]{->}(-0.5,0.85)(0.5,0.85)(0.7,-0.2)(0.53,-0.76)
\psbezier[linecolor=white,linewidth=3pt]{<-}(-1.43,0.9)(0.5,2.5)(2,0)(0.5,-0.85)
\psbezier[linecolor=blue,linewidth=2pt]{<-}(-1.43,0.9)(0.5,2.5)(2,0)(0.5,-0.85)
%
\multiput(-0.5,0.85)(-0.5,-0.85){1}{\psline[linecolor=blue,linewidth=2pt]{->}(-0.05,-0.085)(-.45,-0.765)}
\multiput(-2.5,0.85)(0.5,-0.85){3}{
\multiput(1,0)(1,0){2}{\pscircle[fillstyle=solid,fillcolor=lightgray]{.1}}}
\put(-1.6,1){\makebox(0,0)[bc]{\hbox{{$f$}}}}
\put(-0.4,1){\makebox(0,0)[bc]{\hbox{{$e$}}}}
\rput(0,-1.5){\makebox(0,0)[tc]{\hbox{{``original''}}}}
}
\end{pspicture}
\begin{pspicture}(-1,-1.5)(1,1.3)
\rput(0,0.3){\makebox(0,0)[bc]{\hbox{{$\mu_f$}}}}
\psline[doubleline=true,linewidth=1pt, doublesep=1pt, linecolor=black]{->}(-0.5,0)(0.5,0)
\end{pspicture}
\begin{pspicture}(-1.5,-1.5)(1.5,1.3)
{
\multiput(-1.25,0.7)(1.25,0){3}{\pscircle[fillstyle=solid,fillcolor=blue]{.05}}
\multiput(-1.25,-0.7)(1.25,0){3}{\pscircle[fillstyle=solid,fillcolor=blue]{.05}}
\put(-1.3,-0.75){\makebox(0,0)[tr]{\hbox{{$c$}}}}
\put(-1.3,0.75){\makebox(0,0)[br]{\hbox{{$d$}}}}
\put(1.3,-0.75){\makebox(0,0)[tl]{\hbox{{$a$}}}}
\put(1.3,0.75){\makebox(0,0)[bl]{\hbox{{$b$}}}}
\put(0,0.8){\makebox(0,0)[bc]{\hbox{{$f$}}}}
\put(0,-0.8){\makebox(0,0)[tc]{\hbox{{$e$}}}}
\psline[linewidth=3pt, linecolor=white]{->}(1.15,0.7)(-1.15,-0.7)
\psline[linewidth=2pt, linecolor=blue]{->}(1.15,0.7)(-1.15,-0.7)
\psline[linewidth=3pt, linecolor=white]{->}(-1.15,0.7)(1.15,-0.7)
\psline[linewidth=2pt, linecolor=blue]{->}(-1.15,0.7)(1.15,-0.7)
\multiput(0,0)(-1.25,0){2}{\psline[linewidth=3pt, linecolor=white]{->}(1.13,0.6)(0.1,-0.6)
\psline[linewidth=2pt, linecolor=blue]{->}(1.13,0.6)(0.1,-0.6)}
\multiput(0,0)(1.25,0){2}{\psline[linewidth=3pt, linecolor=white]{->}(-1.13,0.6)(-0.1,-0.6)
\psline[linewidth=2pt, linecolor=blue]{->}(-1.13,0.6)(-0.1,-0.6)}
\multiput(-1.25,0)(1.25,0){3}{\psline[linewidth=3pt, linecolor=white]{->}(0,-0.6)(0,0.6)
\psline[linewidth=2pt, linecolor=blue]{->}(0,-0.6)(0,0.6)}
\rput(0,-1.2){\makebox(0,0)[tc]{\hbox{{$K_{3,3}$, or ``symmetric''}}}}
}
\end{pspicture}
\begin{pspicture}(-1,-1.5)(1,1.3)
\rput(0,0.3){\makebox(0,0)[bc]{\hbox{{$\mu_e\mu_f$}}}}
\psline[doubleline=true,linewidth=1pt, doublesep=1pt, linecolor=black]{->}(-0.5,0)(0.5,0)
\end{pspicture}
\begin{pspicture}(-1.5,-1.5)(1.5,1.3)
{
\multiput(-1.25,0.7)(1.25,0){3}{\pscircle[fillstyle=solid,fillcolor=blue]{.05}}
\multiput(-1.25,-0.7)(1.25,0){3}{\pscircle[fillstyle=solid,fillcolor=blue]{.05}}
\rput(-1.3,-0.75){\makebox(0,0)[tr]{\hbox{{$c$}}}}
\rput(-1.3,0.75){\makebox(0,0)[br]{\hbox{{$d$}}}}
\rput(1.3,-0.75){\makebox(0,0)[tl]{\hbox{{$a$}}}}
\rput(1.3,0.75){\makebox(0,0)[bl]{\hbox{{$b$}}}}
\rput(0,0.8){\makebox(0,0)[bc]{\hbox{{$f$}}}}
\rput(0,-0.8){\makebox(0,0)[tc]{\hbox{{$e$}}}}
\psline[linewidth=3pt, linecolor=white]{->}(1.15,0.7)(-1.15,-0.7)
\psline[doubleline=true,linewidth=1pt, doublesep=1pt, linecolor=blue]{->}(1.15,0.7)(-1.15,-0.7)
\psline[linewidth=3pt, linecolor=white]{->}(-1.15,0.7)(1.15,-0.7)
\psline[doubleline=true,linewidth=1pt, doublesep=1pt, linecolor=blue]{->}(-1.15,0.7)(1.15,-0.7)
\psline[linewidth=3pt, linecolor=white]{<-}(1.13,0.6)(0.1,-0.6)
\psline[linewidth=2pt, linecolor=blue]{<-}(1.13,0.6)(0.1,-0.6)
\psline[linewidth=3pt, linecolor=white]{<-}(-1.13,0.6)(-0.1,-0.6)
\psline[linewidth=2pt, linecolor=blue]{<-}(-1.13,0.6)(-0.1,-0.6)
\psline[linewidth=3pt, linecolor=white]{->}(-1.15,-0.7)(-0.1,-0.7)
\psline[linewidth=2pt, linecolor=blue]{->}(-1.15,-0.7)(-0.1,-0.7)
\psline[linewidth=3pt, linecolor=white]{->}(1.15,-0.7)(0.1,-0.7)
\psline[linewidth=2pt, linecolor=blue]{->}(1.15,-0.7)(0.1,-0.7)
\psline[linewidth=3pt, linecolor=white]{->}(0,-0.6)(0,0.6)
\psline[linewidth=2pt, linecolor=blue]{->}(0,-0.6)(0,0.6)
\rput(0,-1.2){\makebox(0,0)[tc]{\hbox{{``papillon''}}}}
}
\end{pspicture}
$$
Here the rightmost quiver is known as $X_6$ and it is one of a handful of mutation-finite nongeometrical quivers. (Recall that geometrical quivers are those dual to an ideal-triangle decompositions of (open) Riemann surfaces $\Sigma_{g,s}$ with the number of holes $s>0$ and $2g+s-2>0$.)

Let us introduce special ordered combinations (``telescopic sums'') of letters $a,b,c,d,e,f$ denoted by $\langle \dots \rangle$: for example
$$
\langle defa \rangle:=(defa)^{1/2}\left(1+\frac1d+\frac1{de}+\frac1{def}+\frac1{defa}\right).
$$
Then in the middle quiver ($K_{3,3}$) we have 
\begin{align*}
&G_{1,2}=\langle defa \rangle,\quad G_{1,3}=\langle bcde \rangle,\quad G_{2,3}=\langle fabc \rangle \\
&\widetilde G_{1,2}=\langle befc \rangle,\quad \widetilde G_{1,3}=\langle dabc \rangle,\quad \widetilde G_{2,3}=\langle fcda \rangle .
\end{align*}
After mutations: first in $f$, then in $e$, we obtain the right quiver (``papillon'').
It is convenient to introduce the notation
$$
\hat e:= e(abcdf)^{1/2}.
$$
Then the only variable with which $\hat e$ does not commute is $f$; the rule is that $\{ \hat e, f\}=2\hat e f$. 

After the above mutation, the geodesic functions in question become
\begin{align}
&G_{1,2}=(da)^{1/2}\left(1+\frac 1d+\frac1{da}\right)=\langle da\rangle \nonumber\\
&G_{1,3}=a^{-1/2}\hat e +(df)^{1/2}\widetilde G_{1,2}+\frac{a^{1/2}}{\hat e}(1+f)(1+d)\label{GA}\\
&G_{2,3}=d^{-1/2}\hat e \left( 1+\frac 1a \right)+(f/a)^{1/2} \widetilde G_{1,2}+\frac{d^{1/2}}{\hat e}(1+f)\nonumber
\end{align}
and
\begin{align}
&\widetilde G_{1,2}=(bc)^{1/2}\left(1+\frac 1b+\frac1{bc}\right)=\langle bc\rangle\nonumber\\
&\widetilde G_{1,3}=c^{-1/2}\hat e +(bf)^{1/2} G_{1,2}+\frac{c^{1/2}}{\hat e}(1+f)(1+b)\label{GtA}\\
&\widetilde G_{2,3}=b^{-1/2}\hat e \left( 1+\frac 1c \right)+(f/c)^{1/2} G_{1,2}+\frac{b^{1/2}}{\hat e}(1+f),\nonumber
\end{align}

Our novel interpretation of the $X_6$ quiver is that it parameterizes the Teichm\"uller space $\mathcal T_{2,0}$ of \emph{smooth} Riemann surfaces of genus two with one distinguished separating geodesics (the Markov element); we checked that cluster mutations correspond to Dehn twists along $G_{i,j}$ or $\widetilde G_{i,j}$ and leave the Markov element invariant.

Our next goal is to construct the full modular group of $\Sigma_{2,0}$ out of the elements of $X_6$. This task is seemingly possible because the total Poisson dimension of the $\{a,b,c,d,e,f\}$ is six and we therefore have elements not commuting with the Markov element $\mathfrak M$. We need to find the geodesic function $G_B$ (see Fig.~\ref{fi:genus-2}) such that the  \emph{twist} $\tau_B$ that is dual to $\ell_{\mathfrak M}$ and is given by the expression (see \cite{NRS}, \cite{Ch20})
\be\label{twist}
e^{\tau_B}+e^{-\tau_B}=\frac {\mathfrak M G_B-2G_{1,2}\widetilde G_{1,2}}{\sqrt{\bigl(\mathfrak M+G^2_{1,2}\bigr)\bigl(\mathfrak M+\widetilde G^2_{1,2}\bigr)}},
\ee
has the proper Poisson relations.  Here $\mathfrak M=e^{\ell_{{\mathfrak M}}/2}+e^{-\ell_{\mathfrak M}/2}-2$, where we identify $\ell_{\mathfrak M}$ with the hyperbolic length of the separating geodesics that commutes with all $G_{i,j}$ and $\widetilde G_{i,j}$ (but, of course, NOT with $G_B$)

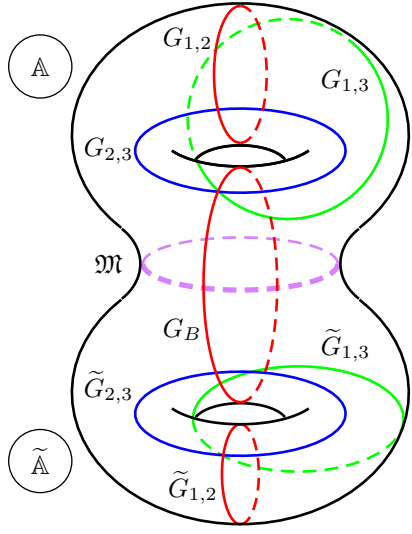
\begin{figure}[H]
\begin{center}
\begin{pspicture}(-1.5,-5.5)(1.5,1)
{\psset{unit=0.7}
\newcommand{\ELLARC}[7]{
\parametricplot[linecolor=#1, linewidth=#2 pt, linestyle=#3]{#4}{#5}{#6 t cos mul #7 t sin mul}
}
\definecolor{lightblue}{rgb}{.85, .5, 1}
\newcommand{\PAT}{%
\rput(-1,-0.3){\ELLARC{black}{1}{solid}{-45}{225}{3.2}{2.5}}
\rput(-1,-0.9){\ELLARC{black}{1}{solid}{20}{160}{.9}{.4}}
\rput(-1,-0.3){\ELLARC{black}{1}{solid}{210}{330}{1.5}{.6}}
\rput(-4.8,1){\makebox(0,0){$\mathbb A$}}
\pscircle[linewidth=0.5pt](-4.8,1){.6}
\rput(-4.8,-6.5){\makebox(0,0){$\widetilde{\mathbb A}$}}
\pscircle[linewidth=0.5pt](-4.8,-6.5){.6}
\rput(-1,-2.75){\ELLARC{lightblue}{1}{dashed}{0}{180}{1.85}{.5}}
\rput(-1,-2.75){\ELLARC{lightblue}{2}{dashed}{180}{360}{1.85}{.5}}
\rput(-4.1,-2.75){\ELLARC{black}{1}{solid}{-45}{45}{1.2}{.94}}
\rput(2.1,-2.75){\ELLARC{black}{1}{solid}{135}{225}{1.2}{.94}}
\rput(-1,-5.2){\ELLARC{black}{1}{solid}{-225}{45}{3.2}{2.5}}
\rput(-1,-0.9){\ELLARC{black}{1}{solid}{20}{160}{.9}{.4}}
\rput(-1,-0.3){\ELLARC{black}{1}{solid}{210}{330}{1.5}{.6}}
\rput(-1,-5.8){\ELLARC{black}{1}{solid}{20}{160}{.9}{.4}}
\rput(-1,-5.2){\ELLARC{black}{1}{solid}{210}{330}{1.5}{.6}}
}
\rput(4,0){
\rput(-0.1,0){\ELLARC{green}{1}{dashed}{45}{210}{1.9}{1.9}}
\rput(-0.1,0){\ELLARC{green}{1}{solid}{-150}{45}{1.9}{1.9}}
\rput(-1,-0.6){\ELLARC{blue}{1}{solid}{0}{360}{2}{0.8}}
\rput{-90}(-1,.85){\ELLARC{red}{1}{solid}{180}{360}{1.3}{0.5}}
\rput{-90}(-1,.85){\ELLARC{red}{1}{dashed}{0}{180}{1.3}{0.5}}
\rput(0.1,-5.7){\ELLARC{green}{1}{dashed}{180}{360}{2}{1}}
\rput(0.1,-5.7){\ELLARC{green}{1}{solid}{0}{180}{2}{1}}
\rput(-1,-5.6){\ELLARC{blue}{1}{solid}{0}{360}{2}{0.8}}
\rput{-90}(-1,-6.75){\ELLARC{red}{1}{solid}{180}{360}{0.95}{0.35}}
\rput{-90}(-1,-6.75){\ELLARC{red}{1}{dashed}{0}{180}{0.95}{0.35}}
\rput(0,0){\PAT}
\rput(-2,1.5){\makebox(0,0){$G_{1,2}$}}
\rput(-3.5,-.6){\makebox(0,0){$G_{2,3}$}}
\rput(1,.7){\makebox(0,0){$G_{1,3}$}}
\rput(-3.5,-2.75){\makebox(0,0){$\mathfrak M$}}
\rput(-1.9,-7){\makebox(0,0){$\widetilde G_{1,2}$}}
\rput(-3.5,-5.1){\makebox(0,0){$\widetilde G_{2,3}$}}
\rput(1,-4.3){\makebox(0,0){$\widetilde G_{1,3}$}}
\rput(-2.1,-4){\makebox(0,0){$G_{B}$}}
\rput(-1,-3.15){\ELLARC{red}{1}{solid}{90}{270}{0.7}{2.23}}
\rput(-1,-3.15){\ELLARC{red}{1}{dashed}{-90}{90}{0.7}{2.23}}
}
%
%
}
\end{pspicture}
\end{center}
\caption{\small
A smooth Riemann surface of genus two with sets of geodesics (labeled by the corresponding geodesic functions): the separating geodesic is the Markov element $\mathfrak M$; geodesics corresponding to $\mathbb A$ are located in the upper part (a torus with a hole, represented by the Markov element), geodesics corresponding to $\widetilde{\mathbb A}$ are located in the lower part; these two sets of geodesic functions obviously commute. A missed geodesic that we need to construct the twist variable dual to the Markov element $\mathfrak M$ is $G_B$.
}
\label{fi:genus-2}
\end{figure}

To find a candidate for $G_B$ we explore different expressions for the Markov element $\mathfrak M$. Two original expressions are 
\begin{align}
\mathfrak M&= G_{1,2}G_{1,3}G_{2,3}-G^2_{1,2}-G^2_{1,3}-G^2_{2,3}\label{M3-1}\\
&= \widetilde G_{1,2}\widetilde G_{1,3}\widetilde G_{2,3}-\widetilde G^2_{1,2}-\widetilde G^2_{1,3}-\widetilde G^2_{2,3}.\label{M3-2}
\end{align}
When we substitute expressions from (\ref{GA}) and (\ref{GtA}) into this expression collecting wherever possible terms constituting $G_{1,2}$ and $\widetilde G_{1,2}$, we find that one combination plays a prominent role; this combination is
\be\label{GB}
G_B:=\frac1{f^{1/2}}\left(\hat e+\frac{1+f}{\hat e}\right).
\ee
Then the \emph{Markov element} $\mathfrak M$ can be written as
\be\label{A3}
\mathfrak M= f\bigl( G_{1,2}\widetilde G_{1,2}G_{B} + G^2_{1,2} + \widetilde G^2_{1,2} + G^2_{B}-4\bigr)-4,
\ee
and since $G_B$ commutes with $G_{1,2}$ and $\widetilde G_{1,2}$, it is only the presence of the variable $f$ in $\mathfrak M$ that makes commutation relations between $\mathfrak M$ and $G_B$ nonzero. 
Substituting ~(\ref{GB}) in (\ref{twist}), let us evaluate
\begin{align*}
&\{e^{\tau_B}+e^{-\tau_B},e^{\ell_A/2}+e^{-\ell_A/2}\}=\frac{\mathfrak M}{\sqrt{(\mathfrak M+G_{1,2})(\mathfrak M+\widetilde G_{1,2})}}\{G_B,\mathfrak M\}\\
&=\frac{\mathfrak M}{\sqrt{(\mathfrak M+G_{1,2})(\mathfrak M+\widetilde G_{1,2})}}\left[e(abcd)^{1/2}-\frac 1{fe(abcd)^{1/2}}-\frac1{e(abcd)^{1/2}}\right]f\cdot \bigl( G_{1,2}\widetilde G_{1,2}G_{B} + G^2_{1,2} + \widetilde G^2_{1,2} + G^2_{B}-4\bigr) \\
&=\sqrt{G_B^2-4-4/f}\frac{(\mathfrak M+4)\mathfrak M}{\sqrt{(\mathfrak M+G_{1,2})(\mathfrak M+\widetilde G_{1,2})}}\\
&=\sqrt{G_B^2-4-4\frac{G_{1,2}\widetilde G_{1,2}G_{B} + G^2_{1,2} + \widetilde G^2_{1,2} + G^2_{B}-4}{\mathfrak M+4}}\frac{(\mathfrak M+4)\mathfrak M}{\sqrt{(\mathfrak M+G_{1,2})(\mathfrak M+\widetilde G_{1,2})}}\\
&=\frac{\sqrt{\mathfrak M(\mathfrak M+4)}}{\sqrt{(\mathfrak M+G_{1,2})(\mathfrak M+\widetilde G_{1,2})}}{\sqrt{(\mathfrak MG_B-2G_{1,2}\widetilde G_{1,2})^2-4(\mathfrak M+G_{1,2})(\mathfrak M+\widetilde G_{1,2})}}\\
&=\sqrt{(\mathfrak M+2)^2-4}\sqrt{\left[ \frac{\mathfrak MG_B-2G_{1,2}\widetilde G_{1,2}}{\sqrt{(\mathfrak M+G_{1,2})(\mathfrak M+\widetilde G_{1,2})}} \right]^2-4}\\
&=\sqrt{(e^{\ell_{\mathfrak M}/2}+e^{-\ell_{\mathfrak M}/2})^2-4}\sqrt{(e^{\tau_B}+e^{-\tau_B})^2-4}=(e^{\ell_{\mathfrak M}/2}-e^{-\ell_{\mathfrak M}/2})(e^{\tau_B}-e^{-\tau_B}),
\end{align*}
which implies $\{\tau_B,\ell_{\mathfrak M}/2 \}=1$.

A longer but still straightforward calculation, which we omit, demonstrates that the twists $\tau_{1,2}$ and $\widetilde \tau_{1,2}$ along the geodesics $\gamma_{1,2}$ and $\widetilde\gamma_{1,2}$ determined by the formulas
\be\label{twist-lezginka}
(e^{\tau_{1,2}}+e^{-\tau_{1,2}})^2=-4+\frac{G_{2,3}^2(G_{1,2}^2-4)}{\mathfrak M+G_{1,2}^2},\qquad
(e^{\widetilde\tau_{1,2}}+e^{-\widetilde\tau_{1,2}})^2=-4+\frac{\widetilde G_{2,3}^2(\widetilde G_{1,2}^2-4)}{\mathfrak M+\widetilde G_{1,2}^2}
\ee
commute with $\tau_B$. The quantity $\tau_B$ satisfies therefore all determining properties of a twist variable dual to $\ell_{\mathfrak M}$.

The next and (almost) final step of the construction is to find cluster variable  transformations that correspond to twists along $G_B$. Note that these transformations are not mutations of $X_6$ since the latter preserve the Markov element. All elements $G_{i,j}$ and $\widetilde G_{i,j}$ must nevertheless remain Laurent polynomials upon these transformations.

The new (``extended'') mutations that correspond to the twist along $G_B$ in the $X_6$ quiver are
\be\label{mut-ex}
e\to e\left(1+\frac 1f\right)^{-1},\quad f^{1/2}\to \frac{f^{1/2}+f^{-1/2}}{e(abcd)^{1/2}}.
\ee
These transformations preserve commutation relations in the ``papillon'' quiver, so the quiver is not changed.


Note that the transformation of $f$ in~(\ref{mut-ex}) is not a standard cluster mutation. However, the following extension of quiver $X_6$ to quiver $X_7$ of finite mutation type ~\cite{FST12} provides a framework for cluster interpretation of this "extended" mutation. Namely, the transformation~(\ref{mut-ex}) is a standard cluster mutation in $X_7$ under condition that seven cluster coordinates in $X_7$ satisfy relation (\ref{Cas}).

{\bf Extended quiver $X_7$.} Consider a ``papillon'' quiver with a new vertex added:
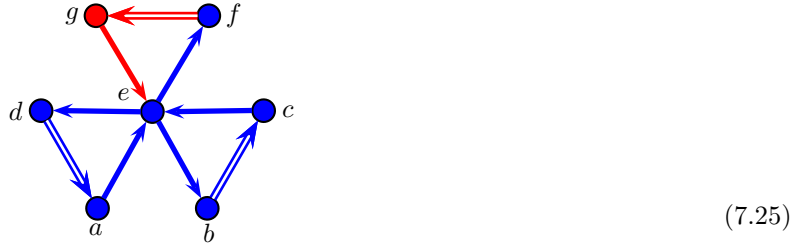
\be\label{X7}
\begin{pspicture}(-1.5,-1.5)(1.5,1.5)
{\psset{unit=1.5}
\newcommand{\PAT}{%
\rput(-0.5,0.85){\pscircle[fillstyle=solid,fillcolor=blue]{.1}}
\rput(0.5,0.85){\pscircle[fillstyle=solid,fillcolor=blue]{.1}}
\psline[linewidth=2pt, linecolor=blue]{<-}(-0.05,0.09)(-0.45,0.76)
\psline[doubleline=true,linewidth=1pt, doublesep=1pt, linecolor=blue]{->}(0.4,0.85)(-0.4,0.85)
\psline[linewidth=2pt, linecolor=blue]{->}(0.05,0.09)(0.45,0.76)
}
\rput(0,0){\pscircle[fillstyle=solid,fillcolor=blue]{.1}}
\rput(-0.5,0.85){\pscircle[fillstyle=solid,fillcolor=red]{.1}}
\rput(0.5,0.85){\pscircle[fillstyle=solid,fillcolor=blue]{.1}}
\psline[linewidth=2pt, linecolor=red]{<-}(-0.05,0.09)(-0.45,0.76)
\psline[doubleline=true,linewidth=1pt, doublesep=1pt, linecolor=red]{->}(0.4,0.85)(-0.4,0.85)
\psline[linewidth=2pt, linecolor=blue]{->}(0.05,0.09)(0.45,0.76)
\rput{120}(0,0){\PAT}
\rput{240}(0,0){\PAT}
\rput(1.15,0){\makebox(0,0)[cl]{\hbox{{$c$}}}}
\rput(-1.15,0){\makebox(0,0)[cr]{\hbox{{$d$}}}}
\rput(-0.5,-1){\makebox(0,0)[tc]{\hbox{{$a$}}}}
\rput(0.5,-1){\makebox(0,0)[tc]{\hbox{{$b$}}}}
\rput(0.65,0.85){\makebox(0,0)[cl]{\hbox{{$f$}}}}
\rput(-0.2,0.1){\makebox(0,0)[br]{\hbox{{$e$}}}}
\rput(-0.65,0.85){\makebox(0,0)[cr]{\hbox{{$g$}}}}
}
\end{pspicture}
\ee
Adding the new vertex does not change the Poisson dimension because we now have a new Casimir operator. We impose the condition that this Casimir is equal the unity:
\be\label{Cas}
e^2abcdfg =1.
\ee

Then the above ``extended'' transformation is just the mutation $\mu_g$ at the newly added vertex $g$
followed by the operation $\mathfrak S_{f,g}$ of interchanging $f\leftrightarrow g$. Indeed, we have:
$$
\mathfrak S_{f,g}\mu_g(g,f^{1/2},e)=\mathfrak S_{f,g} \bigl(g^{-1}, f^{1/2}(1+g), e(1+g^{-1})^{-1}\bigr)= \bigl(f^{-1}, g^{1/2}(1+f), e(1+f^{-1})^{-1}\bigr),
$$
and, upon substitution $g=e^{-1}(abcdf)^{-1/2}$, we obtain transformation (\ref{mut-ex}) for $f^{1/2}$ and $e$.

Note that we can always express $g$ via other variables using (\ref{Cas}); e.g., the above $\hat e$ is nothing but $g^{-1/2}$.

We now formulate the theorem that gives a \emph{complete cluster-algebra description of the Teichm\"uller space of smooth Riemann surfaces of genus two}.

\begin{theorem}\label{genus-2}
The Teichm\"uller space $\mathcal T_{2,0}$ of \emph{smooth} Riemann surfaces of genus two in the Poincar\'e uniformization is identified with the space of real positive cluster variables $\{a,b,c,d,e,f,g\}$ of the $X_7$ quiver (\ref{X7}) with the restriction (\ref{Cas}) imposed. Then, 
\begin{itemize}
\item all geodesic functions are elements of an upper cluster algebra (positive Laurent polynomials) of the $X_7$ quiver variables generated, using skein and Poisson relations, by five elements:
\begin{align*}
&G_{1,2}=(da)^{1/2}\left(1+\frac 1d+\frac1{da}\right),\qquad \widetilde G_{1,2}=(bc)^{1/2}\left(1+\frac 1b+\frac1{bc}\right), \qquad G_B=(fg)^{1/2}\left(1+\frac 1f+\frac1{fg}\right)\\
&G_{2,3}=(gd)^{-1/2} \left( 1+1/a \right)+(f/a)^{1/2} \widetilde G_{1,2}+(gd)^{1/2}(1+f),\\
&\widetilde G_{2,3}=(gb)^{-1/2} \left( 1+1/c \right)+(f/c)^{1/2} G_{1,2}+(gb)^{1/2}(1+f)
\end{align*}
\item every sequence of mutations preserving the graph (\ref{X7}) is a combination of Dehn twists and symmetry transformations in $\Sigma_{2,0}$. All these mutations and Dehn twists are Poisson morphisms of both the $X_7$ cluster algebra and of the Goldman Poisson algebra of geodesic functions.
\end{itemize}
\end{theorem}


\begin{remark}\label{rmrk:ModularRelations} Due to Birman--Hilden the modular group $\operatorname{Mod}(S_2)$ has the following presentation. Let $d_1,d_2,d_3,d_4,d_5$ denote the Dehn twists about the curves shown on the Figure~\ref{fi:genus-2} with geodesic functions $G_{1,2},G_{2,3},G_B,\tilde G_{2,3},\tilde G_{1,2}$. Then,
\begin{center}
  \begin{tabular}{@{} ll @{}}
    $\operatorname{Mod}(S_2)=\langle d_1,d_2,d_3,d_4,d_5\quad |$  & $d_i d_j=d_j d_i$ \text{ for } $|i-j|>1$, \\ 
     & $d_i d_{i+1} d_i=d_{i+1}d_i d_{i+1}$, \\ 
     & $(d_1 d_2 d_3)^4=d_5^2$, \\ 
     & $[(d_5d_4d_3d_2d_1d_1d_2d_3d_4d_5),d_1]=1$, \\ 
     & $(d_5d_4d_3d_2d_1d_1d_2d_3d_4d_5)^2=1\rangle$.
\end{tabular}
\end{center}

Equivalently, 
\begin{center}
  \begin{tabular}{@{} ll @{}}
    $\operatorname{Mod}(S_2)=\langle d_1,d_2,d_3,d_4,d_5\quad |$  & $d_i d_j=d_j d_i$ \text{ for } $|i-j|>1$, \\ 
     & $d_i d_{i+1} d_i=d_{i+1}d_i d_{i+1}$, \\ 
     & $(d_1 d_2 d_3 d_4 d_5)^6=1$, \\ 
     & $(d_5d_4d_3d_2d_1d_1d_2d_3d_4d_5)^2=1,
     \rangle$.
\end{tabular} 
\end{center}

\begin{lemma} Cluster transformations in $X_7$  corresponding to the Dehn twists  $G_{1,2},G_{1,3},G_B,\widetilde G_{2,3},\widetilde G_{1,2}$ satisfy relations~(\ref{rmrk:ModularRelations})
\end{lemma}

\begin{proof}
We checked using symbolic calculations that mutations satisfy relation $d_5 d_4 d_3 d_2 d_1 d_1 d_2 d_3 d_4 d_5=\rm{Id}$.
This implies the last two relations in the first set of relations.
The third relation of the first set is checked directly using symbolic computations.
Symbolic calculations show also that $(d_1 d_2 d_3 d_4 d_5)^2$ realizes permutation $\begin{pmatrix} a& b & c & d & e & f & g \\
c & f & g & b & e & d & a
\end{pmatrix}$. of order $3$. This proves the third relation of the second set. 
\end{proof}

\end{remark}

\begin{remark}
Note that since the modular group transformation is a Poisson morphism, Casimir operators are always preserved by this action, In particular, in any other cluster having the $X_7$ form, the combination (\ref{Cas}) expressed in the transformed variables will be equal to the unity.
\end{remark}

\begin{remark}
Expressions for geodesic functions for separating geodesics are amazingly long when expressed in terms of cluster variables. Depending on the choice of initial cluster, the Laurent polynomial for $\mathfrak M$ comprises $46$ terms for $X_7$ and $50$ terms for both $K_{3,3}$ and for the ``original'' cluster. This is in stark contrast to geometrical clusters describing surfaces with holes where expressions for boundary geodesic functions always contain just two terms, and the lengths of the corresponding geodesics (which are Casimirs) are linear sums of logarithms of cluster variables. 
\end{remark}

\begin{proof}
To proof Theorem for each collection of geodesic functions $G_{1,2},G_{2,3},G_B,\tilde G_{2,3},\tilde G_{1,2}$ we will construct the set of $2\times 2$ matrices representing  these elements. We fix two points $\alpha$ and $\beta$ on the surface and consider the vector bundle over the surface with fiber ${\mathbb R}^2$. Let  ${\mathbb R}^2_\alpha$ be the fiber over $\alpha$, ${\mathbb R}^2_\beta$ be the fiber over $\beta$. A flat $PSL_2({\mathbb R})$-connection associates with every path $\gamma$ connecting $\alpha$ to $\beta$ a transport operator $T_\gamma:{\mathbb R}_\alpha^2\to{\mathbb R}_\beta^2$. Choosing bases in ${\mathbb R}_\alpha^2$ and ${\mathbb R}_\beta^2$ we obtain matrix $M_\gamma$.

Changes of bases 
$C_\alpha:{\mathbb R}^2_\alpha\to {\mathbb R}^2_\alpha$,  $C_\beta:{\mathbb R}^2_\beta\to {\mathbb R}^2_\beta$ lead to the 
action $M_\gamma\mapsto C_\beta^{-1}M_\gamma C_\alpha$. If $\gamma_1$, $\gamma_2$ are two paths connecting $\alpha$ to 
$\beta$ then
we associate the matrix $M_{\gamma_1} M_{\gamma_2}^{-1}$ to a loop bases at $\beta$ which follows $\gamma_2$ in the opposite direction and then $\gamma_1$. Clearly, this construction defines a representation of the fundamental group of the surface with the basepoint $\beta$.  Let us denote by $\gamma_{1,2},\gamma_{2,3},\gamma_B,\tilde\gamma_{2,3},\tilde\gamma_{1,2}$ the elements in the fundamental group represented by the cycles shown on the left part of  Figure~\ref{fi:genus-234}. Note that such cycles are not independent. 
Namely,  $\gamma_B=\gamma_{1,2}\tilde\gamma_{1,2}$, and hence, 
$M_{\tilde\gamma_{1,2}}=M_{\gamma_{1,2}}^{-1}M_{\gamma_B}$. 

Below we denote by $M_i$ the matrix $M_{\gamma_i}$.
In order to construct such representation, we are looking for five $2\times 2$ matrices $M_i$ each matrix representing parallel transport $M_{\gamma_i}$ satisfying $G_{i,j}=\operatorname{tr}\left(M_i M_j^{-1}\right)$ (see Fig.\ref{fi:pathsAB}) for genus two.  As we discussed above, the collection of matrices $M_1,\dots, M_5$ is determined up to left and right multiplications by nondegenerate $2\times 2$ matrices.

%

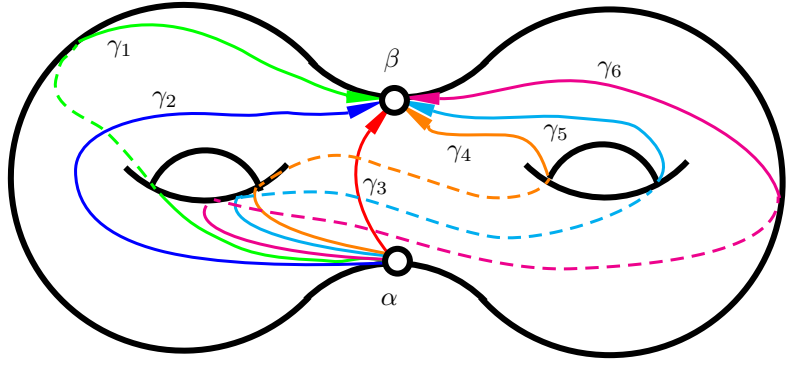
\begin{figure}
\psscalebox{1.0 1.0} 
{
\begin{pspicture}(-3,-5.)(7.5,-1.)
\psarc[linecolor=black, linewidth=0.08, dimen=outer](5.1725054,0.16383636){1.6625}{225.0}{315.0}
\psrotate(5.1225057, -5.3961635){-180.50157}{\psarc[linecolor=black, linewidth=0.08, dimen=outer](5.1225057,-5.3961635){1.6925}{225.0}{315.0}}
\psrotate(2.3200054, -2.5661635){-178.32751}{\psarc[linecolor=black, linewidth=0.08, dimen=outer](2.3200054,-2.5661635){2.3}{218.52817}{134.82762}}
\psarc[linecolor=black, linewidth=0.08, dimen=outer](7.9800053,-2.6261637){2.3}{221.64696}{137.28438}
\psrotate(2.5975056, -1.3811636){2.6401985}{\psarc[linecolor=black, linewidth=0.08, dimen=outer](2.5975056,-1.3811636){1.4875}{222.19043}{314.8429}}
\psrotate(2.6050055, -2.9761636){5.510523}{\psarc[linecolor=black, linewidth=0.08, dimen=outer](2.6050055,-2.9761636){0.775}{12.195461}{149.88597}}
\psrotate(7.9175053, -1.3411636){2.6401985}{\psarc[linecolor=black, linewidth=0.08, dimen=outer](7.9175053,-1.3411636){1.4875}{222.19043}{314.8429}}
\psrotate(7.8850055, -2.8961637){5.510523}{\psarc[linecolor=black, linewidth=0.08, dimen=outer](7.8850055,-2.8961637){0.775}{11.692449}{149.88597}}
\pscustom[linecolor=black, linewidth=0.04]
{
\newpath
\moveto(3.8200054,-8.496163)
}
\psbezier[linecolor=green, linewidth=0.04](5.1600056,-3.7161636)(4.6003366,-3.5473044)(4.5769477,-3.7096674)(4.2800055,-3.676163635253906)(3.9830632,-3.6426597)(3.4003365,-3.6673043)(2.9600055,-3.4561636)(2.5196745,-3.245023)(2.4030633,-3.2026596)(2.0000055,-2.7761636)
\psbezier[linecolor=green, linewidth=0.04, linestyle=dashed, dash=0.17638889cm 0.10583334cm](1.9400055,-2.7161636)(1.7925351,-2.507913)(1.414809,-2.0544438)(1.2400055,-1.9161636352539062)(1.065202,-1.7778835)(0.532535,-1.627913)(0.6200055,-1.2561636)(0.70747596,-0.88441426)(0.70520204,-1.0578835)(0.9000055,-0.7961636)
\psbezier[linecolor=green, linewidth=0.04, arrowsize=0.05291667cm 3.0,arrowlength=3.0,arrowinset=0.0]{->}(0.9200055,-0.75616366)(1.7852474,-0.33878)(2.5004573,-0.57461953)(3.1000054,-0.8761636352539063)(3.6,-1.1777078)(3.6652474,-1.13878)(4.0000057,-1.2961637)(4.3347635,-1.4535472)(4.699554,-1.5)(5.,-1.5161636)
\psbezier[linecolor=blue, linewidth=0.04, arrowsize=0.05291667cm 3.0,arrowlength=3.0,arrowinset=0.0]{->}(5.0600057,-3.6961637)(3.818827,-3.7569213)(2.699954,-3.71634)(2.0000055,-3.536163635253906)(1.3000569,-3.3559873)(0.8816182,-3.013381)(0.8800055,-2.4761636)(0.8783928,-1.9389464)(1.7717702,-1.6691048)(2.6800056,-1.7161636)(3.5,-1.7632225)(3.2225873,-1.7079778)(3.7000055,-1.7)(4.1,-1.75)(4.5137076,-1.7)(5.,-1.5361637)
\psbezier[linecolor=magenta, linewidth=0.04](5.1400056,-3.6561637)(3.6800056,-3.6161637)(2.3800056,-3.3161635)(2.6200056,-2.9361636352539064)
\psbezier[linecolor=cyan, linewidth=0.04](5.1600056,-3.6161637)(4.5600057,-3.5761635)(2.7200055,-3.3161635)(3.0400054,-2.8161636352539063)
\psbezier[linecolor=orange, linewidth=0.04](5.1400056,-3.5961637)(4.7200055,-3.5161636)(3.0600054,-3.2561636)(3.2800055,-2.6961636352539062)
\psbezier[linecolor=red, linewidth=0.04, arrowsize=0.05291667cm 3.0,arrowlength=3.0,arrowinset=0.0]{->}(5.1600056,-3.7161636)
(4.5,-3.)(4.5,-2.3)(5.1200056,-1.5361637)
\psbezier[linecolor=orange, linewidth=0.04, linestyle=dashed, dash=0.17638889cm 0.10583334cm](3.3200054,-2.6761637)(4.2644258,-2.0022411)(4.700094,-2.3428314)(5.0400057,-2.4361636352539064)(5.3799167,-2.5294957)(5.904426,-2.722241)(6.1800056,-2.7961636)(6.455585,-2.8700862)(7.0399165,-2.849496)(7.1800056,-2.5761635)
\psbezier[linecolor=cyan, linewidth=0.04, linestyle=dashed, dash=0.17638889cm 0.10583334cm](3.1000054,-2.8161635)(2.6201158,-2.9110134)(3.8926752,-2.6574855)(4.5000057,-2.7961636352539063)(5.107336,-2.9348416)(5.920116,-3.2910135)(6.4600053,-3.3561637)(6.999895,-3.4213138)(8.107336,-3.1348417)(8.540006,-2.7361636)
\psbezier[linecolor=magenta, linewidth=0.04, linestyle=dashed, dash=0.17638889cm 0.10583334cm](2.7200055,-2.8561637)(3.5419528,-3.078542)(3.84324,-3.0618522)(4.4600053,-3.176163635253906)(5.076771,-3.2904751)(6.361953,-3.818542)(7.4800053,-3.7761636)(8.598058,-3.7337854)(10.256771,-3.570475)(10.220005,-2.8161635)
\psbezier[linecolor=orange, linewidth=0.04, arrowsize=0.05291667cm 3.0,arrowlength=3.0,arrowinset=0.0]{->}(7.1600056,-2.5361636)(6.965696,-1.9201038)(6.7379956,-1.9961789)(6.1000056,-1.9961636352539063)(5.462015,-1.9961483)(5.588966,-1.8822379)(5.1600056,-1.5561637)
\psbezier[linecolor=cyan, linewidth=0.04, arrowsize=0.05291667cm 3.0,arrowlength=3.0,arrowinset=0.0]{->}(8.620006,-2.6161637)(8.8546,-2.339995)(8.490371,-1.9933827)(8.340006,-1.8961636352539062)(8.18964,-1.7989446)(7.8146005,-1.799995)(7.3800054,-1.7561636)(6.9454107,-1.7123322)(6.18964,-1.8589447)(5.1400056,-1.5161636)
\psbezier[linecolor=magenta, linewidth=0.04, arrowsize=0.05291667cm 3.0,arrowlength=3.0,arrowinset=0.0]{->}(10.2400055,-2.8361635)(10.07779,-2.1696446)(9.026431,-1.6131102)(8.060005,-1.3561636352539062)(7.09358,-1.0992172)(6.3818874,-1.5336627)(5.2000055,-1.4961636)
\pscircle[linecolor=black, linewidth=0.08, fillstyle=solid, dimen=outer](5.1600056,-3.6761637){0.2}
\pscircle[linecolor=black, linewidth=0.08, fillstyle=solid, dimen=outer](5.1200056,-1.5361637){0.2}
\rput[bl](4.9400053,-4.2561636){$\alpha$}
\rput[bl](4.9800053,-1.1561637){$\beta$}
\rput[bl](1.3000054,-0.95616364){$\gamma_1$}
\rput[bl](1.9000055,-1.6361636){$\gamma_2$}
\rput[bl](4.7,-2.8){$\gamma_3$}
\rput[bl](5.8200054,-2.3561637){$\gamma_4$}
\rput[bl](7.1000056,-2.0761635){$\gamma_5$}
\rput[bl](7.8200054,-1.2361636){$\gamma_6$}
\end{pspicture}
}
\caption{\small The hyperbolic genus 2 surface is equipped with natural $PSL_2$-connection. The choice of bases in two-dimensional fibers ${\mathbb R}^2_\alpha$ and  ${\mathbb R}^2_\beta$ at points $\alpha$ and $\beta$ associates to every path connecting $\alpha$ to $\beta$ a $2\times 2$ matrix.  $M_1,\dots,M_6$ are matrices of parallel transports along paths $\gamma_1,\dots,\gamma_6$ connecting points $\alpha$ and $\beta$. The matrices $M_i M_{i+1}^{-1}$ for $i=1\dots 6$ give parallel transport along the loops with geodesic functions $G_{1,2},G_{2,3}, G_B, \widetilde G_{2,3},\widetilde G_{1,2}$ (see Fig.~\ref{fi:genus-2}) being their traces. Collection $M_1,\dots,M_5$ up to a diagonal conjugation determines a flat $PSL_2$ connection on a genus two surface.}
\label{fi:pathsAB}
\end{figure}

%
%
%

We choose bases in ${\mathbb R}^2_\alpha$ and ${\mathbb R}^2_\beta$ in such a way that that $M_1=\begin{pmatrix}1 & 0\\ 0& 1\end{pmatrix}$, $M_2=\begin{pmatrix} e^{\ell_{1,2}/2} & 0\\ 0& e^{-\ell_{1,2}/2}\end{pmatrix}$,
$M_3=\begin{pmatrix} a & b\\  b& c\end{pmatrix}$. These conditions exhaust all the freedom in choice of bases and we set   $M_4=\begin{pmatrix} d & e\\  f& g\end{pmatrix}$, $M_5=\begin{pmatrix} h & i\\  j& k\end{pmatrix}$, where parameters $a,b,c,d,e,f,g,h,i,j,k$ satisfy conditions $ac-b^2=1, dg-ef=1, kh-ij=1$.  
Note that parameter $|\ell_{1,2}|$ is determined by $G_{1,2}$ and $\ell_{1,2}$ is therefore completely determined by $G_{1,2}$ up to a sign.
By definition, the geodesic functions satisfy
\begin{align*}
&G_{1,3}=a+c, \quad G_{1,4}=d+g,\quad  G_{1,5}=h+k, \quad  G_{2,3}=e^{\ell_{1,2}/2} c+e^{-\ell_{1,2}/2}a,\\
&G_{2,4}=e^{\ell_{1,2}/2} g+e^{-\ell_{1,2}/2} d, \quad G_{2,5}=e^{\ell_{1,2}/2} k+e^{-\ell_{1,2}/2} h.
\end{align*}
These relations give a system of six linear equations for $a,b,c,d,g,h,k$ whose solution is
\[
\begin{array}{cc}
 c=\dfrac{G_{2,3}-G_{1,3}e^{-\ell_{1,2}/2}}{e^{\ell_{1,2}/2}-e^{-\ell_{1,2}/2}} &    a=\dfrac{-G_{2,3}+G_{1,3}e^{\ell_{1,2}/2}}{e^{\ell_{1,2}/2}-e^{-\ell_{1,2}/2}}   \\
 g=\dfrac{G_{2,4}-G_{1,4}e^{-\ell_{1,2}/2}}{e^{\ell_{1,2}/2}-e^{-\ell_{1,2}/2}} &    d=\dfrac{-G_{2,4}+G_{1,4}e^{\ell_{1,2}/2}}{e^{\ell_{1,2}/2}-e^{-\ell_{1,2}/2}}   \\
 k=\dfrac{G_{2,5}-G_{1,5}e^{-\ell_{1,2}/2}}{e^{\ell_{1,2}/2}-e^{-\ell_{1,2}/2}} &    h=\dfrac{-G_{2,5}+G_{1,5}e^{\ell_{1,2}/2}}{e^{\ell_{1,2}/2}-e^{-\ell_{1,2}/2}}   \\    
\end{array}
\]
Let us now express the remaining parameters $b,e,f,i,j$ using
relations 
\[
\begin{array}{c}
G_{3,4}=tr M_3 M_4^{-1} =a g - b f - b e +c d \\
G_{3,5} =tr M_3 M_5^{-1}= a k -b(i+j)+c h \\
G_{4,5}=tr M_4 M_5^{-1}=d k -e j -i f +g h \\
\end{array}
\]
Let ${\mathfrak M}_{abc}=G_{a,b} G_{b,c} G_{a,c}-G_{a,b}^2-G_{b,c}^2-G_{a,c}^2$.
Notice 
\[b^2=a c-1=\dfrac{(e^{\ell_{1,2}/2}+e^{-\ell_{1,2}/2})G_{1,3}G_{2,3}-G_{2,3}^2-G_{1,3}^2}{(e^{\ell_{1,2}/2}-e^{-\ell_{1,2}/2})^2}-1=\dfrac{G_{1,2}G_{2,3}G_{1,3}-G_{1,2}^2-G_{2,3}^2-G_{1,3}^2+4}{(e^{\ell_{1,2}/2}-e^{-\ell_{1,2}/2})^2}=\dfrac{\M_{123}+4}{G_{1,2}^2-4}
\]
Hence, $b=\left[\dfrac{{\mathfrak M}_{123}+4}{G_{1,2}^2-4}\right]^{1/2}$. Similarly, $$ef=\dfrac{{\mathfrak M}_{124}+4}{G_{1,2}^2-4}$$
 and 
 $$(f+e)b=ag+cd-G_{3,4}=\dfrac{2G_{2,3}G_{2,4}-G_{1,2}G_{1,3}G_{2,4}-G_{1,2}G_{2,3}G_{1,4}+G_{1,3}G_{1,4}(G_{1,2}^2-2)-G_{3,4}G_{1,2}^2+4G_{3,4}}{G_{1,2}^2-4}$$
 Parameter $f$ can be found from the latter equation and $e$ from the former one.\newline
 Parameters $i,j$ are found from equations
 $$ij=\dfrac{{\mathfrak M}_{125}+4}{G_{1,2}^2-4}$$
 and 
 $$(i+j)b=\dfrac{2G_{2,3}G_{2,5}-G_{1,2}G_{1,3}G_{2,5}-G_{1,2}G_{2,3}G_{1,5}+G_{1,3}G_{1,5}(G_{1,2}^2-2)-G_{3,5}G_{1,2}^2+4G_{3,5}}{G_{1,2}^2-4}$$ 
These equations allow to find all parameters $a,b,c,d,e,f,g,h,i,j,k$.
The last equation $dk+gh-ej-if=G_{4,5}$ checked by computer simulations gives the consistency condition for $G_{1,2}$, $G_{1,3}$, $G_{1,4}$, $G_{1,5}$, $G_{2,3}$, $G_{2,4}$, $G_{2,5}$, $G_{3,4}$, $G_{3,5}$, $G_{4,5}$. We have also checked the condition of triviality of monodromy along the composition of commutators of $A$- and $B$-cycles given in this case by the following matrix product:
$$
M_5^{-1}M_4M_3^{-1}M_2M_1^{-1}M_5M_4^{-1}M_3M_2^{-1}M_1=I,
$$
which completes the proof.
\end{proof}

\section{Cluster variable description of higher genus closed Riemann surfaces}

In this section, we discuss extensions of our construction to higher genus surfaces. Note that for genus three, we obtained the description of the finite quotient of Teichm\"uller space $\mathcal{T}_{3,0}/\mathbb{Z}_2$. For genus $g>3$ the construction requires Hamiltonian reduction, which is discussed for $n=3$ (as a toy example) and $n=5$. 

In Figure~\ref{fi:genus-234} we present patterns for surfaces of genera 2,3, and 4.



\begin{figure}[H]
\begin{pspicture}(-2.5,-5)(2.5,2)
{\psset{unit=0.5}
\newcommand{\ELLARC}[7]{
\parametricplot[linecolor=#1, linewidth=#2 pt, linestyle=#3]{#4}{#5}{#6 t cos mul #7 t sin mul}
}
\definecolor{lightblue}{rgb}{.85, .5, 1}
\definecolor{darkgreen}{rgb}{.0, .5, 1}
\newcommand{\PAT}{%
\psline[linewidth =0.2,linestyle=dotted,linecolor=darkgreen](-1,3)(-1,-9)
\rput(-1,-0.3){\ELLARC{black}{1}{solid}{-45}{225}{3.2}{2.5}}
\rput(-1,-0.9){\ELLARC{black}{1}{solid}{20}{160}{.9}{.4}}
\rput(-1,-0.3){\ELLARC{black}{1}{solid}{210}{330}{1.5}{.6}}
\rput(-4.8,1){\makebox(0,0){$\mathbb A$}}
\pscircle[linewidth=0.5pt](-4.8,1){.6}
\rput(-4.8,-6.5){\makebox(0,0){$\widetilde{\mathbb A}$}}
\pscircle[linewidth=0.5pt](-4.8,-6.5){.6}
\rput(-1,-2.75){\ELLARC{lightblue}{1}{dashed}{0}{180}{1.85}{.5}}
\rput(-1,-2.75){\ELLARC{lightblue}{2}{dashed}{180}{360}{1.85}{.5}}
\rput(-4.1,-2.75){\ELLARC{black}{1}{solid}{-45}{45}{1.2}{.94}}
\rput(2.1,-2.75){\ELLARC{black}{1}{solid}{135}{225}{1.2}{.94}}
\rput(-1,-5.2){\ELLARC{black}{1}{solid}{-225}{45}{3.2}{2.5}}
\rput(-1,-0.9){\ELLARC{black}{1}{solid}{20}{160}{.9}{.4}}
\rput(-1,-0.3){\ELLARC{black}{1}{solid}{210}{330}{1.5}{.6}}
\rput(-1,-5.8){\ELLARC{black}{1}{solid}{20}{160}{.9}{.4}}
\rput(-1,-5.2){\ELLARC{black}{1}{solid}{210}{330}{1.5}{.6}}
}
%
\rput(-1,-0.6){\ELLARC{blue}{1}{solid}{0}{360}{2}{0.8}}
\rput{-90}(-1,.85){\ELLARC{red}{1}{solid}{180}{360}{1.3}{0.5}}
\rput{-90}(-1,.85){\ELLARC{red}{1}{dashed}{0}{180}{1.3}{0.5}}
%
\rput(-1,-5.6){\ELLARC{blue}{1}{solid}{0}{360}{2}{0.8}}
\rput{-90}(-1,-6.75){\ELLARC{red}{1}{solid}{180}{360}{0.95}{0.35}}
\rput{-90}(-1,-6.75){\ELLARC{red}{1}{dashed}{0}{180}{0.95}{0.35}}
\rput(0,0){\PAT}
\rput(-2.2,1){\makebox(0,0){$G_{1,2}$}}
\rput(-3.4,0){\makebox(0,0){$G_{2,3}$}}
\rput(-3.5,-2.75){\makebox(0,0){$\mathfrak M$}}
\rput(-2,-7){\makebox(0,0){$\widetilde G_{1,2}$}}
\rput(-3.3,-4.7){\makebox(0,0){$\widetilde G_{2,3}$}}
\rput(-2.2,-4){\makebox(0,0){$G_{B}$}}
\rput(-1,-3.15){\ELLARC{red}{1}{solid}{90}{270}{0.7}{2.23}}
\rput(-1,-3.15){\ELLARC{red}{1}{dashed}{-90}{90}{0.7}{2.23}}
}
\end{pspicture}
\begin{pspicture}(-2.5,-5)(2.5,2)
{\psset{unit=0.5}
\newcommand{\ELLARC}[7]{
\parametricplot[linecolor=#1, linewidth=#2 pt, linestyle=#3]{#4}{#5}{#6 t cos mul #7 t sin mul}
}
\definecolor{lightblue}{rgb}{.85, .5, 1}
\definecolor{darkgreen}{rgb}{.0, .5, 1}
\newcommand{\PAT}{%
\psline[linewidth =0.2,linestyle=dotted,linecolor=darkgreen](0,3)(0,-9)
\rput(0,-0.3){\ELLARC{black}{1}{solid}{-45}{225}{4.2}{2.5}}
\rput(-4,-2.75){\ELLARC{black}{1}{solid}{-45}{45}{1.4}{.97}}
\rput(4,-2.75){\ELLARC{black}{1}{solid}{135}{225}{1.4}{.97}}
\rput(0,-0.9){\ELLARC{black}{1}{solid}{20}{160}{.9}{.4}}
\rput(0,-2.75){\ELLARC{black}{1}{solid}{0}{360}{1}{0.8}}
\rput(0,-0.3){\ELLARC{black}{1}{solid}{210}{330}{1.5}{.6}}
\rput(0,-5.2){\ELLARC{black}{1}{solid}{-225}{45}{4.2}{2.5}}
\rput(0,-5.8){\ELLARC{black}{1}{solid}{20}{160}{.9}{.4}}
\rput(0,-5.2){\ELLARC{black}{1}{solid}{210}{330}{1.5}{.6}}
\rput(-4.8,1){\makebox(0,0){$\mathbb A$}}
\pscircle[linewidth=0.5pt](-4.8,1){.6}
\rput(-4.8,-6.5){\makebox(0,0){$\widetilde{\mathbb A}$}}
\pscircle[linewidth=0.5pt](-4.8,-6.5){.6}
\rput(-1.83,-2.75){\ELLARC{lightblue}{1}{dashed}{0}{180}{0.78}{.3}}
\rput(-1.83,-2.75){\ELLARC{lightblue}{2}{dashed}{180}{360}{0.78}{.3}}
\rput(1.83,-2.75){\ELLARC{lightblue}{1}{dashed}{0}{180}{0.78}{.3}}
\rput(1.83,-2.75){\ELLARC{lightblue}{2}{dashed}{180}{360}{0.78}{.3}}
}
%
\rput(0,-0.6){\ELLARC{blue}{1}{solid}{0}{360}{2}{0.8}}
\rput{-90}(0,.85){\ELLARC{red}{1}{solid}{180}{360}{1.3}{0.5}}
\rput{-90}(0,.85){\ELLARC{red}{1}{dashed}{0}{180}{1.3}{0.5}}
\rput{-90}(0,-1.4){\ELLARC{red}{1}{solid}{180}{360}{0.5}{0.2}}
\rput{-90}(0,-1.4){\ELLARC{red}{1}{dashed}{0}{180}{0.5}{0.2}}
%
\rput(0,-5.6){\ELLARC{blue}{1}{solid}{0}{360}{2}{0.8}}
\rput{-90}(0,-6.75){\ELLARC{red}{1}{solid}{180}{360}{0.95}{0.35}}
\rput{-90}(0,-6.75){\ELLARC{red}{1}{dashed}{0}{180}{0.95}{0.35}}
\rput{-90}(0,-4.45){\ELLARC{red}{1}{solid}{180}{360}{0.9}{0.35}}
\rput{-90}(0,-4.45){\ELLARC{red}{1}{dashed}{0}{180}{0.9}{0.35}}
\rput(0,0){\PAT}
\rput(-1.2,1.5){\makebox(0,0){$G_{1,2}$}}
\rput(-2.7,-.6){\makebox(0,0){$G_{2,3}$}}
\rput(-1.5,-1.6){\makebox(0,0){$G_{3,4}$}}
\rput(-3.5,-2.75){\makebox(0,0){$\mathfrak M_1$}}
\rput(3.5,-2.75){\makebox(0,0){$\mathfrak M_2$}}
\rput(-1.1,-7){\makebox(0,0){$\widetilde G_{1,2}$}}
\rput(-2.7,-5.1){\makebox(0,0){$\widetilde G_{2,3}$}}
\rput(-1.2,-4.1){\makebox(0,0){$\widetilde G_{3,4}$}}
\rput(1.1,-4){\makebox(0,0){$G_{B}$}}
\rput(0,-2.75){\ELLARC{blue}{1}{solid}{0}{360}{1.4}{1.1}}
}
\end{pspicture}
\begin{pspicture}(-2.5,-5)(2.5,2)
{\psset{unit=0.5}
\newcommand{\ELLARC}[7]{
\parametricplot[linecolor=#1, linewidth=#2 pt, linestyle=#3]{#4}{#5}{#6 t cos mul #7 t sin mul}
}
\definecolor{lightblue}{rgb}{.85, .5, 1}
\definecolor{darkgreen}{rgb}{.0, .5, 1}
\newcommand{\PAT}{%
\psline[linewidth =0.2,linestyle=dotted,linecolor=darkgreen](0,5)(0,-10.5)
\rput(0,0.7){\ELLARC{black}{1}{solid}{-45}{225}{3.2}{3.5}}
\rput(-3.1,-2.75){\ELLARC{black}{1}{solid}{-45}{45}{1.2}{1.35}}
\rput(3.1,-2.75){\ELLARC{black}{1}{solid}{135}{225}{1.2}{1.35}}
\rput(0,-6.2){\ELLARC{black}{1}{solid}{-225}{45}{3.2}{3.5}}
\rput(0,-0.9){\ELLARC{black}{1}{solid}{20}{160}{.9}{.4}}
\rput(0,-0.3){\ELLARC{black}{1}{solid}{210}{330}{1.5}{.6}}
\rput(0,-5.8){\ELLARC{black}{1}{solid}{20}{160}{.9}{.4}}
\rput(0,-5.2){\ELLARC{black}{1}{solid}{210}{330}{1.5}{.6}}
\rput(0,2.8){\psset{unit=0.7} 
\rput(0,-0.9){\ELLARC{black}{1}{solid}{20}{160}{.9}{.4}}
\rput(0,-0.3){\ELLARC{black}{1}{solid}{210}{330}{1.5}{.6}}
}
\rput(0,-7.4){\psset{unit=0.7} 
\rput(0,-0.9){\ELLARC{black}{1}{solid}{20}{160}{.9}{.4}}
\rput(0,-0.3){\ELLARC{black}{1}{solid}{210}{330}{1.5}{.6}}
}
\rput(0,-2.75){\ELLARC{lightblue}{1}{dashed}{0}{180}{1.85}{.5}}
\rput(0,-2.75){\ELLARC{lightblue}{2}{dashed}{180}{360}{1.85}{.5}}
\rput(-3.8,2){\makebox(0,0){$\mathbb A$}}
\pscircle[linewidth=0.5pt](-3.8,2){.6}
\rput(-3.8,-7.5){\makebox(0,0){$\widetilde{\mathbb A}$}}
\pscircle[linewidth=0.5pt](-3.8,-7.5){.6}
}
%
\rput(0,-0.6){\ELLARC{blue}{1}{solid}{0}{360}{2}{0.8}}
\rput(0,2.35){\ELLARC{blue}{1}{solid}{0}{360}{1.5}{0.6}}
\rput{-90}(0,.85){\ELLARC{red}{1}{solid}{180}{360}{1.3}{0.5}}
\rput{-90}(0,.85){\ELLARC{red}{1}{dashed}{0}{180}{1.3}{0.5}}
\rput{-90}(0,3.33){\ELLARC{red}{1}{solid}{180}{360}{0.83}{0.3}}
\rput{-90}(0,3.33){\ELLARC{red}{1}{dashed}{0}{180}{0.83}{0.3}}
%
\rput(0,-5.6){\ELLARC{blue}{1}{solid}{0}{360}{2}{0.8}}
\rput(0,-7.8){\ELLARC{blue}{1}{solid}{0}{360}{1.5}{0.6}}
\rput{-90}(0,-6.75){\ELLARC{red}{1}{solid}{180}{360}{0.95}{0.35}}
\rput{-90}(0,-6.75){\ELLARC{red}{1}{dashed}{0}{180}{0.95}{0.35}}
\rput{-90}(0,-8.85){\ELLARC{red}{1}{solid}{180}{360}{0.8}{0.3}}
\rput{-90}(0,-8.85){\ELLARC{red}{1}{dashed}{0}{180}{0.8}{0.3}}
\rput(0,0){\PAT}
\rput(-1,3.4){\makebox(0,0){\small $G_{1,2}$}}
\rput(-2,2){\makebox(0,0){\small $G_{2,3}$}}
\rput(-1.2,0.9){\makebox(0,0){\small $G_{3,4}$}}
\rput(-2.4,0.1){\makebox(0,0){\small $G_{4,5}$}}
\rput(-2.5,-2.75){\makebox(0,0){\small $\mathfrak M$}}
\rput(-1,-8.8){\makebox(0,0){\small $\widetilde G_{1,2}$}}
\rput(-2.1,-7.5){\makebox(0,0){\small $\widetilde G_{2,3}$}}
\rput(-1,-6.8){\makebox(0,0){\small $\widetilde G_{3,4}$}}
\rput(-2.2,-4.7){\makebox(0,0){\small $\widetilde G_{4,5}$}}
\rput(-1.2,-4){\makebox(0,0){\small $G_{B}$}}
\rput(0,-3.15){\ELLARC{red}{1}{solid}{90}{270}{0.7}{2.23}}
\rput(0,-3.15){\ELLARC{red}{1}{dashed}{-90}{90}{0.7}{2.23}}
}
\end{pspicture}
\caption{\small
Smooth Riemann surfaces of genus two, three and four with sets of geodesics $\gamma_{i,i+1}$ and $\widetilde \gamma_{i,i+1}$ (labeled by the corresponding geodesic functions) that generate the corresponding subsets $\mathbb A$ and $\widetilde{\mathbb A}$ of geodesic functions; the corresponding Casimir elements separate these two subsets. In each case we have to add one more geodesic with the geodesic function $G_B$ to produce the complete set of algebraic elements that generate all other geodesic functions using Goldman brackets and skein relations. Vertical dotted axes are the axes of ${\mathbb Z}_2$-symmetry.
}
\label{fi:genus-234}
\end{figure}
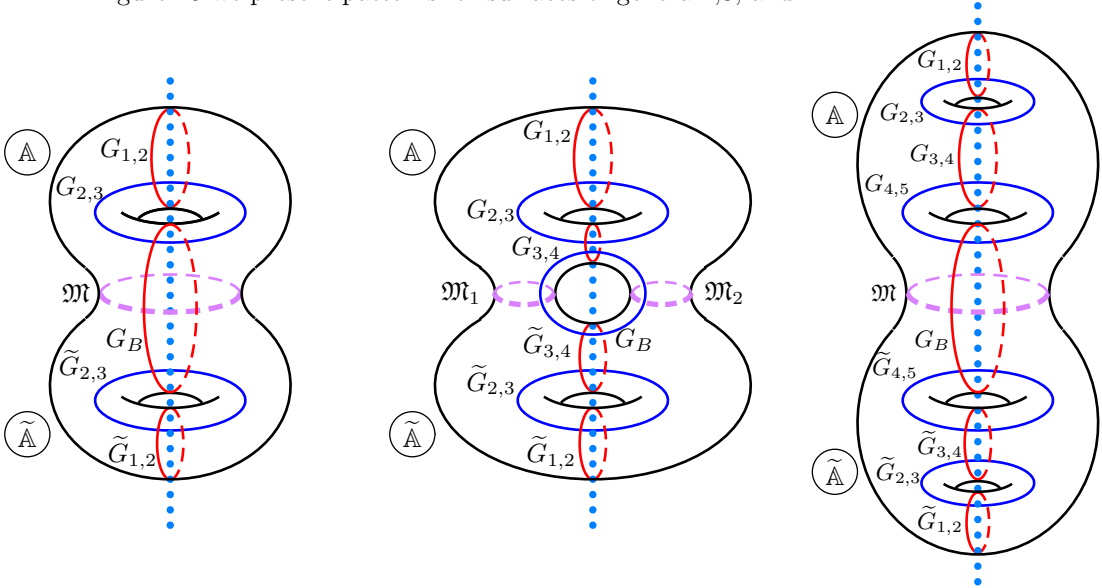

\subsection{Genus three}\label{ss:genus3}
In this section we give cluster expressions for geodesic functions of a genus three surface.

The matrix elements $a_{i,i+1}$, $\widetilde a_{i,i+1}$, which we identify with the geodesic functions $G_{i,i+1}$, $\widetilde G_{i,i+1}$ are
\begin{align*}
&G_{1,2}=\langle b_3a_3d_2c_3 \rangle,\quad G_{2,3}=\langle c_3c_2b_2a_2d_3 \rangle, \quad G_{3,4}=\langle d_3d_2d_1c_1b_1a_1 \rangle,\\
&\widetilde G_{1,2}=\langle d_1a_3b_2c_1 \rangle,\quad \widetilde G_{2,3}=\langle c_1c_2d_2a_2b_1 \rangle, \quad \widetilde G_{3,4}=\langle b_1b_2b_3c_3d_3a_1 \rangle,
\end{align*}
with $|G_{1,2}|=5$, $|G_{2,3}|=6$, and $|G_{3,4}|=7$.  For example, the action of mutations on $G_{1,2}$ gives
\begin{align*}
&\mu_{a_1}G_{1,2}=G_{1,2}=\langle b_3a_3d_2c_3 \rangle,
& \mu_{a_2}G_{1,2}=\langle b_3a_3a_2d_2c_3 \rangle,\\
&\mu_{a_3}G_{1,2}=\langle b_3d_2c_3 \rangle,
&\mu_{b_1}G_{1,2}=G_{1,2}=\langle b_3a_3d_2c_3 \rangle,\\
& \mu_{b_2}G_{1,2}=\langle b_3b_2a_3d_2c_3 \rangle,
&\mu_{b_3}G_{1,2}=\langle a_3d_2c_3b_3 \rangle,\\
&\mu_{c_1}G_{1,2}=G_{1,2}=\langle b_3a_3d_2c_3 \rangle,
& \mu_{c_2}G_{1,2}=\langle b_3a_3d_2c_2c_3 \rangle,\\
& \mu_{c_3}G_{1,2}=\langle c_3b_3a_3d_2 \rangle,
& \mu_{d_1}G_{1,2}=\langle b_3a_3d_1d_2c_3 \rangle,\\
&\mu_{d_2}G_{1,2}=\langle b_3a_3c_3 \rangle,
& \mu_{d_3}G_{1,2}=\langle b_3a_3d_2d_3c_3 \rangle.
\end{align*}

We now transform the original quiver to a more symmetric one: for this, we perform four mutations: $ \mu_{d_1} \mu_{c_2}\mu_{b_3}$ and $ \mu_{a_1}$ (in arbitrary order since those are commuting). The resulting quiver has a Moebius strip-like structure depicted on the right-hand side of Fig.~\ref{fi:SL4}.

Expressions for the same geodesic functions become more symmetric in the quiver in the right-hand side:
\begin{align*}
&G_{1,2}=\langle a_3d_1d_2c_2c_3b_3 \rangle,\quad G_{2,3}=\langle c_3b_3b_2a_2a_1d_3 \rangle, \quad G_{3,4}=\langle a_1d_3d_2c_2c_1b_1 \rangle,\\
&\widetilde G_{1,2}=\langle a_3b_3b_2c_2c_1d_1 \rangle,\quad \widetilde G_{2,3}=\langle c_3d_3d_2a_2a_3b_3 \rangle, \quad \widetilde G_{3,4}=\langle a_1b_1b_2c_2c_3d_3 \rangle,
\end{align*}
with all $|G_{i,i+1}|=7$.

Geodesic functions $\mathfrak M_1$ and $\mathfrak M_2$ corresponding to two separating cycles are determined by algebraic relations
\begin{align}
&\mathfrak M_1+\mathfrak M_2=G_{1,3}G_{2,4}-G_{1,2}G_{3,4}-G_{2,3}G_{1,4}\\
&\mathfrak M_1\mathfrak M_2=G_{1,2}G_{2,3}G_{3,4}G_{1,4}-G_{1,2}G_{2,3}G_{1,3}-G_{2,3}G_{3,4}G_{2,4}-G_{1,2}G_{1,4}G_{2,4}-G_{3,4}G_{1,4}G_{1,3}+\sum_{1\le i<j\le 4}G_{i,j}^2-4
\end{align}
The first relation is obtained as Pfaffian of the matrix $\mathbb A-\mathbb A^{\text{T}}$ (or, $\sqrt{\mathbb A+(-1)\mathbb A^{\text{T}}}$). The second relation is  minus the coefficient of $\lambda$ in $\mathbb A+\lambda \mathbb A^{\text{T}}$.
They remain invariant under the replacement $G_{i,j}\to\widetilde G_{i,j}$ and commute with all geodesic functions $G_{i,j}$ and $\widetilde G_{i,j}$.

\begin{figure}[H]
\begin{pspicture}(-5,-2)(4,3){\psset{unit=1}
%
\multiput(-2,1.7)(1,0){2}{\psline[linecolor=blue,linewidth=2pt]{->}(0.1,0)(.9,0)}
\multiput(-1.5,0.85)(0.5,-0.85){2}{\psline[linecolor=blue,linewidth=2pt]{->}(0.1,0)(.9,0)}
\multiput(-0.5,0.85)(1,0){1}{\psline[linecolor=blue,linewidth=2pt]{->}(0.1,0)(.9,0)}
\multiput(1,0)(0.5,-0.85){2}{\psline[linecolor=blue,linewidth=2pt]{->}(-0.1,0)(-.9,0)}
\multiput(1.5,-0.85)(-1,0){2}{\psline[linecolor=blue,linewidth=2pt]{->}(-0.1,0)(-.9,0)}
%
\multiput(0.5,.85)(0.5,-0.85){2}{\psline[linecolor=blue,linewidth=2pt]{->}(0.05,-0.085)(.45,-0.765)}
\multiput(0,0)(0.5,-0.85){1}{\psline[linecolor=blue,linewidth=2pt]{->}(0.05,-0.085)(.45,-0.765)}
\multiput(-0.5,-.85)(-0.5,0.85){3}{\psline[linecolor=blue,linewidth=2pt]{->}(-0.05,0.085)(-.45,0.765)}
\multiput(0,0)(-0.5,0.85){2}{\psline[linecolor=blue,linewidth=2pt]{->}(-0.05,0.085)(-.45,0.765)}
\multiput(0.5,.85)(-0.5,0.85){1}{\psline[linecolor=blue,linewidth=2pt]{->}(-0.05,0.085)(-.45,0.765)}
\multiput(0.5,-0.85)(0.5,0.85){1}{\psline[linecolor=blue,linewidth=2pt]{->}(0.05,0.085)(.45,0.765)}
\multiput(0,1.7)(-0.5,-0.85){2}{\psline[linecolor=blue,linewidth=2pt]{->}(-0.05,-0.085)(-.45,-0.765)}
\multiput(-1,1.7)(-0.5,-0.85){1}{\psline[linecolor=blue,linewidth=2pt]{->}(-0.05,-0.085)(-.45,-0.765)}
%
\psbezier[linecolor=blue,linewidth=2pt]{->}(0.07,1.75)(2.5,2.5)(2.5,-3)(0.55,-0.93)
\psbezier[linecolor=white,linewidth=3pt]{->}(-1,1.7)(1.5,3.85)(2.7,0.8)(1.53,-0.76)
\psbezier[linecolor=blue,linewidth=2pt]{->}(-1,1.7)(1.5,3.85)(2.7,0.8)(1.53,-0.76)
\psbezier[linecolor=blue,linewidth=2pt]{<-}(-1.07,1.75)(-4,4)(-1.5,-3.5)(0.45,-0.93)
\psbezier[linecolor=white,linewidth=3pt]{<-}(-0.07,1.75)(-4.5,5)(-3.5,-3.5)(-0.55,-0.93)
\psbezier[linecolor=blue,linewidth=2pt]{<-}(-0.07,1.75)(-4.5,5)(-3.5,-3.5)(-0.55,-0.93)
\psbezier[linecolor=white,linewidth=3pt]{<-}(-2.07,1.65)(-4,0.5)(-1.5,-3.5)(1.45,-0.93)
\psbezier[linecolor=blue,linewidth=2pt]{<-}(-2.07,1.65)(-4,0.5)(-1.5,-3.5)(1.45,-0.93)
\multiput(-3,1.7)(0.5,-0.85){4}{
\multiput(1,0)(1,0){3}{\pscircle[fillstyle=solid,fillcolor=lightgray]{.1}}}
\put(-2,1.9){\makebox(0,0)[bc]{\hbox{{$a_1$}}}}
\put(-1,1.9){\makebox(0,0)[bc]{\hbox{{$a_2$}}}}
\put(0,1.9){\makebox(0,0)[bc]{\hbox{{$a_3$}}}}
\put(-1.65,0.9){\makebox(0,0)[br]{\hbox{{$b_1$}}}}
\put(-0.7,1){\makebox(0,0)[br]{\hbox{{$b_2$}}}}
\put(0.3,1){\makebox(0,0)[br]{\hbox{{$b_3$}}}}
\put(-1.15,0.05){\makebox(0,0)[br]{\hbox{{$c_1$}}}}
\put(0.3,0.15){\makebox(0,0)[br]{\hbox{{$c_2$}}}}
\put(1.3,0.15){\makebox(0,0)[br]{\hbox{{$c_3$}}}}
\put(-0.65,-0.8){\makebox(0,0)[br]{\hbox{{$d_1$}}}}
\put(0.3,-0.7){\makebox(0,0)[br]{\hbox{{$d_2$}}}}
\put(1.3,-0.7){\makebox(0,0)[br]{\hbox{{$d_3$}}}}
\psline[doubleline=true,linewidth=1pt, doublesep=1pt, linecolor=black]{->}(2.5,0.85)(4,0.85)
\put(3.25,1){\makebox(0,0)[bc]{\hbox{{$\mu_{a_1}\mu_{d_1}\mu_{c_2}\mu_{b_3}$}}}}
}
\end{pspicture}
\begin{pspicture}(-4,-1.5)(4,2){\psset{unit=1}
\multiput(-1,1.7)(1,-1.7){2}{\psline[linecolor=blue,linewidth=2pt]{->}(0.1,0)(.9,0)}
\multiput(-0.5,0.85)(1,-1.7){2}{\psline[linecolor=blue,linewidth=2pt]{<-}(0.1,0)(.9,0)}
\multiput(-1,1.7)(1,-1.7){2}{\psline[linecolor=blue,linewidth=2pt]{->}(-0.1,0)(-.9,0)}
\multiput(-0.5,0.85)(1,-1.7){2}{\psline[linecolor=blue,linewidth=2pt]{<-}(-0.1,0)(-.9,0)}
\multiput(-0.5,-.85)(0.5,0.85){3}{\psline[linecolor=blue,linewidth=2pt]{<-}(-0.05,0.085)(-.45,0.765)}
\multiput(-1.5,0.85)(3,-1.7){2}{\psline[linecolor=blue,linewidth=2pt]{<-}(-0.05,0.085)(-.45,0.765)}
\multiput(-1,0)(0.5,0.85){2}{\psline[linecolor=blue,linewidth=2pt]{->}(-0.05,0.085)(-.45,0.765)}
\multiput(0.5,-.85)(0.5,0.85){2}{\psline[linecolor=blue,linewidth=2pt]{->}(-0.05,0.085)(-.45,0.765)}
%
\psbezier[linecolor=blue,linewidth=2pt]{<-}(-1.07,1.75)(-4,4)(-1.5,-3.5)(0.45,-0.93)
\psbezier[linecolor=white,linewidth=3pt]{->}(-0.07,1.75)(-4.5,5)(-3.5,-3.5)(-0.55,-0.93)
\psbezier[linecolor=blue,linewidth=2pt]{->}(-0.07,1.75)(-4.5,5)(-3.5,-3.5)(-0.55,-0.93)
\psbezier[linecolor=white,linewidth=3pt]{->}(-2.07,1.65)(-4,0.5)(-1.5,-3.5)(1.45,-0.93)
\psbezier[linecolor=blue,linewidth=2pt]{->}(-2.07,1.65)(-4,0.5)(-1.5,-3.5)(1.45,-0.93)
\multiput(-3,1.7)(0.5,-0.85){4}{
\multiput(1,0)(1,0){3}{\pscircle[fillstyle=solid,fillcolor=lightgray]{.1}}}
\put(-2,1.9){\makebox(0,0)[bc]{\hbox{{$a_1$}}}}
\put(-1,1.9){\makebox(0,0)[bc]{\hbox{{$a_2$}}}}
\put(0,1.9){\makebox(0,0)[bc]{\hbox{{$a_3$}}}}
\put(-1.65,0.9){\makebox(0,0)[br]{\hbox{{$b_1$}}}}
\put(-0.7,1){\makebox(0,0)[br]{\hbox{{$b_2$}}}}
\put(0.3,1){\makebox(0,0)[br]{\hbox{{$b_3$}}}}
\put(-1.15,0.05){\makebox(0,0)[br]{\hbox{{$c_1$}}}}
\put(0.3,0.15){\makebox(0,0)[br]{\hbox{{$c_2$}}}}
\put(1.3,0.15){\makebox(0,0)[br]{\hbox{{$c_3$}}}}
\put(-0.65,-0.8){\makebox(0,0)[br]{\hbox{{$d_1$}}}}
\put(0.3,-0.7){\makebox(0,0)[br]{\hbox{{$d_2$}}}}
\put(1.3,-0.7){\makebox(0,0)[br]{\hbox{{$d_3$}}}}
}
\end{pspicture}
\caption{\small
The original $SL_4$: quiver on the left is transformed by the chain of mutations (from right to left) $\mu_{a_1} \mu_{d_1} \mu_{c_2} \mu_{b_3}:= \mu_{a_1d_1c_2b_3}$ to the symmetric quiver on the right. We use the same letters to denote both the original and transformed cluster variables hoping it will not lead to confusion.
}
\label{fi:SL4}
\end{figure}
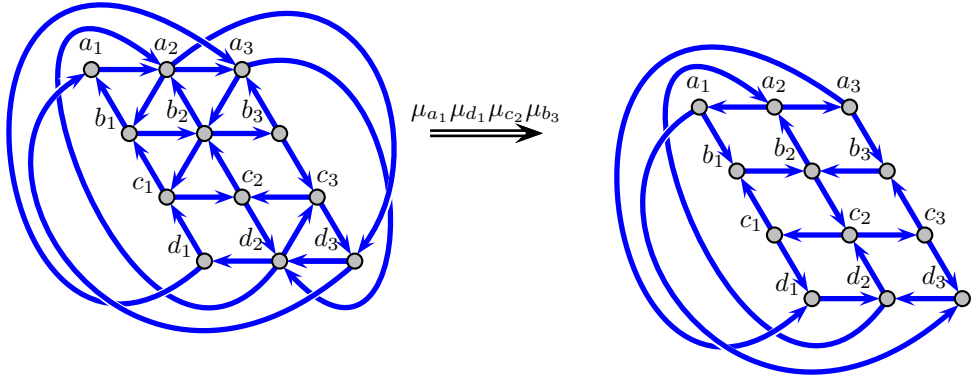

\begin{lemma}
Each $SL_n$ quiver can be transformed into a ``regular'' quiver with vertices $s_{i,j}$, $i=1,\dots,n-1$, $j=1,\dots n$ and arrows joining $s_{i,j}$ with $s_{i\pm 1,j}$ and $s_{i,j\pm 1}$, where we set $s_{i,n+1}=s_{n-i,1}$ and $s_{i,0}=s_{n-i,n}$ (the Moebius-type gluing); the directions of arrows alternate as shown in Fig.~\ref{fi:SL4}. A new vertex $a_t$ is of order $2n$ and is incident to all vertices $a_{1,j}$ and $a_{n-1,j}$ with arrows arrangement such that every vertex $a_{i,j}$ has order four and has exactly two incoming and two outgoing edges. The Casimir of the thus amended quiver is $a_t\prod_{i,j}a_{i,j}=1$.
\end{lemma}

\begin{proof} A sequence of mutations similar to the one shown on Figure~\ref{fi:SL4} proves the statement.
\end{proof}

We now address the problem of constructing a still missing geodesic function $G_B$, which, in particular, does not commute with $\mathfrak M_1$ and $\mathfrak M_2$. For this, we refer to the original $SL_4$ quiver in the left-hand side of Fig.~\ref{fi:SL4}. Mutations first at $a_3$ and then at $a_2$ produce the quiver in which the vertex $a_1$ has order one (Fig.~\ref{fi:SL4-1}, on the right). We then proceed by analogy with the $SL_3$ case adding a new vertex (with the variable $a_t$ associated) and joining it with only $a_1$ and $a_2$ (thus obtaining a "wing of a papillon"); we simultaneously introduce the Casimir $C=a_ta_1a_2^2a_3^2b_1\cdots d_3$ whose value is to be determined below from compatibility conditions.

\begin{figure}[H]
\begin{pspicture}(-3.5,-1.5)(3,3){\psset{unit=1}
%
\multiput(-2,1.7)(1,0){1}{\psline[linecolor=blue,linewidth=2pt]{<-}(0.1,0)(.9,0)}
\multiput(-1,1.7)(1,0){1}{\psline[linecolor=blue,linewidth=2pt]{->}(0.1,0)(.9,0)}
\multiput(-1.5,0.85)(0.5,-0.85){2}{\psline[linecolor=blue,linewidth=2pt]{->}(0.1,0)(.9,0)}
\multiput(0,0)(0.5,-0.85){2}{\psline[linecolor=blue,linewidth=2pt]{<-}(0.1,0)(.9,0)}
%
\multiput(0,1.7)(0.5,-0.85){3}{\psline[linecolor=blue,linewidth=2pt]{->}(0.05,-0.085)(.45,-0.765)}
\multiput(0,0)(0.5,-0.85){1}{\psline[linecolor=blue,linewidth=2pt]{->}(0.05,-0.085)(.45,-0.765)}
\multiput(-0.5,-.85)(-0.5,0.85){2}{\psline[linecolor=blue,linewidth=2pt]{->}(-0.05,0.085)(-.45,0.765)}
\multiput(0,0)(-0.5,0.85){1}{\psline[linecolor=blue,linewidth=2pt]{->}(-0.05,0.085)(-.45,0.765)}
\multiput(0.5,-0.85)(0.5,0.85){1}{\psline[linecolor=blue,linewidth=2pt]{->}(0.05,0.085)(.45,0.765)}
\multiput(-1.5,0.85)(1,0){2}{\psline[linecolor=blue,linewidth=2pt]{->}(0.05,0.085)(.45,0.765)}
\multiput(-0.5,0.85)(-0.5,-0.85){1}{\psline[linecolor=blue,linewidth=2pt]{->}(-0.05,-0.085)(-.45,-0.765)}
%
\psline[linecolor=white,linewidth=3pt]{->}(-0.09,1.67)(-1.41,0.89)
\psline[linecolor=blue,linewidth=2pt]{->}(-0.09,1.67)(-1.41,0.89)
\psline[linecolor=white,linewidth=3pt]{->}(-0.5,-0.75)(-0.5,0.75)
\psline[linecolor=blue,linewidth=2pt]{->}(-0.5,-0.75)(-0.5,0.75)
\psline[linecolor=white,linewidth=3pt]{<-}(0.5,-0.75)(0.5,0.75)
\psline[linecolor=blue,linewidth=2pt]{<-}(0.5,-0.75)(0.5,0.75)
%
\psbezier[linecolor=blue,linewidth=2pt]{<-}(0.07,1.75)(2.5,2.5)(2.5,-3)(0.55,-0.93)
\psbezier[linecolor=white,linewidth=3pt]{<-}(-0.93,1.77)(1.5,3.85)(2.7,0.8)(1.53,-0.76)
\psbezier[linecolor=blue,linewidth=2pt]{<-}(-0.93,1.77)(1.5,3.85)(2.7,0.8)(1.53,-0.76)
\psbezier[linecolor=white,linewidth=3pt]{->}(-0.07,1.75)(-4,5)(-3,-3.5)(-0.55,-0.93)
\psbezier[linecolor=blue,linewidth=2pt]{->}(-0.07,1.75)(-4,5)(-3,-3.5)(-0.55,-0.93)
\psbezier[linecolor=blue,linewidth=2pt]{->}(0.08,1.67)(1,1.2)(1.5,1)(1.5,-0.75)
\multiput(-3,1.7)(0.5,-0.85){4}{
\multiput(1,0)(1,0){3}{\pscircle[fillstyle=solid,fillcolor=lightgray]{.1}}}
\put(-2,1.9){\makebox(0,0)[bc]{\hbox{{$a_1$}}}}
\put(-1,1.9){\makebox(0,0)[bc]{\hbox{{$a_2$}}}}
\put(0,1.9){\makebox(0,0)[bc]{\hbox{{$a_3$}}}}
\put(-1.45,0.8){\makebox(0,0)[tr]{\hbox{{$b_1$}}}}
\put(-0.35,0.85){\makebox(0,0)[cl]{\hbox{{$b_2$}}}}
\put(0.65,0.85){\makebox(0,0)[cl]{\hbox{{$b_3$}}}}
\put(-1.15,0.05){\makebox(0,0)[br]{\hbox{{$c_1$}}}}
\put(0.3,0.15){\makebox(0,0)[br]{\hbox{{$c_2$}}}}
\put(1.3,0.15){\makebox(0,0)[br]{\hbox{{$c_3$}}}}
\put(-0.65,-0.8){\makebox(0,0)[br]{\hbox{{$d_1$}}}}
\put(0.3,-0.8){\makebox(0,0)[br]{\hbox{{$d_2$}}}}
\put(1.3,-0.75){\makebox(0,0)[br]{\hbox{{$d_3$}}}}
\psline[doubleline=true,linewidth=1pt, doublesep=1pt, linecolor=black]{->}(2.2,0.85)(3,0.85)
}
\end{pspicture}
\begin{pspicture}(-3,-1.5)(2.5,2){\psset{unit=1}
%
\multiput(-2,1.7)(1,0){1}{\psline[linecolor=blue,linewidth=2pt]{<-}(0.1,0)(.9,0)}
\multiput(-1,1.7)(1,0){1}{\psline[linecolor=blue,linewidth=2pt]{->}(0.1,0)(.9,0)}
\multiput(-1.5,0.85)(0.5,-0.85){2}{\psline[linecolor=blue,linewidth=2pt]{->}(0.1,0)(.9,0)}
\multiput(0,0)(0.5,-0.85){2}{\psline[linecolor=blue,linewidth=2pt]{<-}(0.1,0)(.9,0)}
%
\multiput(0,1.7)(0.5,-0.85){3}{\psline[linecolor=blue,linewidth=2pt]{->}(0.05,-0.085)(.45,-0.765)}
\multiput(0,0)(0.5,-0.85){1}{\psline[linecolor=blue,linewidth=2pt]{->}(0.05,-0.085)(.45,-0.765)}
\multiput(-0.5,-.85)(-0.5,0.85){2}{\psline[linecolor=blue,linewidth=2pt]{->}(-0.05,0.085)(-.45,0.765)}
\multiput(0,0)(-0.5,0.85){1}{\psline[linecolor=blue,linewidth=2pt]{->}(-0.05,0.085)(-.45,0.765)}
\multiput(0.5,-0.85)(0.5,0.85){1}{\psline[linecolor=blue,linewidth=2pt]{->}(0.05,0.085)(.45,0.765)}
\multiput(-1.5,0.85)(1,0){2}{\psline[linecolor=blue,linewidth=2pt]{->}(0.05,0.085)(.45,0.765)}
\multiput(-0.5,0.85)(-0.5,-0.85){1}{\psline[linecolor=blue,linewidth=2pt]{->}(-0.05,-0.085)(-.45,-0.765)}
%
\psline[linecolor=white,linewidth=3pt]{->}(-0.09,1.67)(-1.41,0.89)
\psline[linecolor=blue,linewidth=2pt]{->}(-0.09,1.67)(-1.41,0.89)
\psline[linecolor=white,linewidth=3pt]{->}(-0.5,-0.75)(-0.5,0.75)
\psline[linecolor=blue,linewidth=2pt]{->}(-0.5,-0.75)(-0.5,0.75)
\psline[linecolor=white,linewidth=3pt]{<-}(0.5,-0.75)(0.5,0.75)
\psline[linecolor=blue,linewidth=2pt]{<-}(0.5,-0.75)(0.5,0.75)
\rput(-1,0){\psline[linecolor=red,linewidth=2pt]{<-}(-0.09,1.67)(-1.41,0.89)}
\rput(-2,1.7){\psline[doubleline=true,linewidth=1pt, doublesep=1pt, linecolor=red]{->}(-0.05,-0.085)(-.45,-0.765)}
\rput(-2.5,0.85){\pscircle[fillstyle=solid,fillcolor=red]{.1}}
%
\psbezier[linecolor=blue,linewidth=2pt]{<-}(0.07,1.75)(2.5,2.5)(2.5,-3)(0.55,-0.93)
\psbezier[linecolor=white,linewidth=3pt]{<-}(-0.93,1.77)(1.5,3.85)(2.7,0.8)(1.53,-0.76)
\psbezier[linecolor=blue,linewidth=2pt]{<-}(-0.93,1.77)(1.5,3.85)(2.7,0.8)(1.53,-0.76)
\psbezier[linecolor=white,linewidth=3pt]{->}(-0.07,1.75)(-4.5,5)(-3,-3)(-0.55,-0.93)
\psbezier[linecolor=blue,linewidth=2pt]{->}(-0.07,1.75)(-4.5,5)(-3,-3)(-0.55,-0.93)
\psbezier[linecolor=blue,linewidth=2pt]{->}(0.08,1.67)(1,1.2)(1.5,1)(1.5,-0.75)
\multiput(-3,1.7)(0.5,-0.85){4}{
\multiput(1,0)(1,0){3}{\pscircle[fillstyle=solid,fillcolor=lightgray]{.1}}}
\put(-2,1.9){\makebox(0,0)[bc]{\hbox{{$a_1$}}}}
\put(-1,1.9){\makebox(0,0)[bc]{\hbox{{$a_2$}}}}
\put(0,1.9){\makebox(0,0)[bc]{\hbox{{$a_3$}}}}
\put(-1.45,0.8){\makebox(0,0)[tr]{\hbox{{$b_1$}}}}
\put(-0.35,0.85){\makebox(0,0)[cl]{\hbox{{$b_2$}}}}
\put(0.65,0.85){\makebox(0,0)[cl]{\hbox{{$b_3$}}}}
\put(-1.15,0.05){\makebox(0,0)[br]{\hbox{{$c_1$}}}}
\put(0.3,0.15){\makebox(0,0)[br]{\hbox{{$c_2$}}}}
\put(1.3,0.15){\makebox(0,0)[br]{\hbox{{$c_3$}}}}
\put(-0.65,-0.8){\makebox(0,0)[br]{\hbox{{$d_1$}}}}
\put(0.3,-0.8){\makebox(0,0)[br]{\hbox{{$d_2$}}}}
\put(1.3,-0.75){\makebox(0,0)[br]{\hbox{{$d_3$}}}}
\put(-2.5,0.7){\makebox(0,0)[tc]{\hbox{{$a_t$}}}}
\psline[doubleline=true,linewidth=1pt, doublesep=1pt, linecolor=black]{->}(-2,-1.6)(-3,-2.6)
\put(-2.6,-2){\makebox(0,0)[br]{\hbox{{$ \mu_{a_3}\mu_{a_2}$}}}}
}
\end{pspicture}
\\
\begin{pspicture}(-5,-2)(4.2,4){\psset{unit=1}
%
\multiput(-2,1.7)(1,0){2}{\psline[linecolor=blue,linewidth=2pt]{->}(0.1,0)(.9,0)}
\multiput(-1.5,0.85)(0.5,-0.85){2}{\psline[linecolor=blue,linewidth=2pt]{->}(0.1,0)(.9,0)}
\multiput(-0.5,0.85)(1,0){1}{\psline[linecolor=blue,linewidth=2pt]{->}(0.1,0)(.9,0)}
\multiput(1,0)(0.5,-0.85){2}{\psline[linecolor=blue,linewidth=2pt]{->}(-0.1,0)(-.9,0)}
\multiput(1.5,-0.85)(-1,0){2}{\psline[linecolor=blue,linewidth=2pt]{->}(-0.1,0)(-.9,0)}
%
\multiput(0.5,.85)(0.5,-0.85){2}{\psline[linecolor=blue,linewidth=2pt]{->}(0.05,-0.085)(.45,-0.765)}
\multiput(0,0)(0.5,-0.85){1}{\psline[linecolor=blue,linewidth=2pt]{->}(0.05,-0.085)(.45,-0.765)}
\multiput(-0.5,-.85)(-0.5,0.85){3}{\psline[linecolor=blue,linewidth=2pt]{->}(-0.05,0.085)(-.45,0.765)}
\multiput(0,0)(-0.5,0.85){2}{\psline[linecolor=blue,linewidth=2pt]{->}(-0.05,0.085)(-.45,0.765)}
\multiput(0.5,.85)(-0.5,0.85){1}{\psline[linecolor=blue,linewidth=2pt]{->}(-0.05,0.085)(-.45,0.765)}
\multiput(0.5,-0.85)(0.5,0.85){1}{\psline[linecolor=blue,linewidth=2pt]{->}(0.05,0.085)(.45,0.765)}
\multiput(0,1.7)(-0.5,-0.85){2}{\psline[linecolor=blue,linewidth=2pt]{->}(-0.05,-0.085)(-.45,-0.765)}
\multiput(-1,1.7)(-0.5,-0.85){1}{\psline[linecolor=blue,linewidth=2pt]{->}(-0.05,-0.085)(-.45,-0.765)}
%
\psbezier[linecolor=blue,linewidth=2pt]{->}(0.07,1.77)(2,4)(3,-3.5)(0.55,-0.93)
\psbezier[linecolor=white,linewidth=3pt]{->}(-1,1.7)(1.5,3.85)(2.7,0.8)(1.53,-0.76)
\psbezier[linecolor=blue,linewidth=2pt]{->}(-1,1.7)(1.5,3.85)(2.7,0.8)(1.53,-0.76)
\psbezier[linecolor=blue,linewidth=2pt]{<-}(-1.07,1.75)(-4,4)(-1.5,-3.5)(0.45,-0.93)
\psbezier[linecolor=white,linewidth=3pt]{<-}(-0.07,1.75)(-4.5,5)(-3.5,-3.5)(-0.55,-0.93)
\psbezier[linecolor=blue,linewidth=2pt]{<-}(-0.07,1.75)(-4.5,5)(-3.5,-3.5)(-0.55,-0.93)
\psbezier[linecolor=white,linewidth=3pt]{<-}(-2.07,1.65)(-4,0.5)(-1.5,-3.5)(1.45,-0.93)
\psbezier[linecolor=blue,linewidth=2pt]{<-}(-2.07,1.65)(-4,0.5)(-1.5,-3.5)(1.45,-0.93)
\multiput(-3,1.7)(0.5,-0.85){4}{
\multiput(1,0)(1,0){3}{\pscircle[fillstyle=solid,fillcolor=lightgray]{.1}}}
\rput(0,1.7){\psline[linecolor=red,linewidth=2pt]{->}(0.1,0)(0.9,0)}
\rput(1,1.7){\psline[linecolor=red,linewidth=2pt]{->}(-0.05,-0.085)(-.45,-0.765)}
\rput(1,1.7){\pscircle[fillstyle=solid,fillcolor=red]{.1}} 
\psbezier[linecolor=white,linewidth=3pt]{->}(-1.93,1.77)(-1,2.5)(0,2.8)(0.93,1.77)
\psbezier[linecolor=red,linewidth=2pt]{->}(-1.93,1.77)(-1,2.5)(0,2.8)(0.93,1.77)
\psbezier[linecolor=red,linewidth=2pt]{<-}(-0.43,-0.92)(2,-2.5)(2.5,-0.85)(1.07,1.63)
\put(-2,1.9){\makebox(0,0)[bc]{\hbox{{$a_1$}}}}
\put(-1,1.9){\makebox(0,0)[bc]{\hbox{{$a_2$}}}}
\put(0,1.9){\makebox(0,0)[bc]{\hbox{{$a_3$}}}}
\put(-1.65,0.9){\makebox(0,0)[br]{\hbox{{$b_1$}}}}
\put(-0.7,1){\makebox(0,0)[br]{\hbox{{$b_2$}}}}
\put(0.3,1){\makebox(0,0)[br]{\hbox{{$b_3$}}}}
\put(-1.15,0.05){\makebox(0,0)[br]{\hbox{{$c_1$}}}}
\put(0.3,0.15){\makebox(0,0)[br]{\hbox{{$c_2$}}}}
\put(1.3,0.15){\makebox(0,0)[br]{\hbox{{$c_3$}}}}
\put(-0.65,-0.8){\makebox(0,0)[br]{\hbox{{$d_1$}}}}
\put(0.3,-0.75){\makebox(0,0)[br]{\hbox{{$d_2$}}}}
\put(1.3,-0.75){\makebox(0,0)[br]{\hbox{{$d_3$}}}}
\put(1,1.4){\makebox(0,0)[tc]{\hbox{{$a_t$}}}}
\psline[doubleline=true,linewidth=1pt, doublesep=1pt, linecolor=black]{->}(2.2,0.85)(3.8,0.85)
\put(3,1){\makebox(0,0)[bc]{\hbox{{$ \mu_{a_1}\mu_{d_1}\mu_{c_2}\mu_{b_3}$}}}}
}
\end{pspicture}
\begin{pspicture}(-3,-2)(2,4){\psset{unit=1}
\multiput(-1,1.7)(1,-1.7){2}{\psline[linecolor=blue,linewidth=2pt]{->}(0.1,0)(.9,0)}
\multiput(-0.5,0.85)(1,-1.7){2}{\psline[linecolor=blue,linewidth=2pt]{<-}(0.1,0)(.9,0)}
\multiput(-1,1.7)(1,-1.7){2}{\psline[linecolor=blue,linewidth=2pt]{->}(-0.1,0)(-.9,0)}
\multiput(-0.5,0.85)(1,-1.7){2}{\psline[linecolor=blue,linewidth=2pt]{<-}(-0.1,0)(-.9,0)}
\multiput(-0.5,-.85)(0.5,0.85){3}{\psline[linecolor=blue,linewidth=2pt]{<-}(-0.05,0.085)(-.45,0.765)}
\multiput(-1.5,0.85)(3,-1.7){2}{\psline[linecolor=blue,linewidth=2pt]{<-}(-0.05,0.085)(-.45,0.765)}
\multiput(-1,0)(0.5,0.85){2}{\psline[linecolor=blue,linewidth=2pt]{->}(-0.05,0.085)(-.45,0.765)}
\multiput(0.5,-.85)(0.5,0.85){2}{\psline[linecolor=blue,linewidth=2pt]{->}(-0.05,0.085)(-.45,0.765)}
%
\psbezier[linecolor=blue,linewidth=2pt]{<-}(-1.07,1.75)(-4,4)(-1.5,-3.5)(0.45,-0.93)
\psbezier[linecolor=white,linewidth=3pt]{->}(-0.07,1.75)(-4.5,5)(-3.5,-3.5)(-0.55,-0.93)
\psbezier[linecolor=blue,linewidth=2pt]{->}(-0.07,1.75)(-4.5,5)(-3.5,-3.5)(-0.55,-0.93)
\psbezier[linecolor=white,linewidth=3pt]{->}(-2.07,1.65)(-4,0.5)(-1.5,-3.5)(1.45,-0.93)
\psbezier[linecolor=blue,linewidth=2pt]{->}(-2.07,1.65)(-4,0.5)(-1.5,-3.5)(1.45,-0.93)
\multiput(-3,1.7)(0.5,-0.85){4}{
\multiput(1,0)(1,0){3}{\pscircle[fillstyle=solid,fillcolor=lightgray]{.1}}}
%
%
\rput(1,0){\psline[linecolor=red,linewidth=2pt]{<-}(0,0.1)(-0.02,1.6)}
\rput(0,1.7){\psline[linecolor=red,linewidth=2pt]{<-}(0.1,0)(0.9,0)}
\rput(1,1.7){\psline[linecolor=red,linewidth=2pt]{<-}(-0.05,-0.085)(-.45,-0.765)}
\rput(1,-0.85){\psline[linecolor=red,linewidth=2pt]{->}(0.5,0.1)(0.06,2.45)}
\rput(1,1.7){\pscircle[fillstyle=solid,fillcolor=red]{.1}} 
\psbezier[linecolor=white,linewidth=3pt]{<-}(-1.93,1.77)(-1,2.5)(0,2.8)(0.93,1.77)
\psbezier[linecolor=red,linewidth=2pt]{<-}(-1.93,1.77)(-1,2.5)(0,2.8)(0.93,1.77)
\psbezier[linecolor=white,linewidth=3pt]{->}(-0.43,-0.92)(2,-3)(3,0)(1.07,1.63)
\psbezier[linecolor=red,linewidth=2pt]{->}(-0.43,-0.92)(2,-3)(3,0)(1.07,1.63)
\psbezier[linecolor=white,linewidth=3pt]{->}(-1.6,0.89)(-4.5,3)(0,4.5)(0.98,1.8)
\psbezier[linecolor=red,linewidth=2pt]{->}(-1.6,0.89)(-4.5,3)(0,4.5)(0.98,1.8)
\psbezier[linecolor=white,linewidth=3pt]{<-}(-1,-0.1)(-1,-4)(7,-1.5)(1.12,1.7)
\psbezier[linecolor=red,linewidth=2pt]{<-}(-1,-0.1)(-1,-4)(7,-1.5)(1.12,1.7)
\put(-2,1.9){\makebox(0,0)[bc]{\hbox{{$a_1$}}}}
\put(-1,1.9){\makebox(0,0)[bc]{\hbox{{$a_2$}}}}
\put(0,1.9){\makebox(0,0)[bc]{\hbox{{$a_3$}}}}
\put(-1.5,1){\makebox(0,0)[bl]{\hbox{{$b_1$}}}}
\put(-0.7,1){\makebox(0,0)[br]{\hbox{{$b_2$}}}}
\put(0.3,1){\makebox(0,0)[br]{\hbox{{$b_3$}}}}
\put(-1.15,0.05){\makebox(0,0)[br]{\hbox{{$c_1$}}}}
\put(0.3,0.15){\makebox(0,0)[br]{\hbox{{$c_2$}}}}
\put(0.8,0.15){\makebox(0,0)[br]{\hbox{{$c_3$}}}}
\put(-0.5,-0.75){\makebox(0,0)[bl]{\hbox{{$d_1$}}}}
\put(0.3,-0.75){\makebox(0,0)[br]{\hbox{{$d_2$}}}}
\put(1.3,-0.75){\makebox(0,0)[br]{\hbox{{$d_3$}}}}
\put(1.1,1.8){\makebox(0,0)[bl]{\hbox{{$a_t$}}}}
}
\end{pspicture}
\caption{\small
The mutations at $a_3$ and then at $a_2$ brings the original quivers for $SL_4$ into the form depicted on the left; in the middle picture we add the new vertex $a_t$ together with two directed edges; the new geodesic function, which we identify with $G_B$, is $\langle a_1a_t \rangle$. Transforming this quiver back to the original one by performing mutations at $a_2$ and then at $a_3$, we obtain the original quiver with the vertex $a_t$ of order four added; note that with the added edges, the number of incoming and outgoing edges is now the same at all thirteen vertices of the quiver. If we transform the original quiver into a ``symmetric'' one by performing mutations at $d_1$, $c_2$, $b_3$ and then in $a_1$, we obtain the quiver with the vertex $a_t$  of order eight and with all other vertices of order four (it is the last quiver in the chain of transformations in the figure).
}
\label{fi:SL4-1}
\end{figure}
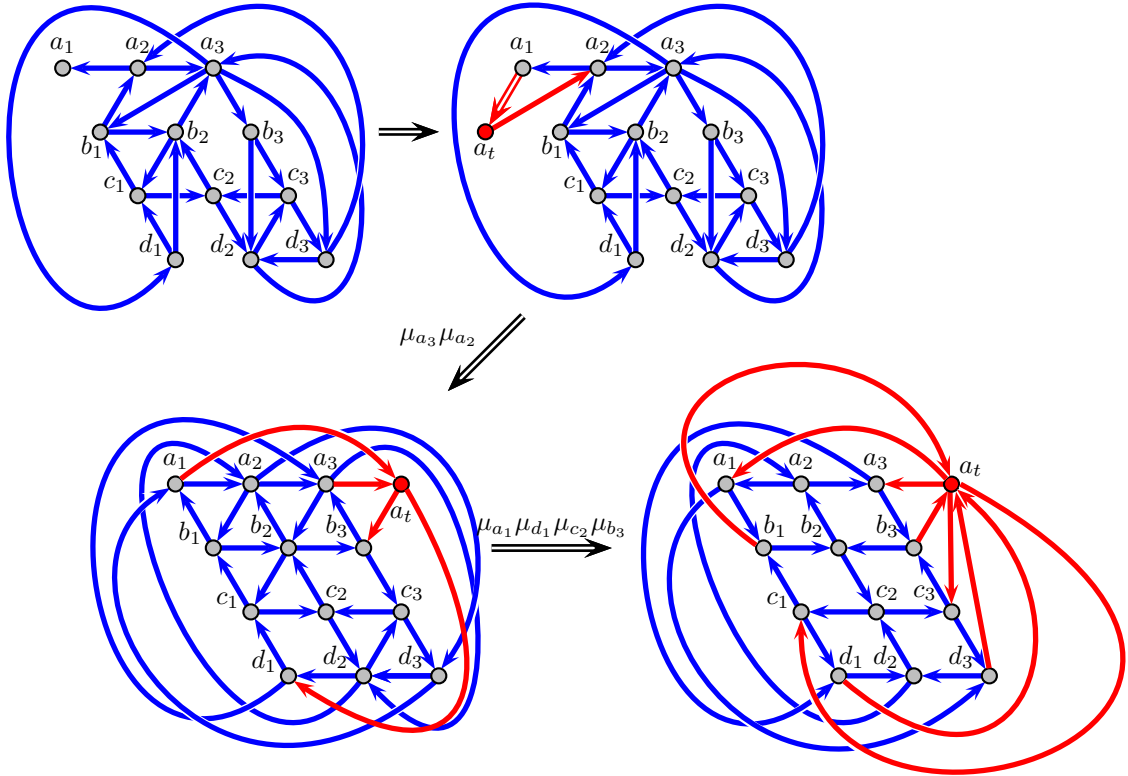

The new geodesic function, which we associate with the missing cycle $B$, in the transformed quiver (the second picture in Fig.~\ref{fi:SL4-1} ) reads
$$
G_B=\langle a_1a_t \rangle=(a_1a_t)^{1/2}\left( 1+ \frac{1}{a_1}+\frac{1}{a_1a_t}\right).
$$
It is easy to see that $G_B$ commutes with $G_{1,2}$ and $G_{2,3}$. For this, let us examine geodesic functions for the upper-right quiver in Fig.~\ref{fi:SL4-1}):
\begin{align*}
&\mu_{a_3}  \mu_{a_2} G_{1,2}=\langle b_3d_2c_3 \rangle,\quad &\mu_{a_3}  \mu_{a_2}\widetilde G_{1,2}=\langle d_1b_2c_1 \rangle, \\
& \mu_{a_3}  \mu_{a_2} G_{2,3}=\langle c_3c_2b_2a_3d_3 \rangle,\quad  &\mu_{a_3}  \mu_{a_2}\widetilde G_{2,3}=\langle c_1c_2d_2a_3b_1 \rangle, \\
& \mu_{a_3}  \mu_{a_2} G_{3,4}=\langle d_3a_2d_2a_3d_1c_1b_1a_2a_1 \rangle,\quad  &\mu_{a_3}  \mu_{a_2}\widetilde G_{3,4}=\langle b_1(b_2a_2)a_3b_3c_3d_3a_2a_1 \rangle.
\end{align*}
The first two expressions do not contain variables $a_1$ and $a_2$ and therefore $G_B=\langle a_1a_t \rangle$ commutes with the transformed $G_{1,2}$ and $G_{2,3}$ and  $\widetilde G_{1,2}$ and $\widetilde G_{2,3}$ (and does not commute with $G_{3,4}$ and $\widetilde G_{3,4}$).

Transforming back to the original $SL_4$ quiver, we obtain the third quiver in chain in Fig.~\ref{fi:SL4-1}: the added vertex $a_t$ becomes of order four, and the numbers of incoming and outgoing edges coincide in every vertex, so the Casimir is just the product of all variables taken in power one:
\be\label{eq:CasN4}
C=a_ta_1a_2a_3b_1b_2b_3c_1c_2c_3d_1d_2d_3.
\ee
In the amended original $SL_4$ quiver, $G_B$ becomes
\be
\label{GB-4}
G_B=\langle a_1a_2a_3a_t\rangle.
\ee

If we proceed, by performing mutations at $d_1$, $c_2$, $b_3$ and $a_1$, we obtain the last in the chain ``symmetric'' quiver in Fig.~\ref{fi:SL4-1}. The vertex $a_t$ then becomes of order eight, all other vertices will be of order four, and the $G_B$ geodesic function becomes $\langle a_2a_3(d_1b_3)a_ta_1 \rangle$, where $\dots (xy)\dots$ indicates that we have terms $\dots\bigl(1+\frac1x+\frac1y+\frac1{xy}\bigr)\dots $ in the corresponding expression. For example, this expression for $G_B$ contains eight terms:
$$
G_B=(a_2a_3d_1b_3a_ta_1)^{1/2}\left[1+\frac{1}{a_2}+\frac{1}{a_2a_3}+\frac{1}{a_2a_3d_1}+\frac{1}{a_2a_3b_3}+\frac{1}{a_2a_3d_1b_3}+\frac{1}{a_2a_3d_1b_3a_t}+\frac{1}{a_2a_3d_1b_3a_ta_1} \right]
$$

We hope to identify the constructed quantities 
\be\label{eq:G-sequence}\{G_{12},G_{23},G_{34},G_B,\widetilde G_{34},\widetilde G_{23},\widetilde G_{12}\}\ee
 with geodesic functions associated with loops on genus three surface (see Figure~\ref{fi:genus-234}). Note that these loops form a chain with intersection numbers one for two consecutive loops and zero otherwise. We construct a matrix $U\in \A_8$
 with $U_{i,i+1},\ i=[1,7]$ given by the corresponding element of sequence ~(\ref{eq:G-sequence}) and all other terms computed using the skein and Poisson  relations 
 $$
 U_{i,j}=\frac{1}{2}U_{i,k}U_{k,j}+\{U_{i,k},U_{k,j}\}, \text{ here } i<k<j.
$$
 
To identify entries $U_{i,j}$ with geodesic functions of the corresponding loops on the genus three surface, the matrix $U$ must satisfy the rank condition: 
$\operatorname{rank}(U+U^{\text{T}})\le 4$. Symbolic computation on Maple demonstrates that the rank condition holds only for the Casimir~(\ref{eq:CasN4})  $C=-1$.


\subsection{$\mathbb Z_2$ symmetry}\label{ss:Z2}
We now address the topic mostly overseen in the literature on the braid-group transformations in $\mathcal A_n$: do these transformations generate the full modular group of $\mathcal T_{g,s}$ ($s=1,2$)? Alternatively, do the twists over $G_{i,i+1}$, $\widetilde G_{i,i+1}$, and $G_B$ generate the full modular group of $\mathcal T_{n-1,0}$? The answer is affirmative for $\mathcal T_{1,1}$ and $\mathcal T_{2,0}$, but for higher genera we have a problem of {\sl $\mathbb Z_2$ symmetry}: indeed, leave for a moment metric structures on $\Sigma_{g,s}$ aside and consider only topology of  $\Sigma_{n-1,0}$ and homotopy types of closed paths on this surface. 

Let us examine the structure of closed paths in Fig.~\ref{fi:genus-234}. Note that all paths corresponding to $G_{i,i+1}$, $\widetilde G_{i,i+1}$, and $G_B$ enjoy $\mathbb Z_2$ symmetry w.r.t. rotation by $180^o$ (or reflection) w.r.t. the vertical axis. It follows immediately that if we perform any binary operation of summation, multiplication, or the Poisson bracket on $\mathbb Z_2$-symmetric objects, then the result will be $\mathbb Z_2$-symmetric itself. This means that we can never obtain a single geodesic function corresponding to a non-$\mathbb Z_2$ symmetric curve (say, either $\mathfrak M_1$ or $\mathfrak M_2$ in the $n=4$ case) in the right-hand side of these relations; all such functions will arise only in $\mathbb Z_2$-symmetric combinations (like $\mathfrak M_1+\mathfrak M_2$ or $\mathfrak M_1\mathfrak M_2$).


Worth mentioning is that the genus two case is free of this subtlety   as all closed curves are $\mathbb Z_2$-symmetric (which reflects the well-known fact that genus two curves admit a hyperelliptic uniformization).

Note that the same problem of $\mathbb Z_2$ symmetry exists also in the case of $\mathcal A_n$ algebras with $n\ge 5$: if we do not restrict to geometrical Poisson leaves related to ideal triangulations of the  corresponding Riemann surfaces $\Sigma_{g,s}$, then expressions for non-$\mathbb Z_2$-symmetric elements of the algebra generally cease to be Laurent  polynomial.

{
\begin{conjecture}
The cluster algebra constructed in Section~\ref{ss:genus3} for the extended amalgamated quiver for $GL_4$  gives a description of the $\mathbb{Z}_2$-quotient of the Teichmuller space ${\mathcal T}_{3,0}$ of flat $SL_2(\mathbb C)$-connections on a smooth Riemann surface of genus three in such a way that
\begin{itemize}
\item[{(A)}] to every $Z_2$-invariant geodesic $\gamma$ we set into the correspondence a geodesic function $G_\gamma\in Z_+[a_1^{\pm 1/2},\dots, a_t^{\pm 1/2}]$ that is an element of an upper cluster algebra.
\item[{(B)}] all thus constructed $G_\gamma$ satisfy skein relations and Poisson Goldman bracket.
\item [{(C)}] in order to satisfy the rank condition, which is a necessary condition for $U_{i,j}=\tr M_i^{-1}M_j$ wtih $M_i\in SL_2(\mathbb C)$, we have to set Casimir $C=-1$.
\end{itemize}
\end{conjecture}
}

\subsection{Braid group transformations for $n>3$}


Consider a subgraph in the ``symmetric'' quiver: it goes along a ``zig-zag'' path passing through $2n-2$ vertices such that the first two vertices $a_1$ and $a_2$ and the last two vertices $a_{2n-3}$ and $a_{2n-2}$ are on the outer boundary of the quiver. We also include the vertex $a_t$ and all vertices $b_1,\dots, b_{2n-4}$ adjacent to the vertices $a_i$; in the symmetric quiver (say, in Fig.~\ref{fi:SL4}), these vertices are situated alternatively on the right and on the left from the ``zig-zag'' path. The pattern that we have looks as follows:
$$
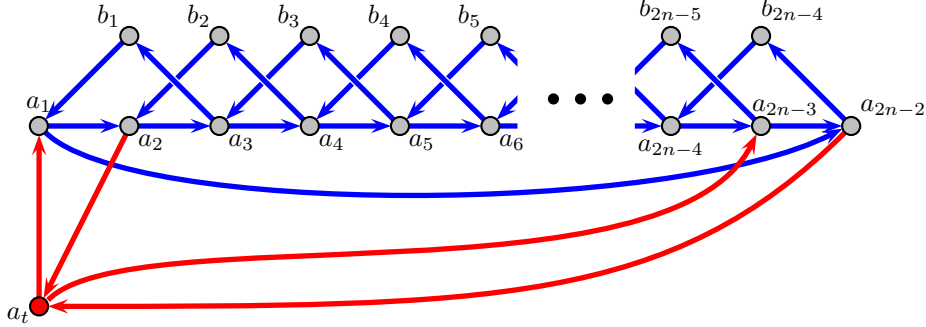
\begin{pspicture}(-5,-2.5)(5,1.5){\psset{unit=1.2}
\multiput(-5,0)(1,0){9}{\psline[linecolor=blue,linewidth=2pt]{->}(0.1,0)(.9,0)}
\multiput(-5,0)(1,0){8}{\psline[linecolor=blue,linewidth=2pt]{<-}(0.07,0.07)(.93,0.93)}
\multiput(-3,0)(1,0){8}{\psline[linecolor=white,linewidth=3pt]{->}(-0.07,0.07)(-0.93,0.93)\psline[linecolor=blue,linewidth=2pt]{->}(-0.07,0.07)(-0.93,0.93)}
\psbezier[linecolor=blue,linewidth=2pt]{->}(-4.93,-0.07)(-4,-1)(2,-1)(3.9,-0.03)
\psline[linecolor=red,linewidth=2pt]{->}(-5,-1.9)(-5,-0.1)
\psline[linecolor=red,linewidth=2pt]{<-}(-4.95,-1.91)(-4.03,-0.08)
\psbezier[linecolor=red,linewidth=2pt]{->}(-4.9,-1.93)(-4,-1)(2,-2)(2.95,-0.09)
\psbezier[linecolor=red,linewidth=2pt]{<-}(-4.9,-2)(0,-2)(2,-2)(3.93,-0.07)
%
\multiput(-5,0)(1,0){10}{\pscircle[fillstyle=solid,fillcolor=lightgray]{.1}}
\multiput(-4,1)(1,0){8}{\pscircle[fillstyle=solid,fillcolor=lightgray]{.1}}
\rput(-5,-2){\pscircle[fillstyle=solid,fillcolor=red]{.1}}
\psframe[linecolor=white, fillstyle=solid, fillcolor=white](0.3,1.4)(1.6,-0.3)
\multiput(0.7,0.3)(0.3,0){3}{\pscircle[fillstyle=solid,fillcolor=black]{.05}}
\put(-5,0.15){\makebox(0,0)[bc]{\hbox{{$a_1$}}}}
\put(-3.9,-0.1){\makebox(0,0)[tl]{\hbox{{$a_2$}}}}
\put(-2.9,-0.1){\makebox(0,0)[tl]{\hbox{{$a_3$}}}}
\put(-1.9,-0.1){\makebox(0,0)[tl]{\hbox{{$a_4$}}}}
\put(-0.9,-0.1){\makebox(0,0)[tl]{\hbox{{$a_5$}}}}
\put(0.1,-0.1){\makebox(0,0)[tl]{\hbox{{$a_6$}}}}
\put(2,-0.15){\makebox(0,0)[tc]{\hbox{{$a_{2n{-}4}$}}}}
\put(2.9,0.1){\makebox(0,0)[bl]{\hbox{{$a_{2n{-}3}$}}}}
\put(4.1,0.1){\makebox(0,0)[bl]{\hbox{{$a_{2n{-}2}$}}}}
\put(-4.1,1.1){\makebox(0,0)[br]{\hbox{{$b_1$}}}}
\put(-3.1,1.1){\makebox(0,0)[br]{\hbox{{$b_2$}}}}
\put(-2.1,1.1){\makebox(0,0)[br]{\hbox{{$b_3$}}}}
\put(-1.1,1.1){\makebox(0,0)[br]{\hbox{{$b_4$}}}}
\put(-0.1,1.1){\makebox(0,0)[br]{\hbox{{$b_5$}}}}
\put(2,1.15){\makebox(0,0)[bc]{\hbox{{$b_{2n{-}5}$}}}}
\put(3,1.15){\makebox(0,0)[bl]{\hbox{{$b_{2n{-}4}$}}}}
\put(-5.1,-2.1){\makebox(0,0)[cr]{\hbox{{$a_{t}$}}}}
}
\end{pspicture}
$$
To this subgraph we set into the correspondence a matrix element (a geodesic function) $\langle  a_1a_2\cdots a_{2n-3}a_{2n-2} \rangle$. 
We now perform mutations consecutively at vertices $a_2$, $a_3$, $\dots$, $a_{2n-4}$, $a_{2n-3}$. The resulting subgraph is  
$$
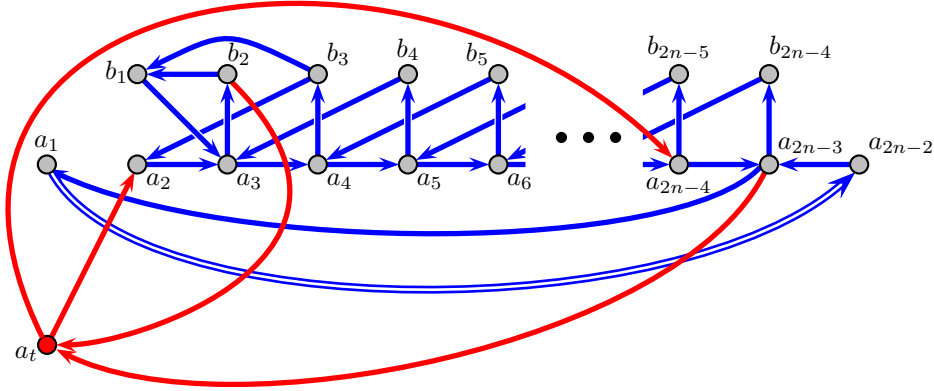
\begin{pspicture}(-5,-3)(5,2){\psset{unit=1.2}
\multiput(-4,0)(1,0){7}{\psline[linecolor=blue,linewidth=2pt]{->}(0.1,0)(.9,0)}
\psline[linecolor=blue,linewidth=2pt]{<-}(3.1,0)(3.9,0)
\multiput(-4,0)(1,0){6}{\psline[linecolor=blue,linewidth=2pt]{<-}(0.09,0.05)(1.91,0.95)}
\multiput(-3,0)(1,0){7}{\psline[linecolor=white,linewidth=3pt]{->}(0,0.1)(0,0.9)\psline[linecolor=blue,linewidth=2pt]{->}(0,0.1)(0,0.9)}
\psline[linecolor=blue,linewidth=2pt]{<-}(-3.9,1)(-3.1,1)
\psline[linecolor=white,linewidth=3pt]{->}(-3.93,0.93)(-3.07,0.07)
\psline[linecolor=blue,linewidth=2pt]{->}(-3.93,0.93)(-3.07,0.07)
\psbezier[linecolor=blue,linewidth=2pt]{<-}(-3.91,1.05)(-3,1.5)(-3,1.5)(-2.09,1.05)
\psbezier[linecolor=blue,linewidth=2pt]{<-}(-4.99,-0.05)(-3,-1)(2,-1)(2.9,-0.03)
\psbezier[doubleline=true,linewidth=1pt, doublesep=1pt, linecolor=blue]{->}(-4.95,-0.09)(-4,-1.8)(2,-1.8)(3.95,-0.09)
\psline[linecolor=red,linewidth=2pt]{->}(-4.95,-1.91)(-4.03,-0.08)
\psbezier[linecolor=red,linewidth=2pt]{<-}(-4.9,-2.05)(-3,-3)(2,-2)(2.95,-0.09)
\psbezier[linecolor=red,linewidth=2pt]{->}(-2.97,0.93)(-1,-1)(-4,-2)(-4.9,-2)
%
\multiput(-5,0)(1,0){10}{\pscircle[fillstyle=solid,fillcolor=lightgray]{.1}}
\multiput(-4,1)(1,0){8}{\pscircle[fillstyle=solid,fillcolor=lightgray]{.1}}
\rput(-5,-2){\pscircle[fillstyle=solid,fillcolor=red]{.1}}
\psframe[linecolor=white, fillstyle=solid, fillcolor=white](0.3,1.4)(1.6,-0.3)
\multiput(0.7,0.3)(0.3,0){3}{\pscircle[fillstyle=solid,fillcolor=black]{.05}}
\psbezier[linecolor=red,linewidth=2pt]{->}(-5.05,-1.92)(-7,2)(-1,3)(1.93,0.07)
\put(-5,0.15){\makebox(0,0)[bc]{\hbox{{$a_1$}}}}
\put(-3.9,-0.1){\makebox(0,0)[tl]{\hbox{{$a_2$}}}}
\put(-2.9,-0.1){\makebox(0,0)[tl]{\hbox{{$a_3$}}}}
\put(-1.9,-0.1){\makebox(0,0)[tl]{\hbox{{$a_4$}}}}
\put(-0.9,-0.1){\makebox(0,0)[tl]{\hbox{{$a_5$}}}}
\put(0.1,-0.1){\makebox(0,0)[tl]{\hbox{{$a_6$}}}}
\put(2,-0.15){\makebox(0,0)[tc]{\hbox{{$a_{2n{-}4}$}}}}
\put(3.1,0.1){\makebox(0,0)[bl]{\hbox{{$a_{2n{-}3}$}}}}
\put(4.1,0.1){\makebox(0,0)[bl]{\hbox{{$a_{2n{-}2}$}}}}
\put(-4.1,0.9){\makebox(0,0)[br]{\hbox{{$b_1$}}}}
\put(-3,1.1){\makebox(0,0)[bl]{\hbox{{$b_2$}}}}
\put(-1.9,1.1){\makebox(0,0)[bl]{\hbox{{$b_3$}}}}
\put(-1,1.15){\makebox(0,0)[bc]{\hbox{{$b_4$}}}}
\put(-0.1,1.1){\makebox(0,0)[br]{\hbox{{$b_5$}}}}
\put(2,1.15){\makebox(0,0)[bc]{\hbox{{$b_{2n{-}5}$}}}}
\put(3,1.15){\makebox(0,0)[bl]{\hbox{{$b_{2n{-}4}$}}}}
\put(-5.1,-2.1){\makebox(0,0)[cr]{\hbox{{$a_{t}$}}}}
}
\end{pspicture}
$$
and it contains a ``papillon wing'' subgraph with vertices $a_1$, $a_{2n-2}$, and $a_{2n-3}$. Upon the same chain of mutations, the matrix element $\langle a_1\cdots a_{2n-2}. \rangle$ becomes just $\langle a_1a_{2n-2} \rangle$. Then, mutation at $a_1$ or $a_{2n-2}$ results in the braid-group transformation (direct or inverse). So the chain of mutations that correspond to this twist in the ``symmetric'' quiver is
\be\label{twst}
\mu_{a_2}\mu_{a_3}\cdots \mu_{a_{2n-3}}\mu_{a_{2n-2}} \mu_{a_{2n-3}}\cdots  \mu_{a_3}\mu_{a_2},
\ee
or a similar chain with replacing $\mu_{a_{2n-2}}$ by $\mu_{a_1}$. 


After the sequence of mutations (\ref{twst}) we obtain the same quiver in which the transformed variables $a'_i$ and $b'_j$ are
\begin{align}
&a'_i=a_i\frac{\eta_{i+1}}{\eta_{i-1}},\ i=2,\dots,2n-3,\quad a'_1=\frac{\eta_2}{a_{2n-2}a_{2n-3}\cdots a_2},\quad a'_{2n-2}=\frac{a_1a_2\cdots a_{2n-3}a^2_{2n-2}}{\eta_1\eta_{2n-3}},\label{a-t}\\
&b'_j=b_j\frac{\eta_{j}}{\eta_{j+2}},\ j=1,\dots, 2n-4 \quad \left( a'_t=a_t\frac{\eta_1\eta_{2n-3}}{\eta_2a_1a_{2n-2}} \right)\label{b-t},
\end{align}
where
\begin{align}
\eta_{2n-2}&=1,\nonumber\\
\eta_{2n-3}&=1+a_{2n-2},\nonumber\\
\eta_{2n-4}&=1+a_{2n-2}+a_{2n-2}a_{2n-3},\label{eta}\\
&\vdots \nonumber\\
\eta_1&=1+a_{2n-2}+a_{2n-2}a_{2n-3}+\cdots + a_{2n-2}a_{2n-3}\cdots a_3a_2,\nonumber
\end{align}
(note the absence of $a_1$ in formulas for $\eta_i$.) 

\begin{lemma}\label{lm:twst}
Transformations (\ref{a-t}), (\ref{b-t}) with $\eta_k$ defined by (\ref{eta}) correspond to the braid-group transformation with the matrix element $G_{i,i+1}=\langle a_1 a_2\cdots a_{2n-2}\rangle$.
\end{lemma}

\proof
We begin with demonstrating that $G_{i,i+1}$ is preserved under transformation (\ref{a-t}):
\begin{align*}
G'_{i,i+1}= &(a'_1\cdots a'_{2n-2})^{1/2}\left(1+\frac 1{a'_1}+\frac 1{a'_1a'_2}+\cdots +\frac 1{a'_1a'_2\cdots a'_{2n-3}}+\frac 1{a'_1a'_2\cdots a'_{2n-3}a'_{2n-2}} \right)\\
=&(a_1\cdots a_{2n-2})^{1/2}\frac1{\eta_1}\left(1+\frac{a_{2n-2}\cdots a_2}{\eta_2}+\frac{a_{2n-2}\cdots a_3 \eta_1}{\eta_2\eta_3}+\frac{a_{2n-2}\cdots a_4}{\eta_1\eta_2}{\eta_2\eta_3\eta_4}+\right.\\
&\left. + \frac{a_{2n-2}\cdots a_5}{\eta_1\eta_2\eta_3}{\eta_2\eta_3\eta_4\eta_5}+\cdots +\frac{a_{2n-2 \eta_1}}{\eta_{2n-3}\eta_{2n-2}}+\frac{\eta_1^2 \eta_{2n-3}}{a_1a_2\cdots a_{2n-2}\eta_{2n-3}\eta_{2n-2}}\right)\\
=&(a_1\cdots a_{2n-2})^{1/2}\left(\frac{\eta_2+a_{2n-2}\cdots a_2}{\eta_1\eta_2}+\frac{a_{2n-2}\cdots a_3}{\eta_2\eta_3}+\cdots +\frac{a_{2n-2}}{\eta_{2n-3}}+\frac{\eta_1}{a_1a_2\cdots a_{2n-2}}\right)\\
=&(a_1\cdots a_{2n-2})^{1/2}\left(\left[\frac{1}{\eta_2}+\frac{a_{2n-2}\cdots a_3}{\eta_2\eta_3}\right]+\cdots +\frac{a_{2n-2}}{\eta_{2n-3}}+\frac{\eta_1}{a_1a_2\cdots a_{2n-2}}\right)\\
=&(a_1\cdots a_{2n-2})^{1/2}\left(\left[\frac{1}{\eta_3}+\frac{a_{2n-2}\cdots a_4}{\eta_3\eta_4}\right]+\cdots +\frac{a_{2n-2}}{\eta_{2n-3}}+\frac{\eta_1}{a_1a_2\cdots a_{2n-2}}\right)\\
=&\cdots =(a_1\cdots a_{2n-2})^{1/2}\left(\left[\frac{1}{\eta_{2n-3}}+\frac{a_{2n-2}}{\eta_{2n-3}}\right]+\frac{\eta_1}{a_1a_2\cdots a_{2n-2}}\right)\\
=&(a_1\cdots a_{2n-2})^{1/2}\left(1+\frac{\eta_1}{a_1a_2\cdots a_{2n-2}}\right)\\
=&(a_1\cdots a_{2n-2})^{1/2}\left(1+\frac{1}{a_1} + \frac{1}{a_1a_2}+\cdots + \frac{1}{a_1a_2\cdots a_{2n-2}}\right)=G_{i,i+1}.
\end{align*}
Next, since the transformation (\ref{a-t}), (\ref{b-t}) is a Poisson isomorphism, it preserves both the product and Poisson algebras, so it suffices to consider transformations for generating elements $G_{l,l+1}$, $\widetilde G_{l.l+1}$ and $G_B$. 

An important observation is that, for any path $\cdots \to b_k\to a_k\to a_{k+1}\to b_{k-1}\to \cdots$ crossing the path $a_1\to a_2\to \cdots \to a_{2n-2}$, both the product $b_k a_k a_{k+1} b_{k-1}$ and the combination $\frac1{b_k}+\frac1{b_ka_k}+\frac1{b_ka_ka_{k+1}}$ are preserved. Whereas the invariance of $b_k a_k a_{k+1} b_{k-1}$ is straightforward and left as an easy exercise,  the invariance of $\frac1{b_k}+\frac1{b_ka_k}+\frac1{b_ka_ka_{k+1}}$ follows from the chain of equalities,
\begin{align*}
&\frac1{b'_k}+\frac1{b'_ka'_k}+\frac1{b'_ka'_ka'_{k+1}}=\frac{\eta_{k+2}}{\eta_k b_k}+\frac{\eta_{k+2}\eta_{k-1}}{b_ka_k\eta_k\eta_{k+1}}+\frac{\eta_{k-1}}{b_ka_ka_{k+1}\eta_{k+1}}\\
&=\frac{\eta_{k+2}}{\eta_k b_k}\frac{a_k\eta_{k+1}+a_k\cdots a_{2n-2}+\eta_k}{a_k\eta_{k+1}}+\frac{a_k\cdots a_{2n-2}+a_{k+1}\cdots a_{2n-2}+\eta_{k+1}}{b_ka_ka_{k+1}\eta_{k+1}}\\
&=\frac{a_{k+2}\cdots a_{2n-2}}{b_k\eta_{k+1}}+\frac{a_{k+2}\cdots a_{2n-2}+\eta_{k+2}}{b_ka_k\eta_{k+1}}+\frac{\eta_{k+2}(\eta_{k+1}+a_{k+1}\cdots a_{2n-2})}{b_k\eta_k\eta_{k+1}}+\frac1{b_ka_ka_{k+1}}\\
&=\frac{a_{k+2}\cdots a_{2n-2}}{b_k\eta_{k+1}}+\frac{1}{b_ka_k}+\frac{\eta_{k+2}}{b_k\eta_{k+1}}+\frac1{b_ka_ka_{k+1}}\\
&=\frac{a_{k+2}\cdots a_{2n-2}+\eta_{k+2}}{b_k\eta_{k+1}}+\frac{1}{b_ka_k}+\frac1{b_ka_ka_{k+1}}\\
&=\frac{1}{b_k}+\frac{1}{b_ka_k}+\frac1{b_ka_ka_{k+1}}
\end{align*}
This implies that $\widetilde G_{l,l+1}$ for all $l$ as well as $G_{l,l+1}$ with $|l-i|\ge 2$ are invariant under this sequence of mutations. It remains to check the transformation laws for 
\be\label{Gi-1i}
G_{i-1,i}=\langle b_2 x_3 b_4 x_5\cdots x_{2n-5}b_{2n-4}b_1a_1a_2\rangle
\ee
and for 
\be\label{Gi+1i+2}
G_{i+1,i+2}=\langle a_{2n-3} a_{2n-2}b_{2n-4}b_1 x_2 b_3 x_3\cdots x_{2n-6} b_{2n-5}\rangle,
\ee
where we let $x_k$ denote variables not transformed by the above sequence of mutations (for brevity we often omit subscripts of $x$-variables in calculations below). After some algebra we obtain the transformed $G_{i-1,i}$ and $G_{i+1,i+2}$:
\begin{align}
G'_{i-1,i}=&\frac{(b_2xb_4x\cdots xb_{2n-4}b_1a_1a_2)^{1/2}}{(a_1a_2\cdots a_{2n-2})^{1/2}}\left(\eta_2+\frac{\eta_4}{b_2}+\frac{\eta_4}{b_2 x}+\frac{\eta_6}{b_2 x b_4}+\frac{\eta_6}{b_2 x b_4 x}+\cdots 
\right.\nonumber\\
\label{Gprime}
&\qquad \qquad \left.
+\frac{1}{b_2 x b_4 x\cdots b_{2n-4}}\left[1+\frac1{b_1}\right] \right),
\end{align}
and
\begin{align*}
G'_{i+1,i+2}=&\frac{(a_{2n-3}a_{2n-2}b_{2n-4}b_1xb_3x\cdots xb_{2n-5})^{1/2}}{(a_1a_2\cdots a_{2n-2})^{1/2}}\left( (a_1a_2\cdots a_{2n-2})\left[1+\frac{1}{a_{2n-3}}\right]+\frac{\eta_{2n-4}\eta_1}{a_{2n-2}a_{2n-3}}\right.\\
&\left. +\frac{1}{a_{2n-2}a_{2n-3}b_{2n-4}}\left[\eta_1 +\frac{\eta_3}{b_1}+\frac{\eta_3}{b_1 x}+\frac{\eta_5}{b_1 x b_3}+\frac{\eta_5}{b_1 x b_3 x}+\cdots +\frac{\eta_{2n-3}}{b_1 x b_3 x\cdots x b_{2n-5}} \right]\right).
\end{align*}
We present only the transformation law for $G'_{i-1,i}$, that for $G'_{i+1,i+2}$ can be obtained analogously.

We begin with the Poisson relations:
\begin{align}
\bigl\{ b_k, \langle a_1\cdots a_{2n-2}\rangle \bigr\}=&-b_k(a_1\cdots a_{2n-2})^{1/2}\left(\frac1{a_1\cdots a_k}+\frac1{a_1\cdots a_ka_{k+1}} \right), \ k=1,\dots,2n-4,\\
\bigl\{ a_1, \langle a_1\cdots a_{2n-2}\rangle \bigr\}=&a_1(a_1\cdots a_{2n-2})^{1/2}\left(1+\frac1{a_1}-\frac1{a_1\cdots a_{2n-2}} \right),\\
\bigl\{ a_2, \langle a_1\cdots a_{2n-2}\rangle \bigr\}=&a_2(a_1\cdots a_{2n-2})^{1/2}\left(\frac1{a_1}+\frac1{a_1a_2} \right).
\end{align}
For the Poisson bracket between non-transformed matrix elements, we then have:
\begin{align*}
&\{ G_{i-1,i},G_{i,i+1} \}=G_{i-1,i}(a_1\cdots a_{2n-2})^{1/2}\left[-\frac12\left(\frac1{a_1}+\frac{2}{a_1a_2}+\frac{1}{a_1a_2a_3}+\cdots +\frac{1}{a_1a_2\cdots a_{2n-3}}  \right)\right.\\
&\qquad\qquad\qquad\qquad  \left. +\frac12\left(1+\frac2{a_1} +\frac{1}{a_1a_2}-\frac{1}{a_1\cdots a_{2n-2}} \right)  \right]\\
&\quad + (b_2xb_4x\cdots xb_{2n-4}b_1a_1a_2)^{1/2}(a_1a_2\cdots a_{2n-2})^{1/2}\left\{ \left(\frac1{b_2}+\frac1{b_2x} \right)\left(\frac{1}{a_1a_2}+\frac{1}{a_1a_2a_3} \right)+\right.\\
&\quad  + \left(\frac1{b_2xb_4}+\frac1{b_2xb_4x} \right)\left(\frac{1}{a_1a_2}+\frac{1}{a_1a_2a_3}+\frac{1}{a_1a_2a_3a_4}+\frac{1}{a_1a_2a_3a_4a_5} \right)\\
&\quad +\cdots +\frac{1}{b_2xb_4x\cdots x b_{2n-4}} \left(\frac{1}{a_1a_2}+\frac{1}{a_1a_2a_3}+\cdots+\frac{1}{a_1a_2\cdots a_{2n-3}} \right)\\
&\quad +\frac{1}{b_2xb_4x\cdots x b_{2n-4}b_1} \left(\frac{1}{a_1a_2}+\frac{1}{a_1a_2a_3}+\cdots+\frac{1}{a_1a_2\cdots a_{2n-3}}+\frac{1}{a_1}+\frac{1}{a_1a_2}  \right)\\
&\quad +\frac{1}{b_2xb_4x\cdots x b_{2n-4}b_1a_1} \left(-1+\frac{2}{a_1a_2}+\frac{1}{a_1a_2a_3}+\cdots+\frac{1}{a_1a_2\cdots a_{2n-3}}+\frac{1}{a_1a_2\cdots a_{2n-3}a_{2n-2}} \right)\\
&\quad\left. +\frac{1}{b_2xb_4x\cdots x b_{2n-4}b_1a_1a_2} \left(-1-\frac{1}{a_1}+\frac{1}{a_1a_2}+\frac{1}{a_1a_2a_3}+\cdots+\frac{1}{a_1a_2\cdots a_{2n-3}a_{2n-2}} \right)\right\}\\
&=-\frac12 G_{i-1,i}G_{i,i+1}+G_{i-1,i} (a_1a_2\cdots a_{2n-2})^{1/2}\left(1+\frac{1}{a_1}\right)\\
&\quad + \frac{(b_2xb_4x\cdots xb_{2n-4}b_1a_1a_2)^{1/2}}{(a_1a_2\cdots a_{2n-2})^{1/2}}\left\{ \left(\frac1{b_2}+\frac1{b_2x} \right)(\eta_2-\eta_4) ++ \left(\frac1{b_2xb_4}+\frac1{b_2xb_4x} \right)(\eta_2-\eta_6)+\right. \\
&\qquad +\cdots +\frac1{b_2xb_4x\cdots xb_{2n-4}}(\eta_2-\eta_{2n-2})+\frac1{b_2xb_4x\cdots xb_{2n-4}b_1}(\eta_2-\eta_3 +\eta_1 -1)\\
&\quad + \frac1{b_2xb_4x\cdots xb_{2n-4}b_1a_1}(-a_1\cdots a_{2n-2}+2a_3\cdots a_{2n-2}+a_4\cdots a_{2n-2}+\cdots +a_{2n-2}+1)\\
&\quad +\left. \frac1{b_2xb_4x\cdots xb_{2n-4}b_1a_1a_2}(-a_1\cdots a_{2n-2}-a_2\cdots a_{2n-2}+a_3\cdots a_{2n-2}+a_4\cdots a_{2n-2}+\cdots +a_{2n-2}+1)\right\}\\
&=[\hbox{cf. (\ref{Gprime})}]\quad -\frac12 G_{i-1,i}G_{i,i+1}-G'_{i-1,i} + \frac{(b_2xb_4x\cdots xb_{2n-4}b_1a_1a_2)^{1/2}}{(a_1a_2\cdots a_{2n-2})^{1/2}}\left\{ \eta_2+\frac1{b_2xb_4x\cdots xb_{2n-4}b_1}+\right. \\
&+\left(\frac1{b_2}+\frac1{b_2x}+\cdots +\frac1{b_2x\cdots xb_{2n-4}b_1} \right)\eta_2 \\
&+ \left(1+\frac1{b_2}+\frac1{b_2x}+\cdots +\frac1{b_2x\cdots xb_{2n-4}b_1}+\frac1{b_2x\cdots xb_{2n-4}b_1a_1}+\frac1{b_2x\cdots xb_{2n-4}b_1a_1a_2} \right)(a_1\cdots a_{2n-2}+a_2\cdots a_{2n-2}) \\
&+\frac1{b_2xb_4x\cdots xb_{2n-4}b_1} (\eta_1-\eta_3-1)\\
&+ \frac1{b_2xb_4x\cdots xb_{2n-4}b_1a_1}(-a_1\cdots a_{2n-2}+2a_3\cdots a_{2n-2}+a_4\cdots a_{2n-2}+\cdots +a_{2n-2}+1)\\
&+\left. \frac1{b_2xb_4x\cdots xb_{2n-4}b_1a_1a_2}(-a_1\cdots a_{2n-2}-a_2\cdots a_{2n-2}+a_3\cdots a_{2n-2}+a_4\cdots a_{2n-2}+\cdots +a_{2n-2}+1)\right\}\\
&=-\frac12 G_{i-1,i}G_{i,i+1}-G'_{i-1,i} +G_{i-1,i}G_{i,i+1} + \frac{(b_2xb_4x\cdots xb_{2n-4}b_1a_1a_2)^{1/2}}{(a_1a_2\cdots a_{2n-2})^{1/2}}\left\{ \eta_2 -\left(1+\frac{1}{b_2\cdots b_1a_1}+\frac{1}{b_2\cdots b_1a_1a_2} \right)\eta_2 \right.\\
& +\left. \frac{1}{b_2\cdots b_1}(\eta_1-\eta_3) + \frac1{b_2\cdots b_1a_1}(-a_1\cdots a_{2n-2}+a_3\cdots a_{2n-2}+\eta_2)+\frac1{b_2\cdots b_1a_1a_2}(-a_1\cdots a_{2n-2}-a_2\cdots a_{2n-2}+\eta_2)\right\}\\
&=\frac12 G_{i-1,i}G_{i,i+1}-G'_{i-1,i}  + \frac{(b_2xb_4x\cdots xb_{2n-4}b_1a_1a_2)^{1/2}}{(a_1a_2\cdots a_{2n-2})^{1/2}}\left\{ \frac{1}{b_2\cdots b_1}(a_2\cdots a_{2n-2}+a_3\cdots a_{2n-2})\right.\\
&\left. - \frac{a_2\cdots a_{2n-2}}{b_2\cdots b_1}+\frac{a_3\cdots a_{2n-2}}{b_2\cdots b_1a_1}- \frac{a_3\cdots a_{2n-2}}{b_2\cdots b_1} -\frac{a_3\cdots a_{2n-2}}{b_2\cdots b_1a_1}\right\}=\frac12 G_{i-1,i}G_{i,i+1}-G'_{i-1,i}.
\end{align*}
We therefore obtain that 
\be
G'_{i-1,i} = \frac12 G_{i-1,i}G_{i,i+1} -\{ G_{i-1,i},G_{i,i+1} \}= G_{i-1,i}G_{i,i+1} - G_{i-1,i+1},
\ee
which is the correct transformation law under the braid-group transformations. This completes the proof.

We do not verify the twist along $G_B$; instead, we use it in the next section to show the preservation of the Hamiltonian reduction condition in the case $n=5$. 

Therefore, all the braid-group transformations along $G_{i,i+1}$, $\tilde G_{i,i+1}$, and $G_B$  can be realized as sequences of mutations in the (extended) $SL_n$ quiver. 

\begin{remark}
Note that the above twists (and the corresponding braid-group action) do not generate the full modular group for $n>3$: all these twists preserve the $\mathbb Z_2$ symmetry (which naturally holds for $g=2$ surfaces since they admit a hyperelliptic uniformization). For higher genera we can only generate a $\mathbb Z_2$ quotient of the full modular group. 

For the corresponding geodesic functions, for those geodesics that are $\mathbb Z_2$-symmetric, their geodesic functions remain Laurent polynomials in $Z_i^{\pm 1/2}$; in order to close the product and Poisson algebras of these $\mathbb Z_2$-symmetric geodesic functions, we need to add symmetric combinations of non-symmetric geodesic functions (an example is the Markov element(s) $\mathfrak M_1$ and $\mathfrak M_2$, which are mirror-symmetric w.r.t. the $\mathbb Z_2$ transformation for $n=4$: their sum $\mathfrak M_1+ \mathfrak M_2$ and their product $\mathfrak M_1 \mathfrak M_2$ are $\mathbb Z_2$ symmetric functions and are Laurent polynomials. However $|\mathfrak M_1+ \mathfrak M_2|=62$ and $|\mathfrak M_1 \mathfrak M_2|=417$, so there is no chance that each of $\mathfrak M_i$ be polynomial. 
\end{remark}

\subsection{Hamiltonian reduction}
\subsubsection{$n=3$}
We begin with a ``toy'' example of the Hamiltonian reduction carried out by imposing the condition $\det(\mathbb A+\mathbb A^{\text{T}})=0$ in $n=3$ case. In terms of the Markov element $\mathfrak M$, for which we have that $\det(\mathbb A+\mathbb A^{\text{T}})=8+2\mathfrak M$, written in three different ways (\ref{M3-1}), (\ref{M3-2}), and (\ref{A3}), this condition takes the respective forms
\begin{align}
\label{G-1}
&G_{1,2}G_{1,3}G_{2,3}-G^2_{1,2}-G^2_{1,3}-G^2_{2,3}+4=0 \\
\label{G-2}
&\widetilde G_{1,2}\widetilde G_{1,3}\widetilde G_{2,3}-\widetilde G^2_{1,2}-\widetilde G^2_{1,3}-\widetilde G^2_{2,3}+4=0,\\
\label{G-3}
&f\cdot\bigl( G_{1,2}\widetilde G_{1,2}G_{B} + G^2_{1,2} + \widetilde G^2_{1,2} + G^2_{B}-4\bigr)=0.
\end{align}
The first two conditions are well-known from the physics literature on $SL(2,\mathbb R)$-connections in the description of $(2+1)$ gravity on the manifold $\mathbb T^2\times \mathbb R$ with $\mathbb T^2$ a two-dimensional torus with a flat metric \cite{Carlip-Nelson}: the solution that induces braid-group and Poisson relations is
\begin{align}
&G_{1,2}=e^{X/2}+e^{-X/2},\quad G_{2,3}=e^{Y/2}+e^{-Y/2},\quad G_{1,3}=e^{X/2+Y/2}+e^{-X/2-Y/2},\quad \{ X,Y\}=2\label{chiral}\\
&\widetilde G_{1,2}=e^{\widetilde X/2}+e^{-\widetilde X/2},\quad \widetilde G_{2,3}=e^{\widetilde Y/2}+e^{-\widetilde Y/2},\quad \widetilde G_{1,3}=e^{\widetilde X/2+\widetilde Y/2}+e^{-\widetilde X/2-\widetilde Y/2},\quad \{ \widetilde X,\widetilde Y\}=2,\label{anti-chiral}
\end{align}
where the total algebra of connections splits into mutually commuting chiral--anti-chiral subalgebras, $SL(2,\mathbb R)\times SL(2,\mathbb R)$, each being a copy of the standard $SL(2,\mathbb R)$ algebra. We are interested in the third way (\ref{G-3}) of representing the Hamiltonian reduction condition: a solution $f=0$ looks not feasible, so we resort to the second choice that 
\be\label{GB-1}
G_{1,2}\widetilde G_{1,2}G_{B} + G^2_{1,2} + \widetilde G^2_{1,2} + G^2_{B}-4=0.
\ee
This equation differs from (\ref{G-1}) or (\ref{G-2}) only by the sign of $G_B$, so we solve it immediately:
\be\label{GB-2}
G_B=-e^{X/2+\widetilde X/2}-e^{-X/2-\widetilde X/2}.
\ee
Since expression (\ref{GB-1}) Poisson commutes with $G_B$, it is preserved by twists along $G_B$, so the reduction is Hamiltonian.

The quiver corresponding to this reduction consists of two disjoint parts, $\raisebox{-5pt}{\mbox{\begin{pspicture}(-1.5,-0.3)(1.5,0.7){\psset{unit=1}
\rput(-0.8,0){
\psline[doubleline=true,linewidth=1pt, doublesep=1pt, linecolor=blue]{->}(-0.4,0)(0.4,0)
\multiput(-0.5,0)(1,0){2}{\pscircle[fillstyle=solid,fillcolor=lightgray]{.1}}
\put(-0.5,0.15){\makebox(0,0)[bc]{\hbox{{$X$}}}}
\put(0.5,0.15){\makebox(0,0)[bc]{\hbox{{$Y$}}}}
}
\rput(0.8,0){
\psline[doubleline=true,linewidth=1pt, doublesep=1pt, linecolor=blue]{->}(-0.4,0)(0.4,0)
\multiput(-0.5,0)(1,0){2}{\pscircle[fillstyle=solid,fillcolor=lightgray]{.1}}
\put(-0.5,0.15){\makebox(0,0)[bc]{\hbox{{$\widetilde X$}}}}
\put(0.5,0.15){\makebox(0,0)[bc]{\hbox{{$\widetilde Y$}}}}
}
}
\end{pspicture}
}}.$ We were so far not able to identify the element $G_B$ in the corresponding $(2+1)$-dimensional geometry.

\subsubsection{$n=4$---genus three} In this case computer simulation shows that the condition $\det(\mathbb{A}+\mathbb{A}^\text{T})=0$ is Hamiltonian provided the Casimir~(\ref{eq:CasN4}) $C=-1$. Note that we obtained the same condition as required by the rank condition of Section~\ref{ss:genus3}.

\subsubsection{$n=5$---genus four}
Starting with genus four ($n=5$), the total dimension $n(n-1)$ of the quiver becomes greater than the dimension $6g-6=6n-12$ of the corresponding moduli space, that is, we encounter the same problem as in the case of $\Sigma_{g,s}$ with $s=1,2$: how to segregate Poisson leaves of the symplectic groupoid that correspond to geometrical systems? It is well-known since Nelson, Regge, and Zertuche works \cite{NR}, \cite{NRZ}, \cite{NR93} that the condition is that $\mathop{rank}(\mathbb A+\mathbb A^{\text{T}})\le 4$. For $n=5$, this condition is satisfied if
\be\label{det}
\det(\mathbb A+\mathbb A^{\text{T}})=0.
\ee
We have the following lemma
\begin{lemma}\label{lm:Ham}
The condition (\ref{det}) in the case $n=5$ is preserved by all braid-group transformations, i.e., it is Hamiltonian.
\end{lemma}
\proof The expression $\det(\mathbb A+\mathbb A^{\text{T}})=\det(\widetilde{\mathbb  A}+\widetilde{\mathbb  A}^{\text{T}})$ is invariant under the braid-group transformations generated by $G_{i,i+1}$ and $\widetilde G_{i,i+1}$, so it remains to check preservation of the condition (\ref{det}) for the twist along $G_B$.  For this, we first transform the quiver for $n=5$ to the quiver in Fig.~\ref{fi:N5} in which $G_B$ has the simplest form. In this quiver,
\begin{align}
&G_B=\langle a_1a_t \rangle,\nonumber \\
&G_{1,2}=\langle b_4e_2d_3c_4 \rangle,\nonumber \\
&G_{2,3}=\langle c_4c_3b_3a_4e_3d_4 \rangle,\nonumber \\
&G_{3,4}=\langle d_4d_3d_2c_2b_2a_3e_4 \rangle,\nonumber \\
&G_{4,5}=(e_4e_3a_2a_3e_2a_4e_1d_1c_1b_1a_2a_1)^{1/2}\left[ 1+\frac{1}{e_4}+\frac{1}{e_4e_3}+\frac{1}{e_4e_3e_2}  \right.\nonumber \\
&\qquad  +\frac{1}{e_4a_2}+\frac{1}{e_4e_3a_2}+\frac{1}{e_4e_3e_2a_2}+\frac{1}{e_4e_3a_2a_3}+\frac{1}{e_4e_3e_2a_2a_3}+\frac{1}{e_4e_3e_2a_2a_3a_4\nonumber }\\
&\qquad +\frac{1}{e_4e_3e_2a_2a_3a_4e_1}+\frac{1}{e_4e_3e_2a_2a_3a_4e_1d_1}+\frac{1}{e_4e_3e_2a_2a_3a_4e_1d_1c_1}+\frac{1}{e_4e_3e_2a_2a_3a_4e_1d_1c_1b_1}\nonumber \\
&\qquad \left.+\frac{1}{e_4e_3e_2a_2a_3a_4e_1d_1c_1b_1a_2}+\frac{1}{e_4e_3e_2a_2a_3a_4e_1d_1c_1b_1a_2a_1}\right],\label{G-45}\\
&a_ta_1a_2^2a_3^2a_4^2b_1b_2b_3 b_4c_1c_2c_3 c_4d_1d_2d_3 d_4e_1e_2e_3 e_4=1.\nonumber 
\end{align}
Observe that $G_{1,2}$, $G_{2,3}$, and $G_{3,4}$ do not depend on $a_1$ and $a_2$ and therefore commute with $a_1$ and $a_t$.

\begin{figure}[H]
\begin{pspicture}(-5,-3)(4,3){\psset{unit=1.2}
\multiput(-2,1.7)(1,0){1}{\psline[linecolor=blue,linewidth=2pt]{<-}(0.1,0)(.9,0)}
\multiput(-1,1.7)(1,0){2}{\psline[linecolor=blue,linewidth=2pt]{->}(0.1,0)(.9,0)}
\multiput(-1.5,0.85)(0.5,-0.85){3}{\psline[linecolor=blue,linewidth=2pt]{->}(0.1,0)(.9,0)}
\multiput(1,0)(0.5,-0.85){3}{\psline[linecolor=blue,linewidth=2pt]{<-}(0.1,0)(.9,0)}
\multiput(0.5,-0.85)(0.5,-0.85){2}{\psline[linecolor=blue,linewidth=2pt]{<-}(0.1,0)(.9,0)}
\multiput(0,0)(-0.5,0.85){2}{\psline[linecolor=blue,linewidth=2pt]{->}(0.1,0)(.9,0)}
\multiput(1,1.7)(0.5,-0.85){4}{\psline[linecolor=blue,linewidth=2pt]{->}(0.05,-0.085)(.45,-0.765)}
\multiput(1,0)(0.5,-0.85){2}{\psline[linecolor=blue,linewidth=2pt]{->}(0.05,-0.085)(.45,-0.765)}
\multiput(0.5,-0.85)(0.5,-0.85){1}{\psline[linecolor=blue,linewidth=2pt]{->}(0.05,-0.085)(.45,-0.765)}
\multiput(-0.5,-.85)(-0.5,0.85){2}{\psline[linecolor=blue,linewidth=2pt]{->}(-0.05,0.085)(-.45,0.765)}
\multiput(0,0)(-0.5,0.85){1}{\psline[linecolor=blue,linewidth=2pt]{->}(-0.05,0.085)(-.45,0.765)}
\multiput(1,0)(-0.5,-0.85){3}{\psline[linecolor=blue,linewidth=2pt]{->}(-0.05,0.085)(-.45,0.765)}
\multiput(1,-1.7)(0.5,0.85){2}{\psline[linecolor=blue,linewidth=2pt]{->}(0.05,0.085)(.45,0.765)}
\multiput(2,-1.7)(0.5,0.85){1}{\psline[linecolor=blue,linewidth=2pt]{->}(0.05,0.085)(.45,0.765)}
\multiput(-1.5,0.85)(1,0){3}{\psline[linecolor=blue,linewidth=2pt]{->}(0.05,0.085)(.45,0.765)}
\multiput(-0.5,0.85)(-0.5,-0.85){1}{\psline[linecolor=blue,linewidth=2pt]{->}(-0.05,-0.085)(-.45,-0.765)}
\multiput(0.5,0.85)(-0.5,-0.85){2}{\psline[linecolor=blue,linewidth=2pt]{->}(-0.05,-0.085)(-.45,-0.765)}
\multiput(0,0)(1,0){2}{
\psline[linecolor=white,linewidth=3pt]{->}(-0.09,1.67)(-1.41,0.89)
\psline[linecolor=blue,linewidth=2pt]{->}(-0.09,1.67)(-1.41,0.89)
}
\psline[linecolor=white,linewidth=3pt]{->}(0,-1.7)(0.5,0.75)
\psline[linecolor=blue,linewidth=2pt]{->}(0,-1.7)(0.5,0.75)
\psline[linecolor=white,linewidth=3pt]{<-}(1,-1.7)(1.5,0.75)
\psline[linecolor=blue,linewidth=2pt]{<-}(1,-1.7)(1.5,0.75)
\rput(-1,0){\psline[linecolor=red,linewidth=2pt]{<-}(-0.09,1.67)(-1.41,0.89)}
\rput(-2,1.7){\psline[doubleline=true,linewidth=1pt, doublesep=1pt, linecolor=red]{->}(-0.05,-0.085)(-.45,-0.765)}
\rput(-2.5,0.85){\pscircle[fillstyle=solid,fillcolor=red]{.1}}
\psbezier[linecolor=blue,linewidth=2pt]{<-}(1.07,1.75)(3.5,4)(7,-6)(1.05,-1.78)
\psbezier[linecolor=white,linewidth=3pt]{<-}(-0.93,1.77)(1.5,3.85)(4,0.8)(3.05,-1.61)
\psbezier[linecolor=blue,linewidth=2pt]{<-}(-0.93,1.77)(1.5,3.85)(4,0.8)(3.05,-1.61)
\psbezier[linecolor=white,linewidth=3pt]{->}(0.93,1.75)(-4.5,5)(-3,-3)(-0.05,-1.78)
\psbezier[linecolor=blue,linewidth=2pt]{->}(0.93,1.75)(-4.5,5)(-3,-3)(-0.05,-1.78)
\psbezier[linecolor=white,linewidth=3pt]{->}(0.08,1.73)(2,3)(3,1)(3,-1.6)
\psbezier[linecolor=blue,linewidth=2pt]{->}(0.08,1.73)(2,3)(3,1)(3,-1.6)
\psbezier[linecolor=white,linewidth=3pt]{<-}(0.03,1.78)(3.5,5.5)(6,-5.5)(2.03,-1.78)
\psbezier[linecolor=blue,linewidth=2pt]{<-}(0.03,1.78)(3.5,5.5)(6,-5.5)(2.03,-1.78)
\psbezier[linecolor=white,linewidth=3pt]{->}(1.08,1.65)(4,0)(5,-3)(2.07,-1.74)
\psbezier[linecolor=blue,linewidth=2pt]{->}(1.08,1.65)(4,0)(5,-3)(2.07,-1.74)
\multiput(-3,1.7)(0.5,-0.85){5}{
\multiput(1,0)(1,0){4}{\pscircle[fillstyle=solid,fillcolor=lightgray]{.1}}}
\put(-2,1.9){\makebox(0,0)[bc]{\hbox{{$a_1$}}}}
\put(-1,1.9){\makebox(0,0)[bc]{\hbox{{$a_2$}}}}
\put(-0.1,1.9){\makebox(0,0)[bc]{\hbox{{$a_3$}}}}
\put(1.2,1.7){\makebox(0,0)[lc]{\hbox{{$a_4$}}}}
\put(-1.45,0.8){\makebox(0,0)[tr]{\hbox{{$b_1$}}}}
\put(-0.55,0.95){\makebox(0,0)[br]{\hbox{{$b_2$}}}}
\put(0.65,0.85){\makebox(0,0)[cl]{\hbox{{$b_3$}}}}
\put(1.65,0.85){\makebox(0,0)[cl]{\hbox{{$b_4$}}}}
\put(-1.15,0.05){\makebox(0,0)[br]{\hbox{{$c_1$}}}}
\put(-0.15,0.1){\makebox(0,0)[br]{\hbox{{$c_2$}}}}
\put(1.3,0.15){\makebox(0,0)[br]{\hbox{{$c_3$}}}}
\put(2.3,0.15){\makebox(0,0)[br]{\hbox{{$c_4$}}}}
\put(-0.65,-0.8){\makebox(0,0)[br]{\hbox{{$d_1$}}}}
\put(0.6,-0.7){\makebox(0,0)[bl]{\hbox{{$d_2$}}}}
\put(1.68,-0.75){\makebox(0,0)[bl]{\hbox{{$d_3$}}}}
\put(2.5,-0.75){\makebox(0,0)[bl]{\hbox{{$d_4$}}}}
\put(-0.15,-1.65){\makebox(0,0)[br]{\hbox{{$e_1$}}}}
\put(0.9,-1.7){\makebox(0,0)[cr]{\hbox{{$e_2$}}}}
\put(1.93,-1.77){\makebox(0,0)[tr]{\hbox{{$e_3$}}}}
\put(3.15,-1.7){\makebox(0,0)[cl]{\hbox{{$e_4$}}}}
\put(-2.5,0.7){\makebox(0,0)[tc]{\hbox{{$a_t$}}}}
%
}
\end{pspicture}
\caption{\small
In this quiver, $G_B=\langle a_1a_t\rangle$ depends only on $a_1$ and $a_t$.
}
\label{fi:N5}
\end{figure}
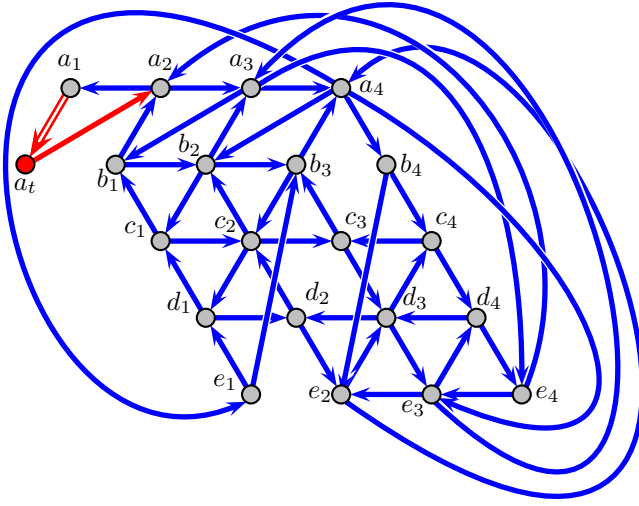

From the expression (\ref{G-45}) we have that $G_{4,5}$ has the structure
\be\label{G45-s}
G_{4,5}=s_1 a_2(a_1)^{1/2}+s_2 (a_1)^{1/2} + s_3 a_2^{-1}\bigl((a_1)^{1/2}+(a_1)^{-1/2}\bigr), 
\ee
where we let $s_i$ denote combinations commuting with $a_1$ and $a_t$ (and therefore not depending on $a_1$ and $a_2$). For the bracket $\{G_B,G_{4,5}\}$ we then have
\begin{align*}
\{G_B,G_{4,5}\}=&-\frac12 \bigl( (a_1a_t)^{1/2}-(a_1a_t)^{-1/'2} \bigr)\bigl(s_1 a_2(a_1)^{1/2}+s_2 (a_1)^{1/2} + s_3 a_2^{-1}(a_1^{1/2}-a_1^{-1/2}) \bigr)\\
&\quad+\left(\frac{a_t}{a_1}\right)^{1/2}\left( \frac12 s_1a_2(a_1)^{1/2}-\frac12 s_2(a_1)^{1/2}-\frac32 s_3a_2^{-1}a_1^{1/2}-\frac12 s_3(a_1)^{-1/2} \right),
\end{align*}
and we have that 
\begin{align}
&\frac12 G_BG_{4,5}+\{ G_B, G_{4,5}\}=(a_1a_t)^{-1/'2} \bigl(s_1 a_2(a_1)^{1/2}+s_2 (a_1)^{1/2} + s_3 a_2^{-1}a_1^{1/2}\bigr)+ (a_1a_t)^{1/'2} s_3 a_2^{-1}a_1^{-1/2} \nonumber\\
&\qquad\qquad\qquad\qquad\qquad\qquad+\left(\frac{a_t}{a_1}\right)^{1/2}\bigl( s_1a_2(a_1)^{1/2}-s_3a_2^{-1}a_1^{1/2}\bigr)\nonumber\\
&\qquad\qquad =(a_t)^{-1/2}a_2(1+a_t) s_1+(a_t)^{-1/2} s_2+(a_t)^{-1/2}a_2^{-1}s_3=G_{4,5}\bigl(a_1\to 1/a_t; \, a_2\to a_2(1+a_t)\bigr)
\end{align}
that is, the twist along $G_B$ is given by the mutation at $a_t$ with subsequent interchanging $a_1\leftrightarrow 1/a_t$ in the expression for $G_{4,5}$. Since all other $G_{i,5}$ are generated by the product and Poisson relations of $G_{4,5}$ and other $G_{i,i+1}$, all these relations are linear in $G_{4,5}$, and because $a_t$ commutes with all other $G_{i,i+1}$, we conclude that the action of the twist along $G_B$ on $\det(\mathbb A+\mathbb A^{\text{T}})$ is given by exactly the same substitution for variables $a_1$ and $a_2$. Next, since $G_{4,5}$ has the structure (\ref{G45-s}) and since $\det(\mathbb A+\mathbb A^{\text{T}})$ comprises terms that either do not depend on $G_{i,5}$ or are of order two in $G_{i,5}$, we obtain that
$$
\det(\mathbb A+\mathbb A^{\text{T}})=\alpha \xi a_1a_2^2+\beta a_1+\gamma \xi^{1/2} a_2a_1 +\delta\xi^{1/2} \frac{1+a_1}{a_2}+\rho\xi \frac{(1+a_1)^2}{a_1a_2^2}+\omega,
$$
where we let
$$
a_t=1/(a_1a_2^2\xi)\ \hbox{with}\ \xi=a_3^2a_4^2b_1b_2b_3 b_4c_1c_2c_3 c_4d_1d_2d_3 d_4e_1e_2e_3 e_4.
$$
Denoting
\be\label{subst}
a_1'=1/a_t=\xi a_1a_2^2\ \hbox{and}\  a_2'=a_2(1+a_t)=a_2\left(1+\frac{1}{\xi a_1a_2^2}\right),
\ee
we have that
$$
\det(\mathbb A+\mathbb A^{\text{T}})(a_1',a_2')-\xi a_2^2 \det(\mathbb A+\mathbb A^{\text{T}})(a_1,a_2)=(2\alpha-\xi\omega)a_2^2+\left(\frac{\alpha}{\xi}-\xi\rho \right)\frac 1{a_1} +(\gamma -\xi\delta)a_2+(\omega-2\xi\rho).
$$
It is now easy to see that the condition $\det(\mathbb A+\mathbb A^{\text{T}})=0$ is preserved upon substitution (\ref{subst}) if and only if 
$$
\alpha = \xi^2\rho,\quad \omega=2\xi\rho, \ \hbox{and}\ \gamma=\xi\delta,
$$
which can be checked by computer simulations. Upon satisfaction of these equalities, we can rewrite $\det(\mathbb A+\mathbb A^{\text{T}})$ in the form
\begin{align}
\det(\mathbb A+\mathbb A^{\text{T}})&=\alpha\left(\xi a_1a_2^2+2+2a_1+ \frac{(1+a_1)^2}{\xi a_1a_2^2}\right)+(\beta-2) a_1+\gamma\left(\xi^{1/2}a_2a_1+\frac{1+a_1}{\xi^{1/2}a_2} \right)\nonumber\\
&=a_1\bigl(\alpha G_B^2 +\gamma G_B+\beta-2 \bigr),\label{A5}
\end{align}
where $\alpha$, $\beta$, and $\gamma$ depend only on $a_3,\dots, e_4$. 
\begin{remark}
Note the striking similarity of expressions (\ref{A5}) in the case $n=5$ and (\ref{A3}) for $n=3$: in both expressions there is a single cluster variable ($f$ for $n=3$ and $a_1$ for $n=5$) such that $\det(\mathbb A+\mathbb A^{\text{T}})$ is proportional to this variable (note that $\det(\mathbb A+\mathbb A^{\text{T}})=2(\mathfrak M+4)$ for $n=3$); the proportionality coefficient is a second-degree polynomial in $G_B$ and is invariant under the twist along $G_B$; it is tempting to learn whether we can identify the coefficient functions $\alpha$, $\beta$, and $\gamma$ in (\ref{A5}) with some expressions polynomial in $G_{i,j}$ and $\widetilde G_{i,j}$ with $1\le i<j\le 4$ and whether similar statements are valid for higher odd $n$. We checked that it is not the case for $n=4$.
\end{remark}

\section{Conclusion}

In the conslusion we would like to indicate few questions that can be approached using the same technique.  
\begin{itemize}
\item quantization: the construction described in this paper allows straightforward quantization which will be described in a future publication;
\item symplectic structure: since (non-extended) $SL_n$ quivers are full-dimensional, their Poisson relations can be inverted producing the corresponding symplectic structures. Say, for $n=3$, the symplectic structure is
$$
\frac{ {\rm d} d}{d}\wedge\left(\frac{ {\rm d} e}{e}+\frac{ {\rm d} a}{a} \right) 
+\frac{ {\rm d} f}{f}\wedge\left(\frac{ {\rm d} c}{c}+\frac{ {\rm d} a}{a} \right) 
+\frac{ {\rm d} b}{b}\wedge\left(\frac{ {\rm d} c}{c}+\frac{ {\rm d} e}{e} \right), 
$$
and it must coincide with the standard Fenchel--Nielsen symplectic structure
$$
{\rm d}\ell_{1.2}\wedge {\rm d}\tau_{1,2}+{\rm d}\widetilde\ell_{1.2}\wedge {\rm d}\widetilde \tau_{1,2}+{\rm d}\ell_{\mathfrak M}\wedge {\rm d}\tau_{B},
$$
with the twists $\tau_{1,2}$, $\widetilde \tau_{1,2}$ defined in (\ref{twist-lezginka}). We expect similar relations to hold for higher $n$ as well;
\item calculating volumes $V_{g,0}$ of moduli spaces using Bowditch-like technique \cite{Bowditch} based on twist transformations; these volumes are known from the Mirzakhani's-type recursion relations \cite{DoNorbury} and from the topological recursion relations \cite{EO07}. For example, $V_{2,0}=\dfrac{43\pi^6}{2160}$.
\item explore the connection with link/knot invariants (see e.g.,~\cite{B-MSch}).
\end{itemize}

\section*{Acknowledgments}
Special thanks are due to Gus Schrader and Alexander Shapiro for the valuable discussion. The research of M.S. was partially supported by the NSF research grant DMS-1702115. While working on this project, M.S. was benefited from support of the following institutions and programs: Haifa University (Summer, 2022) Mathematical Institute of the University of Heidelberg (Summer 2022). The second author is grateful to all these institutions for their hospitality and outstanding working conditions they provided. The research of L.Ch. was partially supported by the Steklov International Mathematical Center (agreement No. 075-15-2019-1614).

\end{document}